\documentclass[12pt]{amsart}

\usepackage{url}

\usepackage{amssymb}
\usepackage{amsmath}
\usepackage{amsthm}
\newcounter{nushka}

\newtheorem*{uthm}{Theorem}

\newtheorem*{ulem}{Lemma}

\newtheorem{lem}[nushka]{Lemma}

\newtheorem{prop}[nushka]{Proposition}

\theoremstyle{definition}

\newtheorem*{udef}{Definition}

\newtheorem*{claim}{Claim}
\newtheorem*{alternative}{Alternative}

\newcommand{\goal}{\par\vspace{10pt}\noindent\begingroup\it }
\newcommand{\goalend}{\endgroup\par\vspace{10pt}}

\newcommand{\supp}{\operatorname{supp}}
\newcommand{\Lip}{\operatorname{Lip}}

\newcommand{\ci}[1]{_{{}_{\scriptstyle #1}}}
\newcommand{\cci}[1]{_{{}_{\scriptscriptstyle #1}}}

\newcommand{\Calderon}{Calder\'on}

\newcommand{\R}{\mathbb R}
\newcommand{\e}{\varepsilon}
\newcommand{\f}{\varphi}
\newcommand{\la}{\lambda}

\newcommand{\D}{\mathcal D}
\newcommand{\F}{\mathcal F}
\newcommand{\LL}{\mathcal L}
\newcommand{\HH}{\mathcal H}

\newcommand{\PP}{\mathfrak P}
\newcommand{\QQ}{\mathfrak Q}

\newcommand{\wt}{\widetilde}

\newcommand{\RH}{R^{H}}
\newcommand{\wtRH}{\wt R^{H}}

\renewcommand{\ge}{\geqslant}
\renewcommand{\le}{\leqslant}

\newcommand{\dist}{\operatorname{dist}}

\begin{document}

\title[Uniform rectifiability and the Riesz transform]{On the uniform rectifiability of AD regular measures with bounded Riesz transform operator: the case of codimension $1$}

\author{Fedor Nazarov}
\address{Fedor Nazarov, Department of Mathematical Sciences, Kent State University, Kent, Ohio, USA}

\author{Xavier Tolsa}
\address{Xavier Tolsa, ICREA/Universitat Aut\`onoma de Barcelona, Barcelona, Catalonia}

\author{Alexander Volberg}
\address{Alexander Volberg, Department of Mathematics,  Michigan State University, East Lansing, Michigan, USA}

\date{\today}

\begin{abstract}
We prove that if $\mu$ is a $d$-dimensional Ahlfors-David regular measure in $\R^{d+1}$, then the
boundedness of the $d$-dimensional Riesz transform in $L^2(\mu)$ implies that the non-BAUP David-Semmes cells form a Carleson family. Combined with earlier results of David and Semmes,
this yields the uniform rectifiability of $\mu$.
\end{abstract}

\maketitle

\section{Introduction}

The brilliant $350$-page monograph \cite{DS} by David and Semmes, which, like many other
research monographs, has been
  cited by many and read by few\footnote{Namely by four people: Guy David, Steven Semmes, Peter Jones, and Someone Else, as the saying goes.} is, in a
sense, devoted to a single question: {\em How to relate the boundedness of certain singular
integral operators in $L^2(\mu)$ to the geometric properties of the support of $\mu$?}. At the
moment of its writing, even the case of the Cauchy integral on the complex plane had not been understood.
This changed with the appearance of the pioneering work by Mattila, Melnikov, and Verdera
\cite{MMV}, which led to many far-reaching developments culminating in the full proof of Vitushkin's
conjecture by David \cite{D1} in 1998. Since then, there was a strong temptation to generalize the corresponding
results to kernels of higher dimensions. However, the curvature methods introduced by Melnikov, which
were an indispensable part of every approach known until very recently, fail miserably
above the dimension $1$. The development of curvature free techniques is still an urgent
necessity.

For dimensions greater than $1$, connecting the geometry of the support of $\mu$ with the boundedness
of some singular integral operators in $L^2(\mu)$ is not easy in either direction.
Passing from the geometric properties of the measure to the bounds for
 the operator norms is somewhat simpler. It had been known to David and Semmes already that the uniform rectifiability of an
Ahlfors-David regular (AD regular, for short) $d$-dimensional measure $\mu$ in $\R^n$ suffices for
the boundedness in $L^2(\mu)$ of many reasonable $d$-dimensional \Calderon--Zygmund operators
(more precisely, the ones with smooth antisymmetric convolution type kernels).

It is the other direction that remains a challenging task. We do not know what \cite{DS} looked
like to its authors when they were writing it, but an unexperienced reader would,
most likely, perceive it as a desperate attempt to build a bridge in this direction
starting with the destination point. Formally, the book presents a variety of conditions
equivalent to the uniform rectifiability. Apparently, the hope was that one of those
conditions could be checked using the boundedness of the $d$-dimensional Riesz
transform in $\R^n$, which is the natural analogue of the Cauchy operator in the
high-dimensional setting. David and Semmes did not manage to show that much. Nevertheless,
they proved that the uniform rectifiability of $\mu$ is implied by the simultaneous
boundedness in $L^2(\mu)$ of a sufficiently big class of $d$-dimensional convolution type
\Calderon--Zygmund operators with odd kernels.

The aim of the present paper is to fulfill that hope in the case $n=d+1$ and to supply the
missing part of the bridge, the part that leads from the boundedness of the Riesz transform
in $L^2(\mu)$ to one of the equivalent criteria for uniform rectifiability in \cite{DS}.
Ironically, the condition that we use as a meeting point is an auxiliary condition
that is only briefly mentioned in the David-Semmes book. The result we prove in this paper reads
as follows.

\begin{uthm}
Let $\mu$ be an AD regular measure of dimension $d$ in $\R^{d+1}$. If the
associated $d$-dimensional
Riesz transform operator
$$
f\mapsto K*(f\mu),\quad\text{where }K(x)=\frac{x}{|x|^{d+1}}\,,
$$
is bounded in $L^2(\mu)$, then the non-BAUP cells in the David-Semmes lattice associated
with $\mu$ form a Carleson family.
\end{uthm}

Proposition 3.18 of \cite{DS} (page 141)
asserts that this condition ``implies the WHIP and the WTP'' and hence, by Theorem 3.9
(page 137), the uniform rectifiability of $\mu$. Note that \cite{DS} talks about AD regular
sets rather than AD regular measures, so the notation there is different, and what they denote
by $E$ is the support of $\mu$ in our setting. We want to emphasize here that the current paper
treats only the ``analytic'' part of the passage from the operator boundedness to
the rectifiability.
The full credit (as well as the full responsibility) for the other ``geometric'' part
should go to David and Semmes.

There are two key ingredients of our proof that may be relatively novel.

The first one is the Flattening Lemma (Proposition \ref{flatteninglemma}, Section \ref{flatteninglemmaS}), which ultimately leads to the
conclusion that it is impossible to have many cells on which the support of the measure
is close to a $d$-plane but the measure itself is distributed in a noticeably different
way from the Lebesgue measure on that plane. The exact formulation of the Flattening Lemma
we use here is tailored to our particular approach but it takes its origin in the earlier
works by Tolsa \cite{T1} and \cite{T2} on the relations between $\alpha$-numbers and measure transportation costs and the boundedness of the Riesz transform.

The second crucial ingredient is the Eiderman-Nazarov-Volberg scheme from \cite{ENV}, which was
later exploited by Jaye in \cite{JNV} to show that for the case of a non-integer
$s\in(d,d+1)$, the boundedness in $L^2(\mu)$ of the $s$-dimensional Riesz transform associated
with an $s$-dimensional measure $\mu$ in $\R^{d+1}$  implies the finiteness
of some Wolff  type potential with an exponential gauge function. This scheme allowed one to
fully develop the idea of Mateu and Tolsa in \cite{MT}
and to turn the scales of low density, which were the main enemy in most previous approaches, into
a useful friend.

Roughly speaking, the present paper uses the non-BAUP cells instead of the
scales of high density and the flat cells instead of the scales of low density to introduce
a Cantor type structure, which is then treated similarly to how it was done in \cite{ENV}.
The most essential deviations and additions are using the holes in the non-BAUP cells to hide
the negative part of $R^*(\psi\, m)$, the alignment of the approximating planes in the stopping
flat cells, the quasiorthogonality estimates based on flatness instead of smallness of the density,
and the consideration of only the $d$-dimensional part of the Riesz kernel aligned with approximating
planes.

The main limitation of our approach, which doesn't allow us to extend our result to codimensions
greater than $1$, comes from the reliance of the \cite{ENV} scheme on a certain maximum principle,
of which no analogue is known in codimensions higher than $1$. Extending or bypassing this maximum
principle could possibly lead to the full solution of the problem.

It is worth mentioning here that shortly before our paper was finished, Hofmann, Martell,
and Mayboroda posted a paper \cite{HMM} on arXiv that contains a result equivalent to ours under
the additional assumption that $\mu$ is the surface measure on the boundary of a not
too weird connected domain in $\R^{d+1}$. They also expressed the hope that their techniques may
eventually provide an alternative approach to the full rectifiability conjecture. Unfortunately,
their proof is also heavily based on the harmonicity of the kernel, which seems to make it hard
to extend their techniques to the case of higher codimensions.

Including all the relevant definitions into this introduction would take too much space,
so if the reader has got interested enough at this point to continue reading the paper,
he will find them all in the main body of the article (and if not, all we can do is
to bid him farewell now).

\section{Acknowledgements}

The present work would not be possible without numerous previous attempts of many
mathematicians. We thank them all for sharing their ideas and techniques with us.
The reader can find the (possibly incomplete) list of their names in the notes
\cite{D2} by David and references therein. To engage here into a detailed description
of who did exactly what and when would be tantamount to writing a book on the history
of a subject of which we have neither sufficient knowledge, nor an unbiased judgement.

Our special thanks go to Ben Jaye for helping us to verify and to straighten
some delicate technical details in the proof and to Vladimir Eiderman for his
unwavering support and belief in this project.

\section{The structure of the paper}

We tried to make the paper essentially self-contained. The only thing that the reader is
assumed to be familiar with is the elementary theory of \Calderon--Zygmund operators
in homogeneous spaces. Everything else, including such standard for experts
things as the David-Semmes lattice and weak limit considerations, is developed
almost from scratch. The paper is split into reasonably short sections each of which
is devoted to one step, one construction, or one estimate in the proof. We tried to
explain the goal of each section at its beginning and to give each section some meaningful
title. We hope that this will help the reader to easily separate topics he already
knows well from those that might be new to him. We also believed that it would make sense to
include extra details or routine computations even at the cost of making the paper longer if
they may spare the reader some time and headache when checking the argument. However, despite
all our efforts, the text is still fairly dense and the full logic of the proof will reveal
itself only at the end of the last section.

\section{The notation}

By $c$ and $C$
we denote various positive constants. We usually
think of $c$ as of a small constant used in a bound of some quantity from below and of $C$
as of a large constant used in a bound from above. The constants appearing
in intermediate computations may change from one occurence to another.
Some constants
may depend on parameters, in which case those parameters are always mentioned explicitly
and often included in parentheses after
the constant unless such dependence is absolutely clear from the context like in the case
of the dependence on the dimension $d$: all constants we use do depend on $d$ but, since $d$
is fixed throughout the entire paper, we hardly ever mention this.

Due to the fact that the Riesz transform operator maps scalar valued measures (or functions)
to vector valued functions, scalar and vector valued quantities will be heavily mixed
in many formulae. We leave it to the reader to figure out in every particular case when
the product is a product of two scalars and when it is a product of a scalar and a vector
in $\R^{d+1}$. However, whenever the scalar product of two vector-valued quantities is
meant, we always use angular brackets $\langle\cdot,\cdot\rangle$. Whenever the
angular brackets are also used
for the scalar product or duality coupling in some function spaces,
we indicate that by
writing something like $\langle\cdot,\cdot\rangle\ci{L^2(\mu)}$ or merely
$\langle\cdot,\cdot\rangle_\mu$ .

We will always denote by $B(x,r)$ an open ball of radius $r$ centered at $x\in\R^{d+1}$
and by $\bar B(x,r)$ the corresponding closed ball. The notation $\chi\ci E$ 
will always be used for the characteristic function of a set $E\subset\R^{d+1}$. 

By the support $\supp\mu$ of a measure
$\mu$ we always mean the closed support. The same notation and
the same convention apply to supports of functions. We always specify the measure $\mu$
in the notation when talking about $L^p(\mu)$
norms in the usual sense. However, we also use the notation
$\|f\|\ci{L^\infty(E)}$ for the supremum of $|f|$ over the set $E$. If we omit $E$ and just
write $\|f\|\ci{L^\infty}$, it means that the supremum is taken over the whole
space $\R^{d+1}$. The same convention applies to integrals: if the domain of integration
is not specified, the integral over the whole space is meant.
The Lipschitz norm of a function $f$ on a set $E\subset\R^{d+1}$ is defined
as
$$
\|f\|\ci{\Lip(E)}=\sup_{x,y\in E, x\ne y}\frac{|f(x)-f(y)|}{|x-y|}\,.
$$
If $E$ is omitted in this notation, we mean the Lipschitz norm in the full space $\R^{d+1}$.
We use the letter $m$ to denote the $d+1$-dimensional Lebesgue measure on $\R^{d+1}$.
The $d$-dimensional Lebesgue measure on an affine hyperplane $L\subset \R^{d+1}$
is denoted $m\ci L$.

We use the notation $\dist(x,E)$  for the distance between a point $x\in\R^{d+1}$
and a set $E\subset\R^{d+1}$. Similarly, we write $\dist(E,F)$ for the distance
between two sets $E,F\subset\R^{d+1}$.

 \section{The $d$-dimensional Riesz transform in $\R^{d+1}$}
 \label{riesztransform}

\goal The goal of this section is to remind the reader (or to acquaint him with) the general
notions of the theory of AD regular measures and the associated Riesz transform operators.
\goalend 

Fix a positive integer $d$. Define the $d$-dimensional (vector valued) Riesz kernel in
$\R^{d+1}$ by $K(x)=\frac{x}{|x|^{d+1}}$. For a finite signed Borel measure $\nu$ in $\R^{d+1}$,
define its Riesz transform by
$$
R\nu=K*\nu=\int K(x-y)\,d\nu(y)\,.
$$
The singularity of $K$ at the origin is mild enough to ensure that the integral always converges
absolutely almost everywhere with respect to the $(d+1)$-dimensional Lebesgue measure $m$
in $\R^{d+1}$ and everywhere if $\nu$ is sufficiently smooth (say, has a bounded density
with respect to $m$). Moreover, the Riesz transform $R\nu$ is
infinitely differentiable in $\R^{d+1}\setminus\supp\nu$ and, since
$$
|(\nabla^k K)(x)|\le \frac{C(k)}{|x|^{d+k}}
$$
for all $x\ne 0$ and each $k\ge 0$, we have
\begin{equation}
\label{smoothest}
|(\nabla^k R\nu)(x)|\le C(k)\int\frac{d|\nu|(y)}{|x-y|^{d+k}}
\end{equation}
for each $x\notin\supp\nu$,
where $|\nu|$ stands for the variation of $\nu$.

Note also that the finiteness of the measure
is not so important in these estimates, so the Riesz transform $R\nu$ can
also be defined for any measure $\nu$ satisfying $\int\frac{d|\nu|(x)}{1+|x|^d}<+\infty$.

Similarly, using the estimate
$$
|K(x')-K(x'')|\le C\frac{|x'-x''|}{\min(|x'|,|x''|)^{d+1}}\,,
$$
we also obtain
$$
|(R\nu)(x')-(R\nu)(x'')|\le C \int\frac{|x'-x''|\,d|\nu|(y)}{\min(|x'-y|,|x''-y|)^{d+1}}\,.
$$
An immediate consequence of this bound is that if $\nu$ satisfies the growth restriction
$|\nu(B(x,r))|\le Cr^d$ for all $x\in\R^{d+1}$, $r>0$, and if $E$ is any subset of $\R^{d+1}$
separated from $\supp\nu$, then
\begin{equation}
\label{lipriesz}
\|R\nu\|\ci{\Lip(E)}\le \frac{C}{\operatorname{dist}(E,\supp\nu)}\,.
\end{equation}
Note that this estimate {\em does not} follow from \eqref{smoothest} immediately
because it may
be impossible to connect $x',x''\in E$ by a path of length comparable to $|x'-x''|$
that stays far away from $\supp\nu$.

In general, the singularity of the kernel at the origin is too strong to
allow one to talk of the values of $R\nu$ on $\supp\nu$. The usual way to overcome
this difficulty is to introduce regularized kernels $K_\delta$ ($\delta>0$). The exact
choice of the regularization is not too important as long as the antisymmetry and
the \Calderon--Zygmund properties of the kernel are preserved. For the purposes of this
paper, the definition
$$
K_\delta(x)=\frac{x}{\max(\delta,|x|)^{d+1}}
$$
is the most convenient one, so we will use it everywhere below.
The corresponding regularized Riesz transforms
$$
R_\delta\nu=K_\delta*\nu=\int K_\delta(x-y)\,d\nu(y)
$$
are well-defined and locally Lipschitz in the entire space $\R^{d+1}$ for any signed measure $\nu$
satisfying $\int\frac {d|\nu|(x)}{1+|x|^d}<+\infty$. In particular, if we have a positive
measure $\mu$ satisfying $\mu(B(x,r))\le Cr^d$ for every $x\in\R^{d+1}, r>0$ with some fixed $C>0$,
and a function $f\in L^p(\mu)$, $1<p<+\infty$,
then $R_\delta(f\mu)$ is well-defined pointwise for all $\delta>0$, so it
makes sense to ask whether the corresponding operators $R_{\mu,\delta}f=R_\delta(f\mu)$ are
bounded in $L^p(\mu)$.

The standard theory of \Calderon--Zygmund operators\footnote{Though the measure $\mu$
is not assumed to be doubling at this point, we will apply this theory only
when $\mu$ is an AD regular measure, so we do not really need here the subtler version
of the theory dealing with non-homogeneous spaces.} implies that the answer does not
depend on $p\in(1,+\infty)$. Moreover, if we know the uniform growth
bound $\mu(B(x,r))\le Cr^d$ and an estimate for the norm
$\|R_{\mu,\delta}\|\ci{L^{p_0}(\mu)\to L^{p_0}(\mu)}$ for some $p_0\in(1,+\infty)$,
we can explicitly bound the norms $\|R_{\mu,\delta}\|\ci{L^{p}(\mu)\to L^{p}(\mu)}$
for all other $p$.

These observations lead to the following
\begin{udef}
A positive Borel measure $\mu$ in $\R^{d+1}$ is called $C$-nice if
$\mu(B(x,r))\le Cr^d$ for every $x\in\R^{d+1}, r>0$. It is called $C$-good
if it is $C$-nice and
$\|R_{\mu,\delta}\|\ci{L^{2}(\mu)\to L^{2}(\mu)}\le C$ for every $\delta>0$.
\end{udef}

Often we will just say ``nice'' and ``good'' without specifying $C$, meaning that
the corresponding constants are fixed throughout the argument. A few notes are
in order.

First, for non-atomic measures $\mu$, the uniform norm bounds \linebreak
$\|R_{\mu,\delta}\|\ci{L^{2}(\mu)\to L^{2}(\mu)}\le C$ imply that $\mu$ is $C'$-nice
with some $C'$ depending on $C$ only.

Second, it follows from the above remarks that despite ``goodness'' being defined in terms
of the $L^2$-norms, we will get an equivalent definition using any other $L^p$-norm with
$1<p<+\infty$. What will be important for us below is that for any $C$-good measure $\mu$,
the operator norms  $\|R_{\mu,\delta}\|\ci{L^{4}(\mu)\to L^{4}(\mu)}$ are also bounded by
some constant $C'$.

We now can state formally what the phrase ``the associated Riesz transform is bounded in $L^2(\mu)$''
in the statement of the theorem means. We will interpret it as ``the measure $\mu$ is good''.
By the classical theory of \Calderon--Zygmund operators, this is equivalent to all other reasonable
formulations, the weakest looking of which is, probably, the existence of a
bounded operator $T:L^2(\mu)\to L^2(\mu)$ such that $(Tf)(x)=\int K(x-y)f(y)\,d\mu(y)$ for
$\mu$-almost all $x\notin\supp f$.

A few words should be said about duality and the adjoint operator $R^*$. The formal change of 
order of integration combined with the antisymmetry of $K$ yields the identity
\begin{multline*}
\int\langle R\nu,d\eta\rangle=
\int\left\langle\int K(x-y)\,d\nu(y),d\eta(x)\right\rangle
\\
=-\int\left(\int\langle K(x-y),d\eta(y)\rangle\right)\,d\nu(x)
=-\int\left(\sum_{j=1}^{d+1}\left\langle e_j, R\langle \eta,e_j\rangle\right\rangle\right)\,d\mu
\end{multline*}
leading to the formula
\begin{equation}
\label{RtoRstar}
R^*\eta=-\sum_{j=1}^{d+1}\left\langle e_j, R\langle \eta,e_j\rangle\right\rangle\,,
\end{equation}
where $\nu$ is a scalar (signed) measure, $\eta$ is a vector-valued measure, and $e_1,\dots,e_{d+1}$
is an arbitrary orthonormal basis in $\R^{d+1}$.

This computation is easy to justify if both $\nu$ and $\eta$ are finite and at least one of them
has bounded density with respect to the $(d+1)$-dimensional Lebesgue measure $m$ in $\R^{d+1}$ because
then the corresponding double integral converges absolutely and the classical Fubini theorem applies.
This simple observation will be sufficient for us most of the time. However, in a couple of places
the adjoint operator $R^*$ has to be understood in the usual sense of functional analysis in the 
Hilbert space $L^2(\mu)$ for some good measure $\mu$. All such cases are covered by the 
following general scheme (which is, perhaps, even too general for the purposes of this paper).

The identity 
$$
\langle R_{\mu,\delta}f,g\rangle_\mu=-\left\langle f,\sum_{j=1}^{d+1}\langle e_j,R_{\mu,\delta}
\langle g,e_j\rangle\rangle\right\rangle_\mu
$$
holds for every locally finite measure $\mu$ and any bounded functions $f$ (scalar valued) and
$g$ (vector valued) with compact supports. If $\mu$ is good, both sides of this identity
make sense and define continuous bilinear forms in $L^2(\mu)\times L^2(\mu)$. Since these
forms coinside on a dense set of pairs of test-functions, they must coincide everywhere. However,
the latter is equivalent to saying that 
$$
(R_{\mu,\delta})^*g=-\sum_{j=1}^{d+1}\langle e_j,R_{\mu,\delta}
\langle g,e_j\rangle\rangle
$$
in the usual sense of functional analysis.

Finally, if the operators $R_{\mu,\delta}$ converge at least weakly to some operator $R_\mu$ in $L^2(\mu)$ as $\delta\to 0+$,
so do the operators $(R_{\mu,\delta})^*$ and, therefore, the last identity remains valid for $R_\mu$ in place of $R_{\mu,\delta}$.

The upshot of these observation is that all reasonable properties of or estimates for $R$, $R_{\mu,\delta}$, or $R_{\mu}$
automatically hold for $R^*$, $(R_{\mu,\delta})^*$, or $(R_{\mu})^*$ respectively due to one of the above identities, so we may (and will)
freely refer to the results formally obtained only for the operators themselves when talking about their adjoints. 

In what follows, we will mainly deal with measures $\mu$ that satisfy not only the
upper growth bound, but a lower one as well. Such measures are called Ahlfors--David
regular (AD regular for short). The exact definition is as follows.

\begin{udef}
Let $U$ be an open subset of $\R^{d+1}$. A nice measure $\mu$ is called AD regular in $U$
with lower regularity constant $c>0$ if for every $x\in\supp\mu\cap U$ and every $r>0$ such that
$B(x,r)\subset U$, we have $\mu(B(x,r))\ge cr^d$.
\end{udef}

The simplest example of a good AD regular measure $\mu$ in $\R^{d+1}$ is the $d$-dimensional
Lebesgue measure $m\ci L$ on an affine hyperplane $L\subset\R^{d+1}$. The next section is devoted
to the properties of the Riesz transform with respect to this measure.

\section{Riesz transform of a smooth measure supported on a hyperplane}

Throughout this section, $L$ is a fixed affine hyperplane in $\R^{d+1}$ and $H$ is the
hyperplane parallel to $L$  passing through the origin.

\goal
The main results of this section are the explicit bounds for the $L^\infty$-norm
and the Lipschitz constant of the $H$-restricted Riesz transform $R^H\nu$ of a
measure $\nu=f m\ci L$ with compactly supported $C^2$ density $f$ with respect to $m\ci L$.
\goalend

If we are interested in the values of $R_{m_L,\delta}f$ on the hyperplane $L$ only, we may just
as well project the kernels $K_\delta$ to $H$ and define
$$
K_\delta^H(x)=\frac{\pi\ci H x}{\max(\delta,|x|)^{d+1}}
$$
where $\pi\ci H$ is the orthogonal projection from $\R^{d+1}$ to $H$. The corresponding
operators $R_\delta^H$ will just miss the orthogonal to $H$ component of the difference
$x-y$ in the convolution definition. However, for $x,y\in L$, this component vanishes
anyway. 

Note that everything that we said about the full Riesz transform $R$ and its adjoint $R^*$ in the 
previous section applies to the restricted Riesz transform $R^H$ as well, except in the
identities relating the adjoint operator $(R^H)^*$ to the operator $R^H$ itself an 
orthonormal basis $e_1,\dots,e_d$ of $H$ should be used instead of an orthonormal
basis in the whole space $\R^{d+1}$.

The theory of the $d$-dimensional Riesz transform on a hyperplane $L$ in $\R^{d+1}$
is mainly just the classical
theory of the full-dimensional Riesz transform in $\R^d$. The facts important for us
(which can be found in any decent harmonic analysis textbook) are the following.

The operators $R_{m\cci L,\delta}^H$ are uniformly bounded in every $L^p(m\ci L)$ ($1<p<+\infty$).
Moreover, they have a strong limit as $\delta\to 0+$, which we will denote by $R^H_{m\cci L}$. This
operator is also bounded in all $L^p(m\ci L)$, is an isometry in $L^2(m\ci L)$ (up to a constant
factor), and
$$
\left(R_{m\cci L}^H\right)^*R_{m\cci L}^H=-c\operatorname{Id}
$$
for some $c> 0$. Here, $\left(R_{m\cci L}^H\right)^*$ stands for the
adjoint operator to the operator $R_{m\cci L}$. Note that $\left(R_{m\cci L}^H\right)^*$
can also be defined as the strong limit of the pointwise defined operators
$\left(R_{m\cci L,\delta}^H\right)^*$.

\begin{lem}
\label{smoothtolip}
Suppose that $f$ is a $C^2$ smooth compactly supported function on $L$. Then
$
R_\delta^H(f\,m\ci L)
$
converge to some limit $R^H(f\,m\ci L)$  uniformly on the entire space $\R^{d+1}$ as $\delta\to 0+$, and 
$R^H(f\,m\ci L)$ coincides with $R_{m\cci L}^Hf$ almost everywhere on $L$
with respect to $m\ci L$.
Moreover, $R^H(f\,m\ci L)$ is a Lipschitz function in $\R^{d+1}$ harmonic outside
$\supp(f\,m\ci L)$, and we have
$$
\sup |R^H(f\,m\ci L)|\le CD^2\sup_{L}|\nabla\ci H^2 f|
$$
and
$$
\|R^H(f\,m\ci L)\|\ci{\Lip}\le CD\sup_{L}|\nabla\ci H^2 f|\,
$$
where $D$ is the diameter of $\supp (f\,m\ci L)$ and $\nabla_H$ is the partial gradient
involving only the derivatives in the directions parallel to $H$.
\end{lem}

Note that the second differential $\nabla\ci H^2 f$
and the corresponding supremum on the right hand side are considered on $L$ only (the function $f$ in
the lemma doesn't even need to be defined outside $L$) while the $H$-restricted Riesz
transform $R^H(f\,m\ci L)$ on the left hand side is viewed as a function on the entire space
$\R^{d+1}$
and its supremum and the Lipschitz norm are also taken in $\R^{d+1}$.

It is very
important that we consider here the $H$-restricted Riesz transform $R^H$ instead of the full
Riesz transform $R$. The reason is that the component of $R(f\,m\ci L)$ orthogonal to $H$  has
a jump discontinuity across $L$ at the points of $L$ where $f\ne 0$. This switch to the
restricted Riesz transform is rather crucial for our proof and is somewhat
counterintuitive given the way the argument
 will develop later, when we use the boundedness
of the Riesz transform $R_\mu$ in $L^2(\mu)$ to show, roughly speaking,
 that almost flat pieces of
$\mu$ parallel to $H$ must be aligned. It would seem more natural to do exactly the
opposite and to concentrate on the orthogonal
component of $R$ for that purpose. However, the price one has to pay for its discontinuity
is very high and we could not make the ends meet in that way.

\begin{proof}
The statement about the harmonicity of $R^H(f\,m\ci L)$ follows from the observation
that $K^H(x)=c\nabla\ci H E(x)$ where $E(x)$ is the fundamental solution for the
Laplace operator in $\R^{d+1}$, i.e.,
$E(x)=c\log|x|$ when $d=1$ and $E(x)=-c|x|^{-(d-1)}$ when $d>1$.
Thus, $K^H$ is harmonic outside the origin together with $E$, so $R^H\nu$ is harmonic
outside $\supp\nu$ for every finite signed measure $\nu$ (and so is the full Riesz transform
$R\nu$).

To prove the other statements of the lemma, note that its setup is translation and rotation
invariant, so we can assume without loss of generality that $L=H=\{x\in\R^{d+1}:x_{d+1}=0\}$.
We shall start with proving the uniform bounds for the regularized Riesz transforms
$R^H_\delta (f\,m\ci L)$. Since $R^H_\delta (f\,m\ci L)$ is a Lipschitz function in the entire
space, it is enough to estimate its value and its gradient at each point $x\in\R^{d+1}$. By
translation invariance and symmetry, we may assume without loss of generality that
$x_1=\ldots=x_d=0$ and $x_{d+1}=t\ge 0$.

We have
\begin{multline*}
[R^H_\delta (f\,m\ci L)](x)=\int_L K_\delta^H(x-y)f(y)\,dm\ci L(y)
\\
=
\int_{L\cap B(0,D)}+\int_{L\setminus B(0,D)}=I_1+I_2\,.
\end{multline*}
Note that, for $|y|\ge D$, the integrand is bounded by $D^{-d}\max_L|f|$ and the
$m_L$ measure of the support of $f$ on $L$ is at most $CD^d$, so
$$
|I_2|\le C\max_L|f|\,.
$$
To estimate $I_1$, note first that
$$
\int_{L\cap B(0,D)} K_\delta^H(x-y)\,dm\ci L(y)=0\,,
$$
so we can replace $f(y)$ by $f(y)-f(0)$ and use the inequalities
$|f(y)-f(0)|\le \sup_L|\nabla\ci H f|\cdot|y|$ and $|x-y|\ge |y|$ to get
$$
|I_1|\le \sup_L|\nabla\ci H f|\int_{L\cap B(0,D)}\frac{dm\ci L(y)}{|y|^{d-1}}\le CD\sup_L|\nabla\ci H f|\,.
$$
Adding these bounds and using the inequality $\sup_L |f| \le D \sup_L|\nabla\ci H f|$, we get
$$
\sup |R^H_\delta(f\,m\ci L)|\le CD\sup_L|\nabla\ci H f|\,.
$$
Note that we haven't used that $f\in C^2(L)$ here, only that $f\in C^1(L)$.

Now we will estimate $[\nabla R^H_\delta (f\,m\ci L)](x)$. Note that the partial derivatives
$\frac\partial{\partial x_j}$ for $j=1,\dots,d$ that are taken along the hyperplane $L$
can be passed to $f$, so we have
$$
\frac\partial{\partial x_j}[R^H_\delta (f\,m\ci L)]=R^H_\delta ([\tfrac\partial{\partial x_j}f]\,m\ci L)\,.
$$
Applying the above estimate to $\tfrac\partial{\partial x_j}f$ instead of $f$, we immediately obtain
$$
\sup |\nabla\ci H R^H_\delta(f\,m\ci L)|\le CD\sup_{L}|\nabla^2 f|\,.
$$

To get a bound for the remaining vertical derivative $\frac{\partial}{\partial x_{d+1}}$, note that
$
\frac{\partial}{\partial x_{d+1}}K_\delta^H(x-y)=0
$
for all $x,y\in L$, so the case $t=0$ is trivial. Assuming $t>0$, we write
\begin{multline*}
\left[\frac{\partial}{\partial x_{d+1}} R^H_\delta (f\,m\ci L)\right](x)=\int_L
\left[\frac{\partial}{\partial x_{d+1}}K_\delta^H(x-y)\right]f(y)\,dm\ci L(y)
\\
=
\int_{L\cap B(0,D)}+\int_{L\setminus B(0,D)}=I_1+I_2\,.
\end{multline*}
For $y\in L$, we can use the inequalities
$$
\left|\frac{\partial}{\partial x_{d+1}}K_\delta^H(x-y)\right|\le C\frac{t}{|x-y|^{d+2}}
$$
and $|x-y|\ge t$ and note that the integrand in $I_2$ is bounded by $\sup_L|f|D^{-(d+1)}$.
Since the
$m_L$ measure of the support of $f$ on $L$ is at most $CD^d$, we arrive at the bound
$$
|I_2|\le CD^{-1}\sup_L|f|\,.
$$
To estimate $I_1$, note that we still have the cancellation property
$$
\int_{L\cap B(0, D)} \frac{\partial}{\partial x_{d+1}}K_\delta^H(x-y)\,dm\ci L(y)=0\,,
$$
so we can replace $f(y)$ by $f(y)-f(0)$ and use the inequalities
$|f(y)-f(0)|\le \sup_L|\nabla\ci H f|\cdot|y|$ and $|x-y|\ge |y|$ to get
$$
|I_1|\le C\sup_L|\nabla\ci H f|\int_{L}\frac{t\, dm\ci L(y)}{|x-y|^{d+1}}= C\sup_L|\nabla\ci H f|\,.
$$
Adding these bounds, we get
$$
\sup \left|\frac{\partial}{\partial x_{d+1}}R^H_\delta(f\,m\ci L)\right|\le
 CD^{-1}\left[\sup_L|f|+D\sup_L|\nabla\ci H f|\right]\,.
$$
To get only $\sup_L|\nabla\ci H^2 f|$ on the right hand sides of our estimates, it remains to note that
$$
\sup_L|f|\le D\sup_L|\nabla\ci H f|\le D^2\sup_L|\nabla\ci H^2 f|\,.
$$
Since the estimates obtained are uniform in $\delta>0$ and since $R^H_\delta(f\,m\ci L)$ coincides with
$R^H(f\,m\ci L)$ outside the strip of width $\delta$ around $L$, we conclude that $R^H_\delta(f\,m\ci L)$
converges uniformly to some Lipschitz function in the entire space $\R^{d+1}$ and the limiting function
satisfies the same bounds. Since they also converge to $R^H_{m\cci L}f$ in $L^2(m\ci L)$, this limiting
function must coincide with $R^H_{m\cci L}f$ almost everywhere with respect to the measure $m\ci L$.
\end{proof}

\section{Toy flattening lemma}

The goal of this section is to prove the result that is, in a sense, the converse to Lemma \ref{smoothtolip}.
We want to show that if $R^H_{m\cci L}f$ is smooth in a large ball on $L$, then $f$ itself must be (slightly
less) smooth in the $4$ times smaller ball. The exact version we will need is
\begin{lem}
\label{toyflat}
Let $f\in L^\infty(m\ci L)\cap L^2(m\ci L)$. Assume that $z\in L$ and $R^H_{m\cci L}f$ coincides
with a $C^2$ function $F$ almost everywhere (with respect to $m\ci L$)
on $L\cap B(z,4A)$ for some $A>0$. Then $f$ is Lipschitz
on $L\cap B(z,A)$ (possibly, after a correction on a set of $m\ci L$ measure $0$) and
the norm
$\|f\|\ci{\Lip(L\cap B(z,A))}$
is dominated by
$$
A^{-1}\|f\|\ci{L^\infty(m\cci L)}+\|\nabla\ci H F\|\ci{L^\infty(L\cap B(z,4A))}+
A\|\nabla^2\ci H F\|\ci{L^\infty(L\cap B(z,4A))}
$$
up to a constant factor.
\end{lem}

We will refer to this lemma as the ``toy flattening lemma''. By itself, it is rather elementary
but, combined with some weak limit techniques, it will eventually yield the full flattening lemma
for measures that are not necessarily supported on a hyperplane, which will play a crucial
role in our argument.

\begin{proof}
Write
$$
f=f\chi\ci {B(z,4A)}+f\chi\ci {L\setminus B(z,4A)}=f_1+f_2\,.
$$
Note that $R^H_{m\cci L}f_2$ is smooth in $L\cap B(z,3A)$ and
$$
\|\nabla\ci H R^H_{m\cci L}f_2\|\ci{L^\infty(L\cap B(z,3A))}\le CA^{-1}\|f\|\ci{L^\infty(m\cci L)}
$$
and
$$
\|\nabla^2\ci H R^H_{m\cci L}f_2\|\ci{L^\infty(L\cap B(z,3A))}\le CA^{-2}\|f\|\ci{L^\infty(m\cci L)}
$$
To see it, just recall the estimate (\ref{smoothest})
and note that for $k\ge 1$ and $x\in B(z,3A)$, we have
$$
\int_{L\setminus B(z,4A)}\frac{|f(y)|\,dm\ci L(y)}{|x-y|^{d+k}}\le C(k)\|f\|\ci{L^\infty(m\cci L)}A^{-k}\,.
$$
Thus, $R^H_{m\cci L}f_1$ is $C^2$-smooth on $L\cap B(z,3A)$ as the difference of $F$ and $R^H_{m\cci L}f_2$.
Moreover, we have
$$
\|\nabla\ci H R^H_{m\cci L}f_1\|\ci{L^\infty(L\cap B(z,3A))}\le
CA^{-1}\|f\|\ci{L^\infty(m\cci L)}+\|\nabla\ci H F\|\ci{L^\infty(L\cap B(z,4A))}
$$
and
$$
\|\nabla^2\ci H R^H_{m\cci L}f_1\|\ci{L^\infty(L\cap B(z,3A))}\le
CA^{-2}\|f\|\ci{L^\infty(m\cci L)}+\|\nabla^2\ci H F\|\ci{L^\infty(L\cap B(z,4A))}\,.
$$
Observe also that, by the $L^2(m\ci L)$ boundedness of $R^H\ci{m\cci L}$, we have
$$
\int |R^H_{m\cci L}f_1|^2\,dm\ci L\le C\int |f_1|^2\,dm\ci L\le C A^d\|f\|\ci{L^\infty(m\cci L)}^2\,,
$$
whence there exists a point in $L\cap B(z,3A)$ such that $|R^H_{m\cci L}f_1|\le
C\|f\|\ci{L^\infty(m\cci L)}$
at that point. Combining this with the estimate for the gradient, we conclude that
$$
\|R^H_{m\cci L}f_1\|\ci{L^\infty(L\cap B(z,3A))}\le
C\left[\|f\|\ci{L^\infty(m\cci L)}+A\|\nabla\ci H F\|\ci{L^\infty(L\cap B(z,4A))}\right]\,.
$$
Let now $\f_0$ be a $C^2$ smooth function in $\R^{d+1}$ supported on $B(0,3)$ such that $0\le \f_0\le 1$ and
$\f_0$ is identically $1$ on $B(0,2)$. Put $\f(x)=\f_0\left(\frac{x-z}A\right)$. Then
$|\nabla^k\f|\le C(k)A^{-k}$. We have
$$
-c f_1=\left(R^H_{m\cci L}\right)^* R^H_{m\cci L} f_1=
\left(R^H_{m\cci L}\right)^*[\f R^H_{m\cci L}f_1]+
\left(R^H_{m\cci L}\right)^*[(1-\f)R^H_{m\cci L}f_1]\,.
$$
However, $\f R^H_{m\cci L}f_1$ is a compactly supported $C^2$ function on $L$, the diameter of its
support is not greater than $6A$ and, using the above estimates
and the Leibniz formulae for the derivative of a product, we see that
its second gradient $\nabla^2\ci H[\f R^H_{m\cci L}f_1]$ is
dominated by
$$
A^{-2}\|f\|\ci{L^\infty(m\ci L)}+
A^{-1}\|\nabla\ci H F\|\ci{L^\infty(L\cap B(z,4A))}+
\|\nabla^2\ci H F\|\ci{L^\infty(L\cap B(z,4A))}
$$
up to a constant factor.
Thus, by Lemma \ref{smoothtolip}, $\left(R^H_{m\cci L}\right)^*[\f R^H_{m\cci L}f_1]$ is Lipschitz on
$L$ with Lipschitz constant dominated by the quantity in the statement of the lemma to prove.

To finish the proof of the toy flattening lemma it  just remains to observe that, since
$\left(R^H_{m\cci L}\right)^*[(1-\f)R^H_{m\cci L}f_1]$ is a Riesz transform
of a function supported outside the ball $B(z,2A)$
(or, rather, a finite linear combination of such Riesz transforms), it is automatically
smooth on $B(z,A)$. Moreover, using (\ref{smoothest}) again, we see that
\begin{multline*}
\left|\nabla\ci H \left(R^H_{m\cci L}\right)^*[(1-\f)R^H_{m\cci L}f_1]\right|\le
\left|\nabla\left(R^H \right)^*[(1-\f)(R^H_{m\cci L}f_1)m\ci L]\right|
\\
\le
\int_{L\setminus B(z,2A)}\frac{1}{|x-y|^{d+1}}|R^H_{m\cci L}f_1(y)|\,dm\ci L(y)
\\
\le
\left[\int_{L\setminus B(z,2A)}\frac{dm\ci L(y)}{|x-y|^{2d+2}}\right]^{1/2}
\left[\int_{L\setminus B(z,2A)}|R^H_{m\cci L}f_1(y)|^2\,dm\ci L(y)\right]^{1/2}
\\
\le
\left[CA^{-(d+2)}\right]^{1/2}\left[CA^d\|f\|^2\ci{L^\infty(m\cci L)}\right]^{1/2}
\le
CA^{-1}\|f\|\ci{L^\infty(m\cci L)}\,.
\end{multline*}
\end{proof}

\section{Weak limits}
\label{weaklimits}

\goal
This section has two main goals. The first one is to define the Riesz transform operators $R_\mu$
(and their $H$-restricted versions $R^H_\mu$) in $L^2(\mu)$
for arbitrary good measures $\mu$ as weak limits
of the regularized operators $R_{\mu,\delta}$ as $\delta\to 0+$. The second one
is to show that when a sequence of uniformly
good measures $\mu_k$ tends weakly
(over the space of compactly supported continuous functions in
$\R^{d+1}$) to some other measure $\mu$ in $\R^{d+1}$, then the limiting measure $\mu$ is also good
and for all compactly supported Lipschitz functions $f$ (scalar) and $g$ (vector-valued) in $\R^{d+1}$, we
have $\int \langle R_{\mu_k} f,g\rangle\,d\mu_k\to \int \langle R_{\mu} f,g\rangle\,d\mu$.
\goalend

Our starting point is to fix two compactly supported Lipschitz functions $f$ and $g$ in $\R^{d+1}$,
where $f$ is scalar-valued and $g$ is vector-valued, and to use the antisymmetry of the kernels $K_\delta$
to write the scalar product $\langle R_{\mu,\delta}f,g\rangle_\mu$ as
\begin{multline*}
I_\delta(f,g)=\iint \langle K_\delta(x-y)f(y),g(x)\rangle\,d\mu(x)\,d\mu(y)
\\
=
\iint \langle K_\delta(x-y),H(x,y)\rangle\,d\mu(x)\,d\mu(y)
\end{multline*}
where
$$
H(x,y)=\frac 12[f(y)g(x)-f(x)g(y)]\,.
$$
The vector-valued function $H(x,y)$ is compactly supported and Lipschitz on $\R^{d+1}\times \R^{d+1}$, so the
integral $I_\delta(f,g)$ converges absolutely as an integral of a bounded function over a set of finite
measure for every $\delta>0$ and every locally finite measure $\mu$.
Moreover, since $H$ vanishes on the diagonal $x=y$, we have
$$
|H(x,y)|\le C(f,g)|x-y|
$$
for all $x,y\in\R^{d+1}$.

If $\mu$ is nice, then
$$
\int_{B(x,r)}\frac{d\mu(y)}{|x-y|^{d-1}}\le Cr
$$
for all $x\in\R^{d+1}$ and $r>0$. Therefore, denoting $\supp f\cup\supp g$ by $S$, we get
\begin{multline*}
\iint_{x,y:|x-y|<r}\frac{|H(x,y)|}{|x-y|^{d}}\,d\mu(x)\,d\mu(y)
\\
\le
C(f,g)\int_S\left(\int_{y:|x-y|<r}\frac{d\mu(y)}{|x-y|^{d-1}}\right)\,d\mu(x)
\le C(f,g)\mu(S)r\,.
\end{multline*}
In particular, taking $r=\operatorname{diam} S$ here, we conclude that the full integral
$$
\iint\frac{|H(x,y)|}{|x-y|^{d}}\,d\mu(x)\,d\mu(y)=\iint_{S\times S}<+\infty\,.
$$
Since $|K(x)|=|x|^{-d}$ and $|K_\delta(x)-K(x)|\le |x|^{-d}\chi\ci{B(0,\delta)}(x)$, we infer that
the integral
$$
I(f,g)=\iint\langle K(x-y),H(x,y)\rangle\,d\mu(x)\,d\mu(y)
$$
converges absolutely and, moreover, there exists a constant $C$ depending on $f$, $g$, and the
growth constant of $\mu$ only such that $|I_\delta(f,g)-I(f,g)|\le C\delta$ for all $\delta>0$.

This already allows one to define the bilinear form
$$
\langle R_\mu f,g\rangle_\mu=I(f,g)
$$
and to establish the existence of the limit operator $R_\mu =\lim_{\delta\to 0+} R_{\mu,\delta}$ as
an operator from the space of Lipschitz functions to its dual for every nice measure $\mu$.

However, if $\mu$ is good, we can say much more. Indeed, in this case the bilinear
forms
$$
\langle R_{\mu,\delta} f,g\rangle_\mu=\int\langle R_{\mu,\delta} f,g\rangle\,d\mu
$$
make sense and satisfy the inequality
$$
\left|\langle R_{\mu,\delta} f,g\rangle_\mu\right|\le C\|f\|\ci{L^2(\mu)}\|g\|\ci{L^2(\mu)}
$$
for all $f,g\in L^2(\mu)$. Since the space of compactly supported Lipschitz functions is dense
in $L^2(\mu)$, we can write any $L^2(\mu)$ functions $f,g$ as $f_1+f_2$ and $g_1+g_2$ where
$f_1,g_1$ are compactly supported Lipschitz functions in $\R^{d+1}$ and $f_2,g_2$ have as small
norms in $L^2(\mu)$ as we want. Splitting
$$
\langle R_{\mu,\delta} f,g\rangle_\mu=\langle R_{\mu,\delta} f_1,g_1\rangle_\mu+
[
\langle R_{\mu,\delta} f_1,g_2\rangle_\mu+\langle R_{\mu,\delta} f_2,g\rangle_\mu]\,,
$$
we see that $\langle R_{\mu,\delta} f,g\rangle_\mu$ can be written as a sum of the quantity
$\langle R_{\mu,\delta} f_1,g_1\rangle_\mu=I_\delta(f_1,g_1)$, which converges to a finite
limit $I(f_1,g_1)$ as $\delta\to 0+$ and another quantity that stays as small as we want
as $\delta\to 0+$ if the $L^2(\mu)$ norms of $f_2$ and $g_2$ are chosen small enough. From
here we conclude that the limit of $\langle R_{\mu,\delta} f,g\rangle_\mu$ as $\delta\to 0+$
exists for all $f,g\in L^2(\mu)$. Moreover, this limit is a bilinear form in $L^2(\mu)$ and
it is still bounded by $C\|f\|\ci{L^2(\mu)}\|g\|\ci{L^2(\mu)}$. By the Riesz-Fisher theorem,
there exists a unique bounded linear operator $R_\mu$ in $L^2(\mu)$ such that this bilinear
form is equal to $\langle R_\mu f,g\rangle_\mu$. The convergence
$$
\langle R_{\mu,\delta} f,g\rangle_\mu\to \langle R_\mu f,g\rangle_\mu\quad\text{ as }\delta\to 0+
$$
can be restated as the weak convergence of the operators $R_{\mu,\delta}$ to $R_\mu$.

Similarly, one can consider the duality coupling of $L^p(\mu)$ and $L^q(\mu)$ where
$p,q>1$ and $p^{-1}+q^{-1}=1$ and use the uniform boundedness of the operators
$R_{\mu,\delta}$ in $L^p(\mu)$ to establish the existence of the weak limit of the
operators $R_{\mu,\delta}$ in $L^p(\mu)$ as $\delta\to 0+$. Note that if
$f\in L^{p_1}(\mu)\cap L^{p_2}(\mu)$, then for every $g\in L^\infty(\mu)$ with $\mu(\supp g)<+\infty$,
the value
$\langle R_{\mu,\delta} f,g\rangle_\mu$ can be computed using the pointwise
integral definition of $R_{\mu,\delta} f$ as $R_\delta(f\mu)$, so it does not depend
on whether $f$ is considered as an element of $L^{p_1}(\mu)$ or an element of $L^{p_2}(\mu)$.
Thus
$$
\langle R_{\mu} f,g\rangle_\mu=\lim_{\delta\to 0+} \langle R_{\mu,\delta} f,g\rangle_\mu
$$
also doesn't depend upon that (note that $g\in L^{q_1}(\mu)\cap L^{q_2}(\mu)$, so the left
hand side makes sense in both cases). Since $g$ is arbitrary here, we conclude
that $R_\mu f$ (as a function defined $\mu$-almost everywhere) is the same in both cases.

Another important observation is that if the pointwise limit
$\lim_{\delta\to 0+}R_{\mu,\delta}f$ exists on a set $E$ with $\mu(E)>0$, then $R_\mu f$
coincides with that limit $\mu$-almost everywhere on $E$. To prove it, just observe that, by
Egorov's theorem, we can exhaust $E$ by sets of finite $\mu$ measure on which
the convergence is uniform.

At last, if $R_{\mu,\delta}$ converges strongly in $L^2(\mu)$, then the limit is still the same
as the weak limit we constructed.

The analogous theory can be built for $R^H$, $R^*$, and $\left(R^H\right)^*$. We built it only
for the full operator $R$ because projecting everything to $H$ is trivial and $R^*$ doesn't
really require a separate theory due to relation (\ref{RtoRstar}), which shows that, at least
in principle, we can always view $R^*$ just as a fancy notation for the right hand side of
(\ref{RtoRstar}). From now on, we will always understand $R(f\mu)$ on $\supp\mu$
as $R_\mu f$ whenever $\mu$ is good and $f\in L^p(\mu)$ for some $p\in(1,+\infty)$.
As we have shown above,
this convention is consistent with other reasonable definitions in the sense that when some
other definition is applicable somewhere on $\supp\mu$ as well, the value it gives coincides with $R_\mu f$ except,
maybe, on a set of zero $\mu$ measure.

The idea of defining $R_\mu$ as a weak limit of $R_{\mu,\delta}$ goes back to Mattila and
Verdera \cite{MV}. They prove its existence in a slightly more general setting and their
approach is somewhat different from ours. They also show that $R_\mu f$ can be defined pointwise
by some formula that is almost the expression for the principal
value
$$
\lim_{\delta\to 0+}\int_{y:|x-y|>\delta} K(x-y)f(y)\,d\mu(y)
$$
but not quite. Note that Mattila, Preiss, and Tolsa showed that the existence of the principal
value $\mu$-almost everywhere is strong enough to imply the rectifiability of $\mu$ (see \cite{MP}
and \cite{T2}),
so for a while there was a hope that the Mattila-Verdera result would eventually lead to the
proof of the rectifiability conjecture. However, as far as we can tell, nobody still knows
how to get a proof in this way and we will use a different route below.

We have just attained the first goal of this section: the construction of the limiting operator
$R_\mu$ for {\em one fixed} good measure $\mu$. Now we turn to the relations between the
operators $R_\mu$ corresponding to different measures $\mu$.

We start with the case when
a positive measure $\nu$ has a bounded Borel measurable density $p$ with respect
to a good measure $\mu$. Since $\nu(B(x,r))\le\|p\|\ci{L^\infty(\mu)}\mu(B(x,r))$,
we see that $\nu$ is nice. To show that $\nu$ is good,
note that for every $f\in L^2(\nu)$, we have
$pf\in L^2(\mu)$. Moreover, we have the identity
$$
R_{\delta}(f\nu)=R_{\delta}(pf\mu)
$$
pointwise in $\R^{d+1}$, whence
\begin{multline*}
\int|R_{\delta}(f\nu)|^2\,d\nu=\int|R_{\delta}(pf\mu)|^2p\,d\mu
\\
\le
C\|p\|\ci{L^\infty(\mu)}\int|pf|^2\,d\mu \le
C\|p\|\ci{L^\infty(\mu)}^2\int|f|^2\,d\nu
\end{multline*}
due to the goodness of $\mu$.
Thus, both operators $R_\nu$ and $R_\mu$ exist. Now take any $f,g\in L^2(\nu)$ and write
$$
\langle R_{\nu,\delta}f,g\rangle_\nu=\langle R_{\mu,\delta}(pf),(pg)\rangle_\mu\,.
$$
Passing to the limit on both sides as $\delta\to 0+$, we conclude that
$$
\langle R_{\nu}f,g\rangle_\nu=\langle R_{\mu}(pf),(pg)\rangle_\mu
=\langle R_{\mu}(pf),g\rangle_\nu
$$
(note that the function $R_{\mu}(pf)$ is defined $\mu$-almost everywhere, so it is also
defined $\nu$-almost everywhere). However, the mapping $f\mapsto R_\mu(pf)$ is a bounded
linear operator from $L^2(\nu)$ to $L^2(\mu)\subset L^2(\nu)$, so we conclude that
$$
R_{\nu}f=R_{\mu}(pf)\quad\text{ $\nu$-almost everywhere.}
$$
This identity is, of course, by no means surprising. Still, since we will use it several
times without mentioning, we decided it would be prudent to include a proof. The next
property we need is a bit subtler. 

Suppose that $\mu_k$ ($k\ge 1$) is a sequence of uniformly nice measures that converges to some
locally finite measure $\mu$ weakly over the space $C_0(\R^{d+1})$ of compactly supported continuous functions
in $\R^{d+1}$. We shall start with showing that $\mu$ is also nice. Indeed, take any ball $B(x,r)$.
Then $\mu(B(x,r))$ can be found as the supremum of all integrals $\int f\,d\mu$ with continuous functions
$f$ such that $0\le f\le 1$ and $\supp f\subset B(x,r)$. However, for every such $f$, we have
$$
\int f\,d\mu=\lim_{k\to\infty}\int f\,d\mu_k\le\sup_k\mu_k(B(x,r))\le Cr^d
$$
where $C$ is the uniform growth constant of $\mu_k$, so we have the same bound for $\mu(B(x,r))$.

Fix two compactly
supported Lipschitz functions $f$ and $g$. The bilinear form $\langle R_{\mu} f,g\rangle_\mu$
can be defined as $I(f,g)$ for every nice measure $\mu$.
Once we know that $\mu$ is nice, we can say that
$$
|\langle R_{\mu_k,\delta} f,g\rangle_{\mu_k}-\langle R_{\mu_k} f,g\rangle_{\mu_k}|\le C\delta
$$
for all $k\ge 1$ and also
$$
|\langle R_{\mu,\delta} f,g\rangle_{\mu}-\langle R_{\mu} f,g\rangle_{\mu}|\le C\delta
$$
with some $C>0$ depending only on $f$, $g$, and the uniform growth constant of $\mu_k$.
Note, however, that for every fixed $\delta>0$,
$$
\langle R_{\mu_k,\delta} f,g\rangle_{\mu_k}=\iint \langle K_\delta(x-y)f(y),g(x)\rangle\,d\mu_k(x)\,d\mu_k(y)
$$
and the integrand is a compactly supported Lipschitz function in $\R^{d+1}\times\R^{d+1}$, which is more
than enough to ensure that
$$
\langle R_{\mu_k,\delta} f,g\rangle_{\mu_k}\to \langle R_{\mu,\delta} f,g\rangle_{\mu}
$$
for every fixed $\delta>0$ as $k\to+\infty$.
Since the convergence $\langle R_{\mu_k,\delta} f,g\rangle_{\mu_k}\to
\langle R_{\mu_k} f,g\rangle_{\mu_k}$ as $\delta\to 0+$
is uniform in $k$, we conclude that
$$
\langle R_{\mu_k} f,g\rangle_{\mu_k}\to \langle R_{\mu} f,g\rangle_{\mu}
$$
as well.

It remains to show that if $\mu_k$ are uniformly good, then $\mu$ is also good,
so all the bilinear forms in question can be also interpreted as $L^2(\mu)$ couplings.
Return to the regularized operators $R_{\mu,\delta}$ and note that the uniform $C$-goodness
of $\mu_k$ implies that
$$
|\langle R_{\mu_k,\delta} f,g\rangle_{\mu_k}|\le C\|f\|\ci {L^2(\mu_k)}\|g\|\ci {L^2(\mu_k)}\,.
$$
Since $|f|^2$ and $|g|^2$ are compactly supported Lipschitz functions, we can pass to the
limit on both sides and get
$$
|\langle R_{\mu,\delta} f,g\rangle_{\mu}|\le C\|f\|\ci {L^2(\mu)}\|g\|\ci {L^2(\mu)}\,.
$$
However, the operators $R_{\mu,\delta} f$ are well-defined pointwise for every $f\in L^2(\mu)$
and are bounded from $L^2(\mu)$ to $L^2_{\operatorname{loc}}(\mu)$ as soon as $\mu$ is merely nice.
Using the fact that, for every bounded open set $U$,
the space of compactly supported inside $U$ Lipschitz functions is dense in $L^2(U,\mu)$ and
this a priori boundedness, we conclude that $\|R_{\mu,\delta}\|\ci{L^2(\mu)\to L^2(U,\mu)}\le C$
regardless of the choice of $U$. The monotone convergence lemma then shows that
$\|R_{\mu,\delta}\|\ci{L^2(\mu)\to L^2(\mu)}\le C$ as well, finishing the story.

\section{The flatness condition and its consequences}
\label{flatnesscondition}

Throughout this section, we shall fix a linear hyperplane
$H\subset\R^{d+1}$. Let $z\in\R^{d+1}$, $A,\alpha,\ell>0$ (we view $A$ as a large number,
$\alpha$ as a small number, and $\ell$ as a scale parameter). We will be
interested in the situation when the measure $\mu$ is close inside the ball $B(z,A\ell)$
to a multiple of the $d$-dimensional Lebesgue measure $m\ci L$ on the affine hyperplane $L$
containing $z$ and parallel to $H$.

\begin{udef}
We say that a measure $\mu$ is geometrically $(H,A,\alpha)$-flat at the point $z$ on the scale $\ell$ if
every point of $\supp\mu\cap B(z,A\ell)$ lies within distance $\alpha\ell$ from the 
 affine hyperplane $L$ containing $z$ and parallel to $H$ and every point of
$L\cap B(z,A\ell)$ lies within distance $\alpha\ell$ from $\supp\mu$.

We say that a measure $\mu$ is $(H,A,\alpha)$-flat at the point $z$ on the scale $\ell$ if
it is geometrically $(H,A,\alpha)$-flat at the point $z$ on the scale $\ell$ and, in addition,
for every Lipschitz function $f$ supported on $B(z,A\ell)$ such that $\|f\|_{\Lip}\le \ell^{-1}$ and
$\int f\,dm\ci L=0$, we have
$$
\left|\int f\,d\mu\right|\le \alpha\ell^d\,.
$$

\end{udef}

Note that the geometric $(H,A,\alpha)$-flatness is a condition on $\supp\mu$ only. It doesn't tell one anything
about the distribution of the measure $\mu$ on its support. The latter is primarily controlled by the second, analytic,
condition in the full $(H,A,\alpha)$-flatness.
These two conditions are not completely independent: if, say, $\mu$ is AD regular, then the analytic condition
implies the geometric one with slightly worse parameters. However, it will be convenient for us just to demand
them separately.

One may expect that, for nice enough functions, $(H,A,\alpha)$-flatness of $\mu$ at $z$ on scale $\ell$ would allow
one to switch from the integration with respect to $\mu$ to that with respect to $m\ci L$ in various integrals
over $B(z,A\ell)$ making an error controlled by $\alpha$. This is, indeed, the case and the following lemmata
provide all the explicit estimates of this type that we will need in the future.
\begin{lem}
\label{mutonulip}
Let $\mu$ be a nice measure. Assume that $\mu$ is $(H,A,\alpha)$-flat at $z$ on scale $\ell$ with some $A>5$,
$\alpha\in(0,1)$.
Let $\f$ be any non-negative Lipschitz function supported on $B(z,5\ell)$ with $\int\f\,dm\ci L>0$.
Put
$$
a=\left(\int \f\,dm\ci L\right)^{-1}\int \f\,d\mu,\qquad \nu=a\f m\ci L\,.
$$
Let $\Psi$ be any function with $\|\Psi\|\ci{\Lip(\supp\f)}<+\infty$.
Then
$$
\left|\int \Psi\,d(\f\mu-\nu)\right|\le C\alpha\ell^{d+2}\|\Psi\|\ci{\Lip(\supp\f)}\|\f\|\ci{\Lip}\,.
$$
\medskip
As a corollary, for every $p\ge 1$, we have
$$
\left|\int |\Psi|^p\,d(\f\mu-\nu)\right|\le
C(p)\alpha\ell^{d+2}\|\Psi\|\ci{L^\infty(\supp\f)}^{p-1}\|\Psi\|\ci{\Lip(\supp\f)}\|\f\|\ci{\Lip}\,.
$$
\end{lem}

\begin{lem}
\label{mutonuRH}
Assume in addition to the conditions of Lemma \ref{mutonulip} that $\f\in C^2$, $\mu$ is nice
and that the ratio of integrals $a$ is bounded from above by some known constant. Then
\begin{multline*}
\left|\int \Psi\f [R^H(\f\mu-\nu)]\,d\mu\right|
\\
\le
C\alpha^{\frac 1{d+2}}
\ell^{d+2}\left[\|\Psi\|\ci{L^\infty(\supp\f)}+\ell\|\Psi\|\ci{\Lip(\supp\f)}\right]\|\f\|\ci{\Lip}^2\,.
\end{multline*}
where $C>0$ may, in addition to the dependence on $d$, which goes without mentioning,
depend also on the growth constant of $\mu$ and the upper bound for $a$.
\end{lem}

Note that we can use both scalar and vector-valued functions $\Psi$ in both
lemmas (the product in Lemma \ref{mutonuRH} should be replaced by the scalar
product in the vector-valued version)
and it is enough to prove only the scalar versions because the
vector case can be easily obtained by considering each coordinate separately.

Though we have combined all powers of $\ell$ into one wherever possible
to shorten the formulae, the reader should keep in mind that the scaling
invariant quantities are in fact $\|\cdot\|\ci{L^\infty}$ and $\ell\|\cdot\|\ci{\Lip}$,
so all inequalities actually compare the integrals on the left with
$\ell^d$.

Despite we require $\f\in C^2$ in Lemma \ref{mutonuRH},
only the Lipschitz
norm of $\f$ enters the estimates. The additional smoothness will matter
only because we will use Lemma \ref{smoothtolip} to show that the integral
on the left hand side can be made sense of.

At last, we want to emphasize that only the norm of $\f$ is global
and all norms of $\Psi$ in the bounds are computed on $\supp\f$ only.
We can even assume that $\Psi$ is not defined outside $\supp\f$ because
only the product $\Psi\f$ matters anywhere (don't forget that $\nu$ contains
the factor $\f$ in its definition too).

\begin{proof}
We shall start with proving Lemma \ref{mutonulip}.
Since the signed measure $\f\mu-\nu$ is balanced (i.e., $\int d(\f\mu-\nu)=0$), when proving the first estimate,
we may subtract any constant from $\Psi$, so without loss of generality we may assume
that $\int\Psi\,d\nu=\int\Psi\f\,dm\ci L=0$.

Note now that
$$
\|\Psi\f\|\ci{\Lip}\le \|\Psi\|\ci{\Lip(\supp\f)}\|\f\|\ci{L^\infty}+\|\Psi\|\ci{L^\infty(\supp\f)}\|\f\|\ci{\Lip}\,.
$$
Indeed, when estimating the difference $|\Psi(x)\f(x)-\Psi(y)\f(y)|$, it is enough to consider
the case when at least one of the points $x$ and $y$ belongs to $\supp\f$ because otherwise
the difference is $0$. By symmetry, we may assume without loss of generality that $x\in\supp\f$.
Write
$$
|\Psi(x)\f(x)-\Psi(y)\f(y)|\le |\Psi(x)|\cdot|\f(x)-\f(y)|+|\Psi(x)-\Psi(y)|\cdot|\f(y)|\,.
$$
The first term is, clearly, bounded by $\|\Psi\|\ci{L^\infty(\supp\f)}\|\f\|\ci{\Lip}|x-y|$.
If $y\notin\supp\f$, then the second term is $0$. Otherwise, it is bounded by
$\|\Psi\|\ci{\Lip(\supp\f)}\|\f\|\ci{L^\infty}|x-y|$.

The definition of $(H,A,\alpha)$-flatness at $z$ on scale $\ell$ now
implies that
\begin{multline*}
\left|\int \Psi\,d(\f\mu-\nu)\right|=\left|\int \Psi\f\,d\mu\right|
\le \alpha \ell^{d+1}\|\Psi\f\|\ci{\Lip}
\\
\le
\alpha\ell^{d+1}\left[\|\Psi\|\ci{\Lip(\supp\f)}\|\f\|\ci{L^\infty}+\|\Psi\|\ci{L^\infty(\supp\f)}\|\f\|\ci{\Lip}\right]\,.
\end{multline*}

To get rid of the $L^\infty$ norms, recall that $\f$ is supported on a ball of radius $5\ell$.
Thus $\|\f\|\ci{L^\infty}\le 5\ell\|\f\|\ci{\Lip}$ (within the distance $5\ell$ from any point
$x\in\R^{d+1}$, we can find a point where $\f$ vanishes). 

Since $\int\Psi\f\,dm\ci L=0$
and the diameter
of $\supp \f$ does not exceed $10\ell$, we have $\|\Psi\|\ci{L^\infty(\supp\f)}\le 10\ell\|\Psi\|\ci{\Lip(\supp\f)}$.
Plugging these bounds in, we obtain the first estimate.

The second estimate immediately follows from the first one and the elementary inequality
$$
\|\,|\Psi|^p \|\ci{\Lip(\supp\f)}\le p\|\Psi\|\ci{L^\infty(\supp\f)}^{p-1}\|\Psi\|\ci{\Lip(\supp\f)}\,.
$$
Lemma \ref{mutonulip} is thus fully proved.

We now turn to the proof of Lemma \ref{mutonuRH}. First of all, we need to ensure that the
integral on the left can be understood in some reasonable sense at all. To this end,
split it as $\int\Psi\f [R^H(\f\mu)]\,d\mu-\int\Psi\f [R^H\nu]\,d\mu$. Since  $R^H\nu = R^H (\f m\ci L)$  and $\f \in C^2$ and is compactly supported, $R^H\nu$  
 is well defined and can be viewed  as a Lipschitz function
on the entire space $\R^{d+1}$ by Lemma \ref{smoothtolip}. Thus, integrating it against
a compactly supported finite measure $\Psi\f\mu$ presents no problem. However, if $\mu$ is merely
nice, the first integral may fail to exist as an integral of a pointwise defined function. Still,
by the discussion in the Weak limits section (Section \ref{weaklimits}), we can define
it at least as the bilinear form $\langle R^H_\mu \f, \Psi\f \rangle_\mu=I(\f,\Psi\f)$ because both $\f$ and $\Psi\f$
are compactly supported Lipschitz functions in the entire space $\R^{d+1}$, and this definition
agrees with any reasonable stronger definition whenever the latter makes sense too.

To show that the estimate holds, fix $\delta>0$ to be chosen later and split $R^H=R^H_{\delta\ell}+[R^H-R^H_{\delta\ell}]$.
Note now that the kernel $K^H_{\delta\ell}$ is Lipschitz on the entire space and satisfies the estimate
$
\|K^H_{\delta\ell}\|\ci{\Lip}\le \delta^{-(d+1)}\ell^{-(d+1)}\,.
$
Thus
\begin{multline*}
\|R^H_{\delta\ell}(\Psi\f\mu)\|\ci{\Lip}\le \|K^H_{\delta\ell}\|\ci{\Lip}\|\Psi\f\|\ci{L^1(\mu)}
\\
\le C\delta^{-(d+1)}\ell^{-(d+1)}\|\Psi\|\ci{L^\infty(\supp\f)}\|\f\|\ci {L^\infty}\mu(B(z,5\ell))
\\
\le C\delta^{-(d+1)}\ell^{-1}\|\Psi\|\ci{L^\infty(\supp\f)}\|\f\|\ci {L^\infty}
\le C\delta^{-(d+1)}\|\Psi\|\ci{L^\infty(\supp\f)}\|\f\|\ci {\Lip}
\,.
\end{multline*}
Note that the niceness of $\mu$ was used here to bound $\mu(B(z,5\ell))$ by $C\ell^d$.

Now using the antisymmetry and the (already proved) Lemma \ref{mutonulip}, we get
\begin{multline*}
\left|\int\Psi\f [R^H\ci{\delta\ell}(\f\mu-\nu)]\,d\mu\right|
=
\left|-\int R^H\ci{\delta\ell}(\Psi\f\mu)\,d(\f\mu-\nu)\right|
\\
\le
C\alpha\ell^{d+2}\|R^H_{\delta\ell}(\Psi\f\mu)\|\ci{\Lip}\|\f\|\ci{\Lip}
\le
C\alpha\delta^{-(d+1)}\ell^{d+2}\|\Psi\|\ci{L^\infty(\supp\f)}\|\f\|^2\ci{\Lip}\,.
\end{multline*}
Next observe that (again, by Lemma \ref{smoothtolip})
$
(R^H-R^H_{\delta\ell})\nu
$
is the uniform limit of $(R^H_\Delta-R^H_{\delta\ell})\nu$ as $\Delta\to 0+$. The kernel $K^H_{\Delta}-K^H_{\delta\ell}$
is a continuous function dominated by $|x|^{-d}$ and supported on the ball $\bar{B}(0,\delta\ell)$ for every $\Delta\in(0,\delta\ell)$.
Moreover, the cancellation property
$$
\int_L[K^H_{\Delta}-K^H_{\delta\ell}](x-y)\,dm\ci L(y)=0
$$
holds
for all $x\in\R^{d+1}$. Thus, for $0<\Delta<\delta\ell$, we can write
$$
\left|[(R^H_\Delta-R^H_{\delta\ell})\nu](x)\right|\le a\int_{y:|y-x|<\delta\ell}\frac{|\f(x)-\f(y)|}{|x-y|^d}\,dm\ci L(y)
\le C\delta\ell\|\f\|\ci{\Lip}\,.
$$
Passing to the limit as $\Delta\to 0+$, we conclude that the same estimate holds for $(R^H-R^H_{\delta\ell})\nu$, so
\begin{multline*}
\left|\int\Psi\f [(R^H-R^H\ci{\delta\ell})\nu]\,d\mu\right|
\le
\|[R^H-R^H\ci{\delta\ell}]\nu\|\ci{L^\infty}\|\Psi\f\|\ci{L^1(\mu)}
\\
\le
C\delta\ell\|\f\|\ci{\Lip}\|\Psi\|\ci{L^\infty(\supp\f)}\|\f\|\ci{L^\infty}\mu(B(z,5\ell))
\\
\le
C\delta\ell^{d+2}\|\Psi\|\ci{L^\infty(\supp\f)}\|\f\|^2\ci{\Lip}\,.
\end{multline*}
Finally, to deal with the integral
$
\int \Psi\f[(R^H-R^H_{\delta\ell})(\f\mu)]\,d\mu\,,
$
we will use the same trick as in the Weak limits section and use the antisymmetry to
interpret it as the absolutely convergent integral
$$
\frac 12\iint (K^H-K^H_{\delta\ell})(x-y)(\Psi(x)-\Psi(y))\f(x)\f(y)\,d\mu(x)\,d\mu(y)\,.
$$
Since the domain of integration here can be trivially reduced to $\supp\f\times\supp\f$ and
since $|(K^H-K^H_{\delta\ell})(x-y)|\le|x-y|^{-d}\chi\ci{B(0,\delta\ell)}(x-y)$, we get
\begin{multline*}
\left|\int \Psi\f [(R^H-R^H_{\delta\ell})(\f\mu)]\,d\mu\right|
\\
\le
\frac 12\|\Psi\|\ci{\Lip(\supp\f)}\|\f\|^2\ci{L^\infty}
\iint_{x,y\in\supp\f,|x-y|<\delta\ell}\frac{d\mu(x)\,d\mu(y)}{|x-y|^{d-1}}
\\
\le C\delta\ell^{d+3}\|\Psi\|\ci{\Lip(\supp\f)}\|\f\|^2\ci{\Lip}\,.
\end{multline*}
Bringing these three estimates together, we finally conclude that
\begin{multline*}
\left|\int \Psi\f R^H(\f\mu-\nu)\,d\mu\right|
\\
\le
C(\alpha\delta^{-(d+1)}+\delta)\ell^{d+2}
\left[\|\Psi\|\ci{L^\infty(\supp\f)}+\ell\|\Psi\|\ci{\Lip(\supp\f)}\right]\|\f\|\ci{\Lip}^2\,.
\end{multline*}
To get the estimate of Lemma \ref{mutonuRH}, it just remains to choose $\delta=\alpha^{\frac{1}{d+2}}$\,.
\end{proof}

\section{Tangent measures and geometric flattening}
\label{tangentmeasures}

Fix some continuous function $\psi_0:[0,+\infty)\to[0,1]$ such that $\psi_0=1$ on $[0,1]$ and 
$\psi_0=0$ on $[2,+\infty)$. For $z\in\R^{d+1}$, $0<r<R$, define 
$$
\psi\ci {z,r,R}(x)=\psi_0\left(\frac{|x-z|}{R}\right)-\psi_0\left(\frac{|x-z|}{r}\right)\,.
$$
The goal of this section is to prove the following 

\begin{lem}
\label{geometricflattening}
Fix five positive parameters $A,\alpha,\beta,\wt c, \wt C>0$. There exists $\rho>0$ depending only on
these parameters and the dimension $d$ such that the following implication holds.

Suppose that $\mu$ is a $\wt C$-good measure on a ball $B(x,R)$ centered at a point $x\in\supp\mu$
that is AD regular in $B(x,R)$ with lower regularity constant $\wt c$. Suppose also that
$$
|[R(\psi\ci{z,\delta R,\Delta R}\mu)](z)|\le \beta
$$
for all $\rho<\delta<\Delta<\frac 12$ and all $z\in B(x,(1-2\Delta)R)$ such that $\dist(z,\supp\mu)<\frac\delta 4 R$.

Then there exist a scale $\ell>\rho R$, a point $z\in B(x,R-(A+\alpha)\ell)$, and a linear hyperplane $H$
such that $\mu$ is geometrically $(H,A,\alpha)$-flat at $z$ on the scale $\ell$.
\end{lem}

\begin{proof}
Replacing $\mu$ by $R^{-d}\mu(x+R\cdot)$ if necessary, we may assume without loss of generality
that $x=0, R=1$. We will start with showing that the absence of geometric flatness and the boundedness 
of $[R(\psi\ci{z,\delta,\Delta}\mu)](z)$ are inherited by weak limits. More precisely,
let $\nu_k$ be a sequence of $\wt C$-good measures on $B(0,1)$ and $AD$-regular there
with lower regularity constant $\wt c$. Assume that $\nu$ is another measure on $B(0,1)$ and 
$\nu_k\to\nu$ weakly in $B(0,1)$ (i.e., $\int F\,d\nu_k\to \int F\,d\nu$ for every continuous
function $F$ with $\supp F\subset B(0,1)$). 
We have seen in Section
\ref{weaklimits} that then $\nu$ is also $\wt C$-good and  
AD regular in $B(0,1)$ with the same lower regularity constant $\wt c$. Our first task will 
be to prove the following

\begin{claim}\ 
\par
\begin{itemize}
\item
If for some $A'>A$ and $0<\alpha'<\alpha$, the measure
$\nu$ is geometrically $(H,A',\alpha')$-flat on the scale $\ell>0$ 
at some point $z\in B(0,1-(A'+\alpha)\ell)$,
then for all sufficiently large $k$, the measure $\nu_k$ is
geometrically $(H,A,\alpha)$-flat at $z$ on the scale $\ell$.
\item
If for some $0<\delta<\Delta<\frac 12$ and some $z\in B(0,1-2\Delta)$ with
$\dist(z,\supp\nu)<\frac\delta 4$, we have 
$
|[R(\psi\ci{z,\delta ,\Delta }\nu)](z)|> \beta
$,
then  for all sufficiently large $k$, we also have $\dist(z,\supp\nu_k)<\frac\delta 4$ and
$
|[R(\psi\ci{z,\delta ,\Delta }\nu_k)](z)|>\beta
$.
\end{itemize}
\end{claim}

\begin{proof}
The reason is, of course, that we can check both conditions in question by looking at integrals 
of finitely many continuous functions. It is completely obvious for the second claim because
$$
[R(\psi\ci{z,\delta ,\Delta}\nu)](z)=\int F\,d\nu=\lim_{k\to\infty}\int F\,d\nu_k
=\lim_{k\to\infty}[R(\psi\ci{z,\delta ,\Delta }\nu_k)](z)
$$
where $F(x)=K(z-x)\psi\ci{z,\delta,\Delta}(x)$. Note that $F$ is compactly supported in $B(0,1)$ 
and continuous
because $\psi\ci{z,\delta,\Delta}(x)=0$ when $|x-z|<\delta$ or $|x-z|>2\Delta$. To ensure
that $\dist(z,\supp\nu_k)<\frac\delta 4$, take $F(x)=\max(\frac\delta 4-|x-z|,0)$. Then
$\int F\,d\nu>0$, so $\int F\,d\nu_k>0$ for all sufficiently large $k$, but the latter is
possible only if $B(z,\frac\delta 4)\cap\supp\nu_k\ne\varnothing$. 

Expressing the geometric flatness condition in terms of integrals of continuous functions
is only slightly more difficult. To test that $B(z,A\ell)\cap\supp\nu_k$ is contained in the 
strip of width $\alpha\ell$ around the affine hyperplane $L$ containing $z$ and parallel to
$H$, consider any continuous function $F:\R^{d+1}\to[0,1]$ such that $F(x)=0$ whenever
$|x-z|\ge A'\ell$ or $\dist(x,L)\le\alpha'\ell$ and $F(x)=1$ whenever
$|x-z|\le \frac{A+A'}2 \ell$ and $\dist(x,L)\ge \frac{\alpha+\alpha'}2\ell$. Note that
$\supp F\subset B(0,1)$ and $\int F\,d\nu=0$. Thus $\int F\,d\nu_k<\wt c(\e\ell)^d$ for 
all sufficiently large $k$ where $\e=\frac 12\min(A'-A,\alpha-\alpha')$. However,
for every $x\in B(z,A\ell)$ such that $\dist(x,L)\ge \alpha\ell$, we have $F=1$
on the ball $B(x,\e\ell)$. On the other hand, if any such $x$ were contained in $\supp\nu_k$, we would
have $\int F\,d\nu_k\ge\nu_k(B(x,\e\ell))\ge \wt c(\e\ell)^d$ by the AD-regularity
of $\nu$.

At last, to check that every point of $L\cap B(z,A\ell)$ lies within distance $\alpha\ell$ 
from $\supp\nu_k$, take any finite $\frac {\alpha-\alpha'}2\ell $-net $Y$ in $L\cap B(z,A\ell)$
and for every $y\in Y$ choose any continuous function $F_y(x)$ that vanishes for 
$|x-y|\ge \frac {\alpha+\alpha'}2\ell$ and is strictly positive for 
$|x-y|<\frac {\alpha+\alpha'}2\ell$. Then $\int F_y\,d\nu>0$ for all $y\in Y$ and,
thereby, for all sufficiently large $k$, all the 
integrals $\int F_y\,d\nu_k$ are positive as well. Take any $x\in L\cap B(z,A\ell)$.
Choose $y\in Y$ so that $|x-y|<\frac {\alpha-\alpha'}2\ell$. Since $\int F_y\,d\nu_k>0$,
there exists $x'\in \supp\nu_k$ such that $|x'-y|<\frac {\alpha+\alpha'}2\ell$. But 
then $|x-x'|<\alpha\ell$.   
\end{proof}

Our next aim is to prove the following
\begin{alternative} 
If $\nu$ is any good measure on 
$B(0,1)$ that is AD regular there, then either for every $A,\alpha>0$ there exist a scale $\ell>0$,
a point $z\in B(0,1-(A+\alpha)\ell)$ and a linear hyperplane $H$ such that $\nu$ is
geometrically $(H,A,\alpha)$-flat at $z$ on the scale $\ell$, or
$$
\sup_{\substack{0<\delta<\Delta<\frac 12\\ z\in B(0,1-2\Delta),\dist(z,\supp\nu)<\frac\delta 4}}
|[R(\psi\ci{z,\delta ,\Delta }\nu)](z)|=+\infty\,.
$$
\end{alternative}

\begin{proof} 
We will employ the technique of tangent measures developed by Preiss in \cite{P}.
\begin{udef}
Let $\nu$ be any finite measure on $B(0,1)$. Let $z\in B(0,1)$. The measure
$\nu_{z,\la}(E)=\la^{-d}\mu(z+\la E)$ ($E\subset B(0,1)$), which is well-defined
as a measure on $B(0,1)$ whenever $\la\le 1-|z|$, is called a $\la$-blow-up of $\nu$ 
at $z$. A tangent measure of $\nu$ at $z$ is just any measure on $B(0,1)$ that can
be obtained as a weak limit in $B(0,1)$ of a sequence of $\la$-blow-ups of $\nu$ at $z$ with
$\la\to 0+$. 
\end{udef}
Note that if $\nu$ is $C$-good and AD regular in $B(0,1)$ with lower regularity constant
$c$, then so are all blow-ups of $\nu$ and all tangent measures of $\nu$. Note also that
in this case, if $z\in\supp\nu$, then all blow-ups and tangent measures of $\nu$ at $z$ 
have the origin in their supports. At last, the observations above imply that the (quantitative) negation
of either condition in the alternative we are currently trying to establish for $\nu$ is inherited
by all tangent measures of $\nu$ (because it is, clearly, inherited by all blow-ups 
by simple rescaling and we have just shown that we can pass to weak limits here).

Now assume that a good AD regular in $B(0,1)$ measure $\nu$ containing the origin
in its support satisfies neither of the conditions in the alternative.
Since $\nu$ is finite and AD regular in $B(0,1)$, its support is nowhere dense in $B(0,1)$. 
Take any point $z'\in B(0,\frac 12)\setminus\supp\nu$.
Let $z$ be a closest point to $z'$ in $\supp\nu$. Note that since $0\in\supp\nu$, we 
have $|z-z'|\le |z'|$, so $|z|\le 2|z'|<1$. Also, the ball $B=B(z',|z-z'|)$ doesn't
contain any point of $\supp\nu$. Let $n$ be the outer unit normal to $\partial B$ at $z$.
Consider the blow-ups $\nu_{z,\la}$ of $\nu$ at $z$. As $\la\to 0$, the supports
of $\nu_{z,\la}$ lie in a smaller and smaller neighborhood of the half-space
$S=\{x\in\R^{d+1}:\langle x,n\rangle\ge 0\}$ bounded by the linear hyperplane
$H=\{x\in\R^{d+1}:\langle x,n\rangle=0\}$. So, every tangent measure of $\nu$ at $z$
must have its support in $S$. On the other hand, such tangent measures do exist 
because the masses of $\nu_{z,\la}$ are uniformly bounded. At last, the origin
is still in the support of every tangent measure of $\nu$ at $z$. 
Thus, starting with any measure $\nu$ that gives a 
counterexample to the alternative we are trying to prove, we can modify it so
that it is supported on a half-space. So, we may assume without loss of generality
that $\nu$ was supported on such a half-space $S$ from the very beginning.

Now fix $\Delta<\frac 12$ and note that under this assumption,
$$
-\langle [R(\psi\ci{0,\delta ,\Delta}\nu)](0),n\rangle>\int_{B(0,\Delta)\setminus B(0,2\delta)}
\frac{\langle x,n\rangle}{|x|^{d+1}}\,d\nu(x)\,.
$$
Since the quantity on the left should stay bounded as $\delta\to 0$, we conclude that
$$
\int_{B(0,\Delta)}
\frac{\langle x,n\rangle}{|x|^{d+1}}\,d\nu(x)<+\infty
$$
and, thereby, 
$$
\int_{B(0,\lambda)}
\frac{\langle x,n\rangle}{|x|^{d+1}}\,d\nu(x)\to 0 \text{ as }\la\to 0\,.
$$
Let now $F(x)=\langle x,n\rangle(1-2|x|)$ for $|x|\le \frac 12$ and $\langle x,n\rangle\ge 0$, and $F(x)=0$ otherwise.
Then $F$ is a continuous function supported inside $B(0,1)$ and 
$$
\int F\,d\nu_{0,\la}=\la^{-d}\int F(x/\la)\,d\nu \le\int_{B(0,\lambda)} 
\frac{\langle x,n\rangle}{|x|^{d+1}}\,d\nu(x)\,,
$$
so the integral of $F$ with respect to any tangent measure of $\nu$ at $0$ must vanish.
Since those tangent measures are still supported on $S$, this is possible only if they 
vanish on $B(0,\frac 12)\setminus H$. Taking a $\frac 12$-blow up of any such tangent
measure at $0$, we see that we can just as well assume that our counterexample $\nu$ is supported
on $H$. 

If we had $H\cap B(0,\frac 12)\subset\supp\nu$, then for any $A,\alpha>0$, $\nu$ would be geometrically
$(H,A,\alpha)$-flat at the origin on the scale $\ell=\frac 1{2(A+\alpha)}$, which contradicts the assumption
that the first part of the alternative doesn't hold for $\nu$.

Thus, we can find $z'\in (B(0,\frac 12)\cap H)\setminus \supp\nu$. Again, let $z$ be the closest point to $z'$  of $\supp\nu$, and let $n'$ be the outer unit normal to the boundary of the ball $B(z',|z-z'|)$ at $z$. 
Note that $n'\in H$.
Now repeat all the above steps with this new choice of $z$. The condition
$\supp\nu\subset H$ will be preserved at each step but by the end of the whole process we will also
restrict the support of $\nu$ to some other linear hyperplane $H'$ with the unit normal $n'$. Since 
$n'$ is perpendicular to $n$, the support of $\nu$ is now restricted to the $(d-1)$-dimensional linear plane
$H\cap H'$. However a $(d-1)$-dimensional linear plane cannot carry any non-zero measure $\nu$ satisfying
the growth bound $\nu(B(x,r))\le Cr^d$. This contradiction finishes the proof of the alternative.
\end{proof}

Now we are ready to prove the Lemma \ref{geometricflattening} itself. Suppose that such $\rho$ does not exist. Then for each $\rho>0$, we 
can find a $\wt C$-good measure $\mu_\rho$ on a ball $B(0,1)$  
that is AD regular in $B(0,1)$ with lower regularity constant $\wt c$ and which satisfies $0\in\supp\mu_\rho$ and
$$
|[R(\psi\ci{z,\delta ,\Delta }\mu_\rho)](z)|\le \beta
$$
for all $\rho<\delta<\Delta<\frac 12$ and all $z\in B(x,1-2\Delta)$ with $\dist(z,\supp\mu_\rho)<\frac\delta 4$,
but is not geometrically $(H,A,\alpha)$-flat at $z$ on any scale $\ell>\rho$
at any point $z\in B(x,1-(A+\alpha)\ell)$ for any  linear hyperplane $H$.

Then we can find a sequence $\rho_k\to 0$ so that the measures $\mu_{\rho_k}$ converge weakly 
to some limit measure $\nu$ in $B(0,1)$. This limit measure would satisfy
$$
|[R(\psi\ci{z,\delta ,\Delta }\nu)](z)|\le \beta
$$
for all $0<\delta<\Delta<\frac 12$ and all $z\in B(x,1-2\Delta)$ with $\dist(z,\supp\nu)<\frac\delta 4$ but would not 
be geometrically $(H,A,\alpha)$-flat  on any scale $\ell>0$
at any point $z\in B(x,1-(A+\alpha)\ell)$ for any linear hyperplane $H$.
But this combination of properties clearly contradicts the alternative we have just proved.
\end{proof}

\section{The flattening lemma}
\label{flatteninglemmaS}

The goal of this section is to present a lemma that will allow us to carry out one of the major
steps in our argument: the transition from the absence of large oscillation of $R^H\mu$ on
$\supp\mu$ near some fixed point $z$ on scales comparable to $\ell$ to the flatness of $\mu$
at $z$ on scale $\ell$.
\begin{prop}
\label{flatteninglemma}
Fix four positive parameters $A,\alpha,\wt c,\wt C$. There exist numbers $A',\alpha'>0$ depending
only on these fixed parameters and the dimension $d$ such that the following implication
holds.

Suppose that $H$ is a linear hyperplane in $\R^{d+1}$, $z\in\R^{d+1}$, $L$ is the
affine hyperplane containing $z$ and parallel to $H$,
$\ell>0$, and $\mu$ is a $\wt C$-good finite measure in $\R^{d+1}$
that is AD regular in $B(z,5A'\ell)$ with the lower regularity constant $\wt c$. Assume that
$\mu$ is geometrically $(H,5A',\alpha')$-flat at $z$ on the scale $\ell$
and, in addition, for every (vector-valued) Lipschitz function $g$ with
$\supp g\subset B(z,5A'\ell)$, $\|g\|\ci{\Lip}\le \ell^{-1}$, and $\int g\,d\mu=0$, one has
$$
|\langle R^H_\mu 1,g\rangle_\mu|\le \alpha'\ell^d\,.
$$
Then $\mu$ is $(H,A,\alpha)$-flat at $z$ on the scale $\ell$.
\end{prop}

Before proving this proposition (which we will call the ``Flattening Lemma'' from now on), let us discuss the
meaning of the assumptions. In what follows, we will apply
this result to restrictions of a fixed good AD regular
measure $\mu$ to open balls at various scales and locations. The restriction of a good AD regular measure
to a ball may easily fail to be AD regular in the entire space $\R^{d+1}$, which explains why we have
introduced the local notion of Ahlfors-David regularity. Every restriction of a good measure to any set is,
of course, good with the same goodness constant as the original measure.

The first step in proving the rectifiability of a measure is showing that its support is almost planar
on many scales in the sense of the geometric $(H,5A',\alpha')$-flatness in the
assumptions of the Flattening Lemma implication. This step is not that hard and we will carry it out in Section 
\ref{abundanceofflats}. The second condition involving the Riesz transform  means, roughly speaking, that $R^H_\mu 1$
is almost constant on $\supp\mu\cap B(z,A'\ell)$ in the sense that its ``wavelet coefficients'' near $z$ on the
scale $\ell$ are small. There is no canonical smooth wavelet system in $L^2(\mu)$ when $\mu$ is an arbitrary
measure but mean zero Lipschitz functions serve as a reasonable substitute. The boundedness of $R^H_\mu$ in $L^2(\mu)$
implies that $R^H_\mu 1\in L^2(\mu)$ (because for finite measures $\mu$, we have $1\in L^2(\mu)$), so
an appropriate version of the Bessel inequality can be used to show that
large wavelet
coefficients have to be rare
and the balls satisfying the second assumption of the implication
should also be viewed as typical.

Finally, it is worth mentioning that the full $(H,A,\alpha)$-flatness condition is much stronger than just the
geometric one in the sense
that it allows one to get non-trivial {\em quantitative} information about the Riesz transform operator $R^H_\mu$.
The Flattening Lemma thus provides the missing link between the purely geometric conditions like those in the
David-Semmes monograph and analytic conditions needed to make explicit estimates.

\begin{proof}
Note that the geometric $(H,A,\alpha)$-flatness of $\mu$ is ensured by the geometric
$(H,5A',\alpha')$-flatness assumption of the
Flattening Lemma implication as soon as $A'>A$ and $\alpha'\le\alpha$. The real
problem is to prove the analytic part of the flatness condition.

To this end, note first that the setup of the Flattening Lemma is invariant under translations
and dilations, so, replacing the measure $\mu$ and the test-functions $f$ and $g$ by
$\ell^{-d}\mu(z+\ell\cdot)$, $f(z+\ell\cdot)$, and $g(z+\ell\cdot)$ respectively, we can
always assume without loss of generality that $z=0$ and $\ell=1$.

Now fix $A'>A$. Since the set $\mathcal L$ of all Lipschitz functions $f$ with Lipschitz constant $1$ supported
on $B(0,A)$ and having zero integral with respect to $m\ci L$ is pre-compact in $C_0(\R^{d+1})$,
for every $\beta>0$, we can find a finite family $\mathcal F$ in $\mathcal L$ so that every function $f\in\mathcal L$
is uniformly $\beta$-close to some $f'\in\mathcal F$.
Since we have the a priori bound $\mu(\bar B(0,A))\le \wt CA^d$, this $\beta$-closeness implies that
$$
\left|\int f\,d\mu\right|\le \left|\int f'\,d\mu\right|+\wt CA^d\beta\,,
$$
so choosing $\beta<\frac 12\wt C^{-1}A^{-d}\alpha$, we see that in the proof of the $(H,A,\alpha)$ flatness,
we can consider only test functions $f\in \mathcal F$ if we don't mind showing for them a stronger inequality
with $\alpha$ replaced by $\frac\alpha 2$. Since $\mathcal F$ is finite, we see that if the Flattening Lemma
is false, we can find one fixed test function $f$ and a sequence of measures $\mu_k$ satisfying the assumptions
of the Flattening Lemma implication with our fixed $A'$ and $\alpha'=\frac 1k$ such that
$\int f\,d\mu_k\ge\frac{\alpha}{2}$ for all $k$.

Split each $\mu_k$ as
$$
\mu_k=\chi\ci{B(0,5A')}\mu_k+\chi\ci{\R^{d+1}\setminus B(0,5A')}\mu_k=\nu_k+\eta_k\,.
$$
Note that $\nu_k$ are still $\wt C$ good and AD regular in $B(0,5A')$ with lower AD-regularity constant
$\wt c$. Moreover, $\supp\nu_k$ lies within distance $\frac 1k$ from $L$ and every point in
$L\cap B(0,5A'-\frac 1k)$ lies within distance $\frac 1k$ from $\supp\nu_k$. Passing to a subsequence,
if necessary, we can assume that $\nu_k$ converge weakly to some measure $\nu$. By the results of the
Weak limits section (Section \ref{weaklimits}) this limiting measure $\nu$ is $\wt C$-good and, obviously,
$\supp\nu\subset L\cap \bar B(0,5A')$.

Fix a point $w\in L\cap B(0,5A')$ and $r>0$ such that $B(w,r)\subset B(0,5A')$. Take any $r'<r$ and consider
a continuous function $h:\R^{d+1}\to [0,1]$ that is identically $1$ on $B(w,r')$ and identically $0$
outside $B(w,r)$. Since $w\in B(0,5A'-\frac 1k)$
for all sufficiently large $k$, we can find a sequence of points $w_k\in\supp\nu_k$ so that $|w-w_k|\le\frac 1k$
for all sufficiently large $k$. Note, however, that $B(w_k,r'-\frac 1k)\subset B(w,r')$, so
for all large $k$, we have
$$
\int h\,d\nu_k\ge \nu_k(B(w,r'))\ge \nu_k(B(w_k,r'-\tfrac 1k))\ge \wt c(r'-\tfrac 1k)^d\,.
$$
Passing to the limit, we conclude that $\nu(B(w,r))\ge \int h\,d\nu\ge \wt c(r')^d$. Since this inequality
holds for all $r'<r$, we must have $\nu(B(w,r))\ge \wt cr^d$. Combining this with the upper bound
$\nu(B(w,r))\le \wt Cr^d$ and the inclusion $\supp\nu\subset L$, we see that, by the Radon-Nikodym theorem applied
to $\nu$ and $m\ci L$, the limiting measure $\nu$ can be written as
$
\nu=p\,m\ci L
$
for some Borel function $p$ on $L$ satisfying $\omega_d^{-1}\wt c\le p\le \omega_d^{-1}\wt C$ almost everywhere
with respect to $m\ci L$
on $L\cap B(0,5A')$, where $\omega_d$ is the $d$-dimensional volume
of the unit ball in $\R^d$.

Fix some non-negative Lipschitz function $h$ with $\supp h\subset B(0,4A')$ and $\int h\,d\nu>0$. Take any Lipschitz
vector-valued function $g$ supported on $B(0,4A')$ and satisfying $\|g\|\ci{\Lip}<1$, $\int g\,d\nu=0$. Since
$\int h\,d\nu_k\to \int h\,d\nu>0$ as $k\to\infty$, the integrals $\int h\,d\nu_k$ stay bounded away from $0$ for
sufficiently large $k$.

Put
$$
a_k=\left(\int h\,d\nu_k\right)^{-1}\int g\,d\nu_k\,,\qquad g_k=g-a_kh\,.
$$
The functions $g_k$ are well-defined for all large enough $k$ and satisfy
$$
\int g_k\,d\mu_k=\int g_k\,d\nu_k=0,\quad \supp g_k\subset B(0,4A')\,.
$$
Since $\int g\,d\nu_k\to\int g\,d\nu=0$, we conclude that $a_k\to 0$ as $k\to+\infty$, so
$$
\|g_k\|\ci{\Lip}<1
$$
for large enough $k$.

 Since $\mu_k$ satisfies the assumptions of the Flattening Lemma implication,
we must have
$$
\left|\langle R^H_{\mu_k}1,g_k\rangle_{\mu_k}\right|\le\frac 1k
$$
for large $k$.
Taking into account that $\supp g_k\subset B(0,4A')$, we can rewrite this as
$$
\left|\langle R^H_{\nu_k}1,g_k\rangle_{\nu_k}+\langle R^H \eta_k,g_k\rangle_{\nu_k}\right|\le\frac 1k\,.
$$
Note that
$$
\langle R^H_{\nu_k}1,g_k\rangle_{\nu_k}=
\langle R^H_{\nu_k}1,g\rangle_{\nu_k}-\langle R^H_{\nu_k}1,a_k h\rangle_{\nu_k}
$$
and that $R^H_{\nu_k}1$ and $R^H_{\nu}1$ coincide with $R^H_{\nu_k}\f$ and $R^H_{\nu}\f$ respectively
for any
compactly supported Lipschitz
function $\f$ that is identically $1$ on $B(0,5A')$, say.
Thus, by the results of the Weak limits section (Section \ref{weaklimits}), we get
$$
\langle R^H_{\nu_k}1,g\rangle_{\nu_k}=\langle R^H_{\nu_k}\f,g\rangle_{\nu_k}\to \langle R^H_{\nu}\f,g\rangle_{\nu}
=\langle R^H_{\nu}1,g\rangle_{\nu}\,.
$$
Similarly,
$$
\langle R^H_{\nu_k}1,he\rangle_{\nu_k}\to \langle R^H_{\nu}1,he\rangle_{\nu}
$$
for every vector $e\in H$. Since
$$
\langle R^H_{\nu_k}1,a_k h\rangle_{\nu_k}=\sum_j\langle a_k,e_j\rangle
\langle R^H_{\nu_k}1,he_j\rangle_{\nu_k}
$$
for every orthonormal basis $e_1,\dots,e_d$ in $H$ and $a_k\to 0$ as $k\to +\infty$, we
conclude that
$$
\langle R^H_{\nu_k}1,a_k h\rangle_{\nu_k}\to 0
$$
and, thereby,
$$
\langle R^H_{\nu_k}1,g_k\rangle_{\nu_k}\to \langle R^H_{\nu}1,g\rangle_{\nu}
$$
as $k\to+\infty$.

Note now that the measure $\eta_k$ is supported outside $B(0,5A')$. Together with the cancellation
property $\int g_k\,d\nu_k=0$, this yields
$$
\langle R^H \eta_k,g_k\rangle_{\nu_k}=\langle v_k,g_k\rangle_{\nu_k}
$$
where
$$
v_k=R^H \eta_k-(R^H \eta_k)(0)
$$
is a $C^\infty$ function in $B(0,4A')$ satisfying $v_k(0)=0$ and
$$
|(\nabla^j v_k)(x)|\le C\int\frac{d\eta_k(y)}{|x-y|^{d+j}}\le C(j)\wt C /(A')^{j}
$$
whenever $x\in B(0,4A')$ and $j>0$.

Since the set of functions vanishing at the origin  with $3$ uniformly bounded derivatives
is compact in $C^2(B(0,4A'))$, we may (passing to a subsequence again, if necessary)
assume that $v_k\to v$ in $C^2(B(0,4A'))$, which is more than enough to conclude that
$$
\langle v_k,g_k\rangle_{\nu_k}\to \langle v,g\rangle_{\nu}
$$
(all we need for the latter is the uniform convergence $\langle v_k,g_k\rangle \to \langle v,g\rangle$). Thus, we found a $C^2$-function
$v$ in $B(0,4A')$ such that
$$
\langle R^H_\nu 1,g\rangle_{\nu}=-\langle v,g\rangle_{\nu}
$$
for all Lipschitz functions $g$ with $\supp g\subset B(0,4A')$ and $\int g\,d\nu=0$.
Moreover,
$$
|\nabla^j v|\ci{L^\infty(B(0,4A'))}\le C\wt C /(A')^{j}
$$
for $j=1,2$.
The condition $\|g\|\ci{\Lip}\le 1$ can be dropped now because both sides are linear
in $g$. This equality can be rewritten as
$$
\langle R^H_{m\cci L} p,pg\rangle_{m\cci L}=-\langle v,pg\rangle_{m\cci L}
$$
for all Lipschitz functions $g$ with $\supp g\subset B(0,4A')$ satisfying $\int_L pg\,dm\ci L=0$.
Since $p$ is bounded from below on $L\cap B(0,4A')$, the set of such products $pg$ is
dense in the space of all mean zero functions in
$L^2(L\cap B(0,4A'),m\ci L)$ and we conclude that
$R^H_{m\cci L}p$ differs from $-v$ only by a constant on $L\cap B(0,4A')$. By the toy flattening lemma
(Lemma \ref{toyflat}) applied with $A'$ instead of $A$, this means that $p$ is Lipschitz in
$L\cap B(0,A')$ and
$$
\|p\|\ci{\Lip(L\cap B(0,A'))}\le C\wt C/A'\,.
$$
But then
$$
\left|\int f\,d\nu\right|=\left|\int f(p-p(0))\,dm\ci L\right|\le C\wt CA^{d+1}/A'<\frac\alpha 2
$$
if $A'$ was chosen large enough. On the other hand, we have
$$
\left|\int f\,d\nu_k\right|=\left|\int f\,d\mu_k\right|\ge \frac{\alpha}{2}
$$
for all $k$ and
$$
\int f\,d\nu_k\to \int f\,d\nu\quad\text{ as }k\to+\infty\,.
$$
This contradiction finishes the proof.
\end{proof}

\section{David-Semmes lattice}
\label{DSlattice}

Let $\mu$ be a $d$-dimensional AD regular measure in $\R^{d+1}$. Let $E=\supp\mu$. 

\goal 
The goal of this section
is to construct a family $\D$ of sets $Q\subset\R^{d+1}$ with the following properties:
\begin{itemize}
\item
The family $\D$ is the disjoint union of families $\D_k$ (families of level $k$ cells), $k\in\mathbb Z$. 
\item
If $Q',Q''\in\D_k$, then either $Q'=Q''$ or $Q'\cap Q''=\varnothing$.
\item
Each $Q'\in\D_{k+1}$ is contained in some $Q\in\D_k$ (necessarily unique due to the previous
property).
\item
The cells of each level cover $E$, i.e., $\cup_{Q\in\D_k}Q\supset E$ for every $k$.
\item
For each $Q\in\D_k$, there exists $z\ci Q\in Q\cap E$ (the ``center'' of $Q$) such that 
$$
B(z\ci Q,2^{-4k-3})\subset Q\subset B(z\ci Q,2^{-4k+2})\,.
$$
\item
For each $Q\in\D_k$ and every $\e>0$, we have
$$
\mu\{x\in Q:\operatorname{dist}(x,\R^{d+1}\setminus Q)<\e 2^{-4k}\}\le C\e^\gamma \mu(Q)
$$
where $C,\gamma>0$ depend on $d$ and the constants in the AD-regularity property of $\mu$ only.
\end{itemize} 
\goalend

Since all cells in $\D_k$ have approximately the same size $2^{-4k}$, it will be convenient to introduce the 
notation $\ell(Q)=2^{-4k}$ where $k$ is the unique index for which $Q\in\D_k$. This notation, of course,
makes sense only after the existence of the lattice $\D$ has been established. 
We mention it here just for the readers who
may want to skip the construction and proceed to the next sections where this notation will be used without 
any comment.

We will call $\D$ a David-Semmes lattice associated with $\mu$. Its construction can be traced back
to the papers of David (\cite{D3}) and Christ (\cite{C}). There are several different ways to define 
them, some ways being more suitable than other for certain purposes. The presentation
we will give below is tailored to the Cantor-type construction in our proof, where it is convenient
to think that the cells are ``thick'' sets in $\R^{d+1}$, not just Borel subsets of $E$, so they can 
carry $C^2$-functions, etc.

Despite our ultimate goal being to construct the cells $Q$, we will start with defining their centers. 
The construction makes sense for an arbitrary closed set $E$ and the only place where $\mu$ will
play any role is the last property asserting that small neighborhoods of the boundaries have
small measures. 

For each $k\in\mathbb Z$, fix some maximal $2^{-4k}$-separated set $Z_k\subset E$. Clearly,
$Z_k$ is a $2^{-4k}$-net in $E$ (i.e., each point in $E$ lies in the ball $B(z,2^{-4k})$ 
for some $z\in Z_k$). For each $z\in Z_k$, define the level $k$ Voronoi cell $V_z$ of $z$ by
$$
V_z=\{x\in E:|x-z|=\min_{z'\in Z_k}|x-z'|\}\,.
$$
Note that $\cup_{z\in Z_k}V_z=E$,
$V_z\subset B(z,2^{-4k})$, and
$\operatorname{dist}(z,\cup_{z'\in Z_k\setminus\{z\}}V_{z'})\ge 2^{-4k-1}$.
The first property follows from the fact that every ball contains only finitely 
many points of $Z_k$, so every point $z\in Z_k$ has only finitely many not completely
hopeless competitors $z'\in Z_k$ for every given point $x\in E$ and, thereby, the minimum
$\min_{z'\in Z_k}|x-z'|$ is always attained. The second property is an immediate consequence
of the inequality $\min_{z'\in Z_k}|x-z'|<2^{-4k}$, which is just a restatement of the 
claim that $Z_k$ is a $2^{-4k}$-net in $E$. The last property just says that if $|x-z|<2^{-4k-1}$
for some $z\in Z_k$, then, for every other $z'\in Z_k$, we have
$$
|x-z'|\ge |z-z'|-|z-x|\ge 2^{-4k}-2^{-4k-1}=2^{-4k-1}>|x-z|\,,
$$
so the inclusion $x\in V_{z'}$ is impossible.

Observe also that for each $z\in Z_k$, there are only finitely many $w\in Z_{k-1}$ such that
$V_z\cap V_w\ne\varnothing$ (here, of course, $V_w$ is a level $k-1$ Voronoi cell constructed
using $Z_{k-1}$). Indeed, if $|z-w|>2^{-4k}+2^{-4(k-1)}$, then even the balls $B(z,2^{-4k})$
and $B(w,2^{-4(k-1)})$ are disjoint. However, only finitely many points in $Z_{k-1}$ lie
within distance $2^{-4k}+2^{-4(k-1)}$ from $z$.

Let now $z\in Z_k$, $w\in Z_\ell$, $\ell\ge k$. We say that $w$ is a descendant of $z$ if 
there exists a chain $z_k,z_{k+1},\dots,z_{\ell}$ such that $z_j\in Z_j$ for all $j=k,\dots,\ell$,
$z_k=z$, $z_\ell=w$, and $V_{z_j}\cap V_{z_{j+1}}\ne\varnothing$ for $j=k,\dots,\ell-1$. Note 
that each $z\in Z_k$ is its own descendant (with the chain consisting of just one entry $z$) 
according to this definition. Let $D(z)$ be the set of all descendants of $z$. Put 
$$
\wt V_z=\bigcup_{w\in D(z)}V_w\,.
$$
Note that $\wt V_z$ contains $V_z$ and is contained in the 
$2\sum_{\ell>k}2^{-4\ell}=\frac 2{15}2^{-4k}$-neighborhood of $V_z$.
Thus,
$$
\operatorname{dist}(z,\cup_{z'\in Z_k\setminus\{z\}}\wt V_{z'})\ge 2^{-4k-1}-\frac 2{15}2^{-4k}
>2^{-4k-2}\,.
$$
Our next aim will be to define a partial order $\prec$ on $\cup_k Z_k$ such that each $Z_k$ is linearly
ordered under $\prec$ and the ordering of $Z_{k+1}$ is consistent with that of $Z_k$ in the sense
that if $z',z''\in Z_{k+1}$ and $z'\prec z''$, then for every $w'\in Z_{k}$ such that 
$V_{w'}\cap V_{z'}\ne\varnothing$, there exists $w''\in Z_k$ such that 
$V_{w''}\cap V_{z''}\ne\varnothing$ and $w'\preceq w''$. In other words, the ordering
we are after is analogous to the
classical ``nobility order'' in the society: for $A$ to claim being nobler than $B$, he should, at least,
be able to show that his noblest parent in the previous generation is at least as noble as the 
noblest parent of $B$. Only if the noblest parents of $A$ and $B$ have equal nobility (which,
in the case of linear orderings can happen only if they coincide), the personal qualities
of $A$ and $B$ may be taken into account to determine their relative nobility. This informal
observation leads to the following construction.

First, we fix $k_0\in\mathbb Z$ and construct such an order inductively on $\cup_{k\ge k_0}Z_k$.
Start with any partial order $\dashv$ that linearly orders every $Z_k$ (the ``personal qualities'' order).
On $Z_{k_0}$, put $\prec\,=\,\dashv$. If $\prec$ is already defined on $Z_k$, for each $z\in Z_{k+1}$,
define $w(z)\in Z_k$ as the top (with respect to $\prec$) element of $Z_k$ for which 
$V_w\cap V_z\ne\varnothing $. Note that $w(z)$ always exists because $V_z$ intersects at least one
but at most finitely many Voronoi cells $V_w$ with $w\in Z_k$. Now we say that $z'\prec z''$
if either $w(z')\prec w(z'')$, or $w(z')= w(z'')$ and $z'\dashv z''$. It is easy to check that 
the order $\prec$ defined in this way is a linear order on $Z_{k+1}$ consistent with the order
defined on $Z_k$. 

To define an order on the full union $\cup_{k\in\mathbb Z}Z_k$, consider any sequence $\prec_{k_0}$ of
orders on $\cup_{k\ge k_0}Z_k$ defined above. Since the set of comparisons defining an order on 
$\cup_{k\in\mathbb Z}Z_k$ is countable, we can use the diagonal process to extract a subsequence of 
$\prec_{k_0}$ with $k_0\to-\infty$ so that for every finite set $Z\subset \cup_{k\in\mathbb Z}Z_k$,
the ordering of $Z$ by $\prec_{k_0}$ is defined and does not depend on $k_0$ if $k_0\le K(Z)$.
Now just define $\prec$ as the limit of $\prec_{k_0}$. Note that the linearity and the 
consistency conditions are
``finite'' ones (i.e., they can be checked looking only at how certain finite subsets of
$\cup_{k\in\mathbb Z}Z_k$ are ordered), so they will be inherited by the limit order.

At this point everything is ready to define the David-Semmes cells. For $z\in Z_k$, we just put
$$
E_z=\wt V_z\setminus \left(\bigcup_{z'\in Z_k,z\prec z'}\wt V_{z'}\right)\,.
$$ 
It is clear that $E_{z'}$ and $E_{z''}$ are disjoint for $z',z''\in Z_k$, $z'\ne z''$.
Also, the remarks above imply that
$$
B(z,2^{-4k-2})\cap E\subset E_z \subset B(z,2^{-4k+1})
$$
for all $z\in Z_k$. 

Since $\cup_{z\in Z_k}\wt V_z\supset \cup_{z\in Z_k} V_z\supset E$ and each point $x\in E$ is
contained only in finitely many $\wt V_z$, we have $\cup_{z\in Z_k} E_z=E$ ($x$ is contained in $E_z$
with the top $z$ among those for which $x\in \wt V_z$). Thus, for each fixed $k\in\mathbb Z$, 
the sets $E_z$, $z\in Z_k$ tile $E$.

Now fix $z\in Z_{k+1}$ and let $w$ be the top element of $Z_k$ among those for which 
$V_z\cap V_w\ne\varnothing$. Clearly, $D(z)\subset D(w)$, so $\wt V_z\subset\wt V_w$.
Take any $w'\in Z_k$ with $w\prec w'$. Let $Ch(w')=D(w')\cap Z_{k+1}$ be the set of ``children''
of $w'$. The consistency of $\prec$ implies that $z\prec z'$ for all $z'\in Ch(w')$. But then
$Ch(w')\subset \{z'\in Z_{k+1}:z\prec z'\}$, so 
$$
\bigcup_{z'\in Z_{k+1},z\prec z'}\wt V_{z'} \supset \bigcup_{z'\in Ch(w')}\wt V_{z'}\,.
$$
However, we clearly have
$$
D(w')=\{w'\}\cup \bigcup_{z'\in Ch(w')}D(z')
$$
and 
$$
V_{w'}\subset \bigcup_{z'\in Ch(w')}V_{z'}\subset \bigcup_{z'\in Ch(w')}\wt V_{z'}\,,
$$
so
$$
\wt V_{w'}\subset \bigcup_{z'\in Ch(w')}\wt V_{z'}\subset \bigcup_{z'\in Z_{k+1},z\prec z'}\wt V_{z'}\,.
$$
Thus,
$$
\bigcup_{w'\in Z_k,w\prec w'}\wt V_{w'}\subset \bigcup_{z'\in Z_{k+1},z\prec z'}\wt V_{z'}\,,
$$
so $E_z\subset E_w$.

This shows that the tiling at each level is a refinement of the tiling at the previous 
level and we have a nice dyadic structure on $E$ (except the cell sizes are powers of $16$ 
instead of the customary powers of $2$). We will now expand the cells $E_z\subset E$  to 
spatial cells $Q_z\subset\R^{d+1}$ by adding to each cell $E_z$ ($z\in Z_k$) 
all points $x\in\R^{d+1}\setminus E$
that lie in the $2^{-4k}$ neighborhood of $E_z$ and are closer to $E_z$ than to any other 
cell $E_{z'}$ with $z'\in Z_k$. Note that $Q_z$ defined in this way are disjoint at each level,
$Q_z\cap E=E_z$, and we have $Q_z\subset Q_w$ whenever $E_z\subset E_w$, $z\in Z_{k+1}$, $w\in Z_k$. 
To see the last property, just note that the $2^{-4(k+1)}$ neighborhood of $E_z$ is contained
in the $2^{-4k}$ neighborhood of $E_w$ and if $x\notin E$ is closer to $E_z$ than to any other
level $k+1$ cell, then it is closer to $E_w$ than to any other level $k$ cell as well (every level
$k$ cell is a finite union of level $k+1$ cells). Moreover, for every $z\in Z_k$, we have
$$
B(z,2^{-4k-3})\subset Q_z\subset B(z,2^{-4k+2})\,.
$$
The right inclusion follows immediately from the inclusion $E_z\subset B(z,2^{-4k+1})$ while the 
left one follows from the fact mentioned above  that the ball $B(z,2^{-4k-2})$ doesn't
intersect any cell $E_{z'}$ with $z'\in Z_k$, $z'\ne z$.

The construction of the David-Semmes lattice $\D$ is now complete and all that 
remains to prove is the ``small boundary'' property. Assume that $\mu$ is a $\wt C$-nice
measure that is AD regular in the entire $\R^{d+1}$ with the lower regularity constant
$\wt c$ and that $E=\supp\mu$. We shall use the notation $\D_k$ for the family 
of the level $k$ cells $Q$ and the notation $\ell(Q)$ for $2^{-4k}$ where 
$Q\in\D_k$ from now on. We will also write $z=z\ci Q$ instead of $Q=Q_z$, so 
from this point on, the David-Semmes
cells will be viewed as primary objects and all parameters related to them (like size,
center, etc.) as the derivative ones.

Since $\mu$ is AD regular and the cells $Q$ are squeezed between two balls
centered at $z\ci Q\in E=\supp\mu$ of radii comparable to $\ell(Q)$, we have
$$
c\ell(Q)^d\le \mu(Q)\le C\ell(Q)^d\,,
$$
where $c,C>0$ depend only on $d$ and $\wt c,\wt C$. We will now use the induction on $m\ge 0$
to show that 
$$
\mu(B_m(Q))\le (1-c)^m \mu(Q)
$$
where 
$$
B_m(Q)=\{x\in Q:\operatorname{dist}(x,\R^{d+1}\setminus Q)<16^{-2m}\ell(Q)\}
$$
for some $c>0$. This will yield the small boundary property with $\gamma=-\frac{\log (1-c)}{2\log 16}$.

The base $m=0$ is trivial regardless of the choice of $c\in(0,1)$.
To make the induction step from $m-1$ to $m\ge 1$, 
consider the cell $Q'$ that is two levels below $Q$ and contains
$z\ci Q$. Its diameter does not exceed $8\ell(Q')=\frac 1{32}\ell(Q)$. Since 
$B(z\ci Q,\frac 18\ell(Q))\subset Q$, the whole cell $Q'$ lies at the distance at 
least $(\frac 18-\frac 1{32})\ell(Q)>16^{-2m}\ell(Q)$ from the complement of $Q$. Thus,
$B_m(Q)\cap Q'=\varnothing$. For every other cell $Q''$ that is two levels down
from $Q$ and contained in $Q$, we, clearly, have
$$
B_m(Q)\cap Q''\subset B_{m-1}(Q'')\,.
$$ 
Hence, applying the induction assumption, and taking into account that those cells 
$Q''$ are disjoint and contained in $Q\setminus Q'$, we get
\begin{multline*}
\mu(B_m(Q))\le \sum_{Q''}\mu(B_{m-1}(Q''))
\\
\le
(1-c)^{m-1}\sum_{Q''}\mu(Q'')\le
(1-c)^{m-1}\left(1-\frac{\mu(Q')}{\mu(Q)}\right)\mu(Q)\,.
\end{multline*}
However, $\mu(Q')\ge c\ell(Q')^d=c\ell(Q)^d\ge c\mu(Q)$ (all three $c$ here are different but
depend on $d$, $\wt c$, and $\wt C$ only). If we choose $c$ in the statement to be the last $c$ 
in this chain, we will be able to complete the induction step, thus finishing the proof.

\section{Carleson families}
\label{Cfamilies}

From now on, we will fix a good AD regular in the entire space $\R^{d+1}$
measure $\mu$ and a David-Semmes lattice $\D$ associated with it. All constants
that will appear in this and later sections will be allowed to depend on the 
goodness and the lower AD regularity constants of $\mu$ in addition to the 
dependence on the dimension $d$. This dependence will no longer be mentioned 
explicitly on a regular basis though we may remind the reader about it now and then.

\begin{udef}
A family $\F\subset\D$ is called Carleson with Carleson constant $C>0$ if 
for every $P\in\D$, we have 
$$
\sum_{Q\in\F\cci P}\mu(Q)\le C\mu(P)
$$
where 
$$
\F\ci P=\{Q\in\D:Q\subset P\}\,.
$$
\end{udef}

Note that the right hand side can be replaced with $C\ell(P)^d$ because $\mu(P)$ is comparable
to $\ell(P)^d$ for every $P\in\D$. The main goal of this section is the following property
of non-Carleson families.

\begin{lem} 
\label{nonClayers}
Suppose that $\F$ is not Carleson. Then, for every $M\in\mathbb N$, $\eta>0$, we can find a cell $P\in\F$
and $M+1$ finite families $\LL_0,\dots,\LL_M\subset\F\ci P$ so that
\begin{itemize}
\item
$\LL_0=\{P\}$.
\item
No cell appears in more than one of the families $\LL_0,\dots,\LL_M$.
\item
The cells in each family $\LL_m$ ($m=0,\dots,M$) are pairwise disjoint.
\item
Each cell $Q'\in\LL_{m}$ ($m=1,\dots,M$) is contained in a unique strictly larger cell $Q\in \LL_{m-1}$.
\item
$\displaystyle\sum_{Q\in\LL_M}\mu(Q)\ge(1-\eta)\mu(P)\,.$
\end{itemize}
\end{lem}
We will usually refer to these $\LL_m$ as non-Carleson layers.

\begin{proof}
Note, first of all, that, when checking the Carleson property of $\F$, it is enough
to restrict ourselves to cells $P\in \F$. Indeed, suppose that the inequality 
$$
\sum_{Q\in\F\cci P}\mu(Q)\le C\mu(P)
$$
holds for every $P\in\F$. Take any $P\in\D$ and consider the family $\F\ci {0,P}$ of maximal
cells in $\F\ci P$ (i.e., the cells that aren't contained in any other cell from $\F\ci P$).
Then the cells $P'\in \F\ci {0,P}$ are disjoint and $\F\ci P=\cup_{P'\in \F\cci {0,P}}\F\ci{P'}$.
Thus
$$
\sum_{Q\in\F\cci P}\mu(Q)=\sum_{P'\in \F\cci {0,P}}\sum_{Q\in\F\cci{P'}}\mu(Q)
\le C\sum_{P'\in \F\cci {0,P}}\mu(P')\le C\mu(P)\,,
$$
so we automatically have the desired estimate for all cells $P\in\D$ with the same constant.

Next, observe that if every finite subfamily $\F'\subset \F$ is Carleson with the same Carleson
constant $C$, then the entire family $\F$ is Carleson with the same constant. Indeed, if
$$
\sum_{Q\in\F\cci P}\mu(Q)> C\mu(P)
$$
for some $P\in\D$, then we can restrict the sum on the left to a finite one and still preserve
the inequality.

Now fix $M,\eta$ and assume that $\F$ is not Carleson. Then we can find some 
finite subfamily $\F'\subset \F$ whose Carleson constant is as large as we want (note that
every finite family is Carleson with {\em some} Carleson constant).

Take any $P\in \F'$ and define the families $\F'\ci {m,P}$ of cells inductively as follows: 
$\F'\ci {0,P}=\{P\}$ and, if $\F'\ci {k,P}$
are already defined for $k<m$, then $\F'\ci {m,P}$ is the set of all maximal cells in
$\F'\ci P\setminus\cup_{k<m}\F'\ci {k,P}$. Observe that for every $m\ge 0$, we have
$$
\F'\ci P=\bigcup_{k=0}^{m-1}\F'\ci {k,P}\cup \bigcup_{P'\in\F'\cci{m,P}}\F'\ci {P'}
$$
and that for each $m$, the cells in $\F'\ci {m,P}$ are pairwise disjoint and (if $m>0$) each of them is 
contained in some unique cell from $\F'\ci {m-1,P}$. Thus, the families $\F'\ci {m,P}$ have all 
properties of the non-Carleson layers $\LL_m$ except, maybe, the last one. 
If we can find a starting cell $P\in\F'$ so that
$$
\sum_{Q\in\F'\cci{M,P}}\mu(Q)\ge(1-\eta)\mu(P)\,,
$$
we are done. Let $C(\F')$ be the best
Carleson constant of $\F'$ (it exists because,
to determine the Carleson constant
of $\F'$, we only need to look for the best constant in 
finitely many inequalities corresponding to all cells $P\in\F'$ ). Take $P\in\F'$ for which
this Carleson constant is attained and write
$$
C(\F')\mu(P)=\sum_{Q\in\F'\cci P}\mu(Q)
\le\sum_{k=0}^{M-1}\sum_{Q\in\F'\cci{k,P}}\mu(Q)+\sum_{P'\in\F'\cci{M,P}}\sum_{Q\in\F'\cci{P'}}\mu(Q)\,.
$$
However, the first sum on the right is at most $M\mu(P)$ and the second one can be bounded by 
$$
C(\F')\sum_{P'\in\F'\cci{M,P}}\mu(P')
$$
using the Carleson property of $\F'$.
Thus,
$$
\sum_{P'\in\F'\cci{M,P}}\mu(P')\ge \left(1-\frac{M}{C(\F')}\right)\mu(P)\ge (1-\eta)\mu(P)\,,
$$
provided that $\F'$ was chosen so that $C(\F')\ge M\eta^{-1}$.
\end{proof}

It is worth mentioning that though we stated and proved our lemma only in one direction 
(non-Carlesonness of a family implies the existence of non-Carleson layers in that family
for arbitrary $M,\eta>0$), it is actually a complete characterization of non-Carleson
families. We leave it to the reader to formulate and to prove the converse statement 
(which we will not use in this paper). 

\section{Riesz systems and families}
\label{Rfamilies}

Let $\psi\ci Q$ ($Q\in D$) be a system of Borel $L^2(\mu)$ functions 
(either scalar or vector-valued, as usual).

\begin{udef}
The functions $\psi\ci Q$ form a Riesz family with Riesz constant $C>0$ if 
$$
\left\|\sum_{Q\in\D}a\ci Q\psi\ci Q\right\|\ci{L^2(\mu)}^2\le C\sum a\ci Q^2
$$
for any real coefficients $a\ci Q$ only finitely many of which are non-zero.
\end{udef}

Note that if the functions $\psi\ci Q$ form a Riesz family with Riesz constant $C$, then for every
$f\in L^2(\mu)$, we have
$$
\sum_{Q\in\D}|\langle f,\psi\ci Q\rangle_\mu|^2\le C\|f\|\ci{L^2(\mu)}^2\,.
$$

Indeed, let $\F\subset\D$ be any finite collection of David-Semmes cells. Let 
$a\ci Q=\langle f,\psi\ci Q\rangle_\mu$ for $Q\in\F$. Put $g=\sum_{Q\in\F}a\ci Q\psi\ci Q$.
Then
$$
\sum_{Q\in\F}\langle f,\psi\ci Q\rangle_\mu^2=\langle f,g\rangle_\mu\le
\|f\|\ci{L^2(\mu)}\|g\|\ci{L^2(\mu)}\le
\|f\|\ci{L^2(\mu)}\left[C\sum_{Q\in\F} \langle f,\psi\ci Q\rangle_\mu^2\right]^{1/2}\,,
$$
so
$$
\sum_{Q\in\F}\langle f,\psi\ci Q\rangle_\mu^2\le 
C\|f\|\ci{L^2(\mu)}^2\,.
$$
Since $\F$ was arbitrary here, the same inequality holds for the
full sum over $\D$.

Assume next that for each cell $Q\in D$ we have a set $\Psi_Q$ of $L^2(\mu)$ functions associated
with $Q$.

\begin{udef}
The family $\Psi_Q$ ($Q\in\D$) of sets of functions is a Riesz system with Riesz constant $C>0$ if
for every choice of functions $\psi\ci Q\in\Psi_Q$, the functions $\psi\ci Q$ form a Riesz family with
Riesz constant $C$.
\end{udef}

\goal
The goal of this section is to present two useful Riesz systems: the Haar system $\Psi^h_Q(N)$ and the Lipschitz 
wavelet system $\Psi^\ell_Q(A)$, and to show how Riesz systems
can be used to establish that certain families of cells are Carleson.
\goalend

We shall start with the second task. Suppose that $\Psi_Q$ is any Riesz system. 
Fix any extension factor $A>1$. For each $Q\in\D$, define
\begin{equation}
\label{xi}
\xi(Q)=\inf_{E:B(z\cci Q,A\ell(Q))\subset E,\mu(E)<+\infty}
\sup_{\psi\in\Psi_Q} \mu(Q)^{-1/2}|\langle R_\mu\chi\ci E,\psi\rangle_\mu|\,.
\end{equation}
Then, for every $\delta>0$, the family $\F=\{Q\in\D:\xi(Q)\ge\delta\}$ is Carleson.

Indeed, if $P\in\D$ is any cell, then the set $E=B(z\ci P,(4+A)\ell(P))$ satisfies
$B(z\ci Q,A\ell(Q))\subset E$ for all cells $Q\subset P$. Choosing $\psi\ci Q\in\Psi_Q$
so that 
$$
|\langle R_\mu\chi\ci E,\psi\ci Q\rangle_\mu|>\frac\delta 2\mu(Q)^{1/2}\,,
$$
we see that 
\begin{multline*}
\sum_{Q\in \F\cci P}\mu(Q)\le \left(\frac 2\delta\right)^2\sum_{Q\in\D,Q\subset P}
|\langle R_\mu\chi\ci E,\psi\ci Q\rangle_\mu|^2\le C\delta^{-2}\|R_\mu\chi\ci E\|\ci{L^2(\mu)}^2
\\
\le C\delta^{-2}\|\chi\ci E\|\ci{L^2(\mu)}^2\le C\delta^{-2}(A+4)^d\ell(P)^d
\le C\delta^{-2}(A+4)^d\mu(P)\,,
\end{multline*}
so $\F$ is Carleson with Carleson constant $C\delta^{-2}(A+4)^d$.

Let now $N$ be any positive integer. For each $Q\in\D$, define the set of Haar functions
$\Psi^h_Q(N)$ of depth $N$ as the set of all functions $\psi$ that are supported on $Q$,
are constant on every cell $Q'\in\D$ that is $N$ levels down from $Q$, and satisfy
$\int\psi\,d\mu=0$, $\int\psi^2\,d\mu\le C$. The Riesz property follows immediately from
the fact that $\D$ can be represented as a finite union of the sets 
$\D^{(j)}=\cup_{k:k\equiv j\mod N}\D_k$ ($j=0,\dots,N-1$) and that for every choice 
of $\psi\ci Q\in\Psi^h_Q(N)$, the functions $\psi\ci Q$ 
corresponding to the cells $Q$ from a fixed $\D^{(j)}$ form a bounded orthogonal family.

In the Lipschitz wavelet system, the set $\Psi_Q^\ell(A)$ consists of all Lipschitz 
functions $\psi$ supported on $B(z\ci Q,A\ell(Q))$ such that $\int\psi\,d\mu=0$ and
$\|\psi\|\ci{\Lip}\le C\ell(Q)^{-\frac d2-1}$. Since $\mu$ is nice,
we automatically have $\int|\psi|^2\,d\mu\le C(A)\ell(Q)^{-d}\mu(Q)\le C(A)$ in this case. 

The Riesz property is slightly less obvious here. Note, first of all, that
if $Q,Q'\in\D$ and $\ell(Q')\le\ell(Q)$, then, for any two functions 
$\psi\ci Q\in \Psi_Q^\ell(A)$ and $\psi\ci {Q'}\in \Psi_{Q'}^\ell(A)$,
we can have $\langle\psi\ci Q,\psi\ci{Q'}\rangle_\mu\ne 0$ only if
$B(z\ci Q,A\ell(Q))\cap B(z\ci {Q'},A\ell(Q'))\ne\varnothing$, in which case,
$$
|\langle\psi\ci Q,\psi\ci{Q'}\rangle_\mu|\le C(A)\left[\frac{\ell(Q')}{\ell(Q)}\right]^{\frac d2+1}\,.
$$
Now take any coefficients $a\ci Q$ ($Q\in\D$) and write
\begin{multline*}
\left\|\sum_{Q\in\D}a\ci Q\psi\ci Q\right\|_{L^2(\mu)}^2\le 
2\sum_{Q,Q'\in\D,\,\ell(Q')\le\ell(Q)}|a\ci Q|\cdot |a\ci{Q'}|\cdot |\langle\psi\ci Q,\psi\ci{Q'}\rangle_\mu|
\\
\le C(A)\sum_{\substack{ Q,Q'\in\D,\,\ell(Q')\le\ell(Q) \\
B(z\cci Q,A\ell(Q))\cap B(z\cci {Q'},A\ell(Q'))\ne\varnothing}}
\left[\frac{\ell(Q')}{\ell(Q)}\right]^{\frac d2+1}|a\ci Q|\cdot |a\ci{Q'}|
\\
\le C(A)\sum_{\substack{ Q,Q'\in\D,\,\ell(Q')\le\ell(Q) \\
B(z\cci Q,A\ell(Q))\cap B(z\cci {Q'},A\ell(Q'))\ne\varnothing}}\left\{
\left[\frac{\ell(Q')}{\ell(Q)}\right]^{d+1}|a\ci Q|^2+\frac{\ell(Q')}{\ell(Q)}|a\ci{Q'}|^2
\right\}\,.
\end{multline*}
It remains to note that the sums 
$$
\sum_{\substack{ Q'\in\D:\,\ell(Q')\le\ell(Q) \\
B(z\cci Q,A\ell(Q))\cap B(z\cci {Q'},A\ell(Q'))\ne\varnothing}}
\left[\frac{\ell(Q')}{\ell(Q)}\right]^{d+1}\text{ and }
\sum_{\substack{ Q\in\D:\,\ell(Q')\le\ell(Q) \\
B(z\cci Q,A\ell(Q))\cap B(z\cci {Q'},A\ell(Q'))\ne\varnothing}}
\frac{\ell(Q')}{\ell(Q)}
$$
are bounded by some constants independent of $Q$ and $Q'$ respectively.

\section{Abundance of flat cells}
\label{abundanceofflats}

Fix $A,\alpha>0$. We shall say that a cell $Q\in\D$ is (geometrically) $(H,A,\alpha)$-flat if the measure
$\mu$ is (geometrically) $(H,A,\alpha)$-flat at $z\ci Q$ on the scale $\ell(Q)$.   

\goal
The goal of this section is to show that there exists an integer $N$,
a finite set $\HH$ of linear hyperplanes in $\R^{d+1}$, and a Carleson family $\F\subset\D$ 
(depending on $A,\alpha$)
such that for every cell $P\in\D\setminus\F$, there exist $H\in\HH$ and an 
$(H,A,\alpha)$-flat cell $Q\subset P$ that is at most $N$ levels down from $P$.
\goalend

We remind the reader that the measure $\mu$ has been fixed since Section \ref{Cfamilies}
and all constants and constructions may depend on its parameters in addition to any explicitly
mentioned quantities.

Fix $A'>1,\alpha'\in(0,1),\beta>0$ to be chosen later. We want to show first
that if $N>N_0(A',\alpha',\beta)$, then there exists
a Carleson family $\F_1\subset\D$ and a finite set $\HH$ of linear hyperplanes
such that every cell $P\in\D\setminus\F_1$ contains a geometrically $(H,5A',\alpha')$-flat
cell $Q\subset P$  at most $N$ levels down from $P$ for some linear hyperplane $H\in \HH$ that may depend on $P$.

Let $R=\frac 1{16}\ell(P)$. According to Lemma \ref{geometricflattening}, we can choose $\rho>0$ so 
that either there is a scale $\ell>\rho R$
and a point $z\in B(z\ci P,R-16[(5A'+5)+\frac{\alpha'}3]\ell)\subset P$ such that $\mu$ is geometrically  
$(H',16(5A'+5),\frac{\alpha'}3)$-flat at $z$ on the scale $\ell$ for some linear hyperplane $H'$, 
or there exist $\Delta\in(0,\frac 12)$,
$\delta\in(\rho,\Delta)$ and a point $z\in B(z\ci P,(1-2\Delta)R)$ with $\dist(z,\supp\mu)<\frac\delta 4 R$
such that $|[R(\psi\ci{z,\delta R,\Delta R}\mu)](z)|>\beta$ where $\psi\ci{z,\delta R,\Delta R}$ is the 
function introduced in the beginning of Section \ref{tangentmeasures}.

In the first case, take any point $z'\in\supp\mu$ such that $|z-z'|<\frac{\alpha'}3 \ell$ and choose the cell $Q$
with $\ell(Q)\in[\ell,16\ell)$ that contains $z'$. Since $z'\subset B(z\ci P,R)\subset P$ and $\ell(Q)<\ell(P)$, 
we must have $Q\subset P$. 
Also, since $|z\ci Q-z'|\le 4\ell(Q)$, we have $|z-z\ci Q|<4\ell(Q)+\frac{\alpha'}{3}\ell<5\ell(Q)$.

Note now that, if $\mu$
is geometrically $(H,16A,\alpha)$-flat at $z$ on the scale $\ell$, then it is 
geometrically $(H,A,\alpha)$-flat at $z$ on every scale $\ell'\in[\ell,16\ell)$.   

Note also that the geometric flatness is a reasonably 
stable condition with respect to shifts of the point and rotations of the plane.
More precisely, if $\mu$ is geometrically $(H',A+5,\alpha)$-flat at $z$ on the scale $\ell$, then it is 
geometrically $(H,A,2\alpha+A\e)$-flat at $z'$ on the scale $\ell$ for every $z'\in B(z,5\ell)\cap\supp\mu$ and every 
linear hyperplane $H$ with unit normal vector $n$ such that
the angle between $n$ and the unit normal vector $n'$ to $H'$ is less than $\e$. 
To see it, it is important to observe first that, despite the distance from $z$ to $z'$ may be quite large,
the distance from $z'$ to the affine hyperplane $L'$ containing $z$ and parallel to $H'$ can be only $\alpha\ell$,
so we do not need to shift $L'$ by more than this amount to make it pass through $z'$. Combined with the
inclusion $B(z',A\ell)\subset B(z,(A+5)\ell)$, this allows us to conclude that  
$\mu$ is $(H',A,2\alpha)$-flat at $z'$ on the scale $\ell$.  
After this shift, we can rotate the plane $L'$ around the $(d-1)$-dimensional affine plane containing $z'$ and
orthogonal to both $n$ and $n'$ by an angle less than $\e$ to make it parallel to $H$. Again, no point of $L' \cap B(z,A\ell)$
will move by more than $A\e\ell$ and the desired conclusion follows.

Applying these observations with $\ell'=\ell(Q)$, $z'=z\ci Q$, $\e=\frac{\alpha'}{3A}$,
and choosing any finite $\e$-net $Y$ on the unit sphere, we see that $\mu$ is geometrically
$(H,5A',\alpha')$-flat at $z\ci Q$ on the scale $\ell(Q)$ with some $H$ whose unit normal belongs to $Y$. 
Note also that the number of levels between $P$ and $Q$ in this case is 
$\log_{16}\frac{\ell(P)}{\ell(Q)}\le \log_{16}\rho^{-1}+C$.

In the second case, let $z'$ be a point of $\supp\mu$ such that $|z-z'|<\frac\delta 4R$. Note that 
$z'\in B(z\ci P,2R)\subset P$. Let now $Q$ and $Q'$ be the largest cells containing $z'$ under the restrictions
that $\ell(Q)<\frac \Delta{32} R$ and $\ell(Q')<\frac \delta{32} R$. Since both bounds are less than $\ell(P)$ and the 
first one is greater than the second one,
we have $Q'\subset Q\subset P$. 

Now take any set $E\supset B(z,2R)$ with $\mu(E)<+\infty$ and consider 
the difference of the averages of $R_\mu\chi\ci E$
over $Q$ and $Q'$ with respect to the measure $\mu$. 

We can write $\chi\ci E=\psi\ci{z,\delta R,\Delta R}+f_1+f_2$ where $|f_1|,|f_2|\le 1$ and 
$\supp f_1\subset \bar B(z,2\delta R)$, $\supp f_2\cap B(z,\Delta R)=\varnothing$.

Note that 
$$
\int|f_1|^2\,d\mu\le \mu(\bar B(z,2\delta R))\le C(\delta R)^d\le C\ell(Q')^d\le C\mu(Q')\le C\mu(Q)\,,
$$
so we have the same bound for $\int |R_\mu f_1|^2\,d\mu$, whence the averages of $R_\mu f_1$ 
over $Q$ and $Q'$ are bounded by some constant.

Note also that $Q\subset B(z',8\ell(Q))\subset B(z',\frac\Delta 4R)\subset B(z,\frac\Delta 2R)$, so the distance
from $Q$ to $\supp f_2$ is at least $\frac\Delta 2 R>\ell(Q)$. Thus,
$$
\|R(f_2\mu)\|\ci{\Lip(Q)}\le C\ell(Q)^{-1}
$$ 
so the difference of any two values of $R(f_2\mu)$ on $Q$ is bounded by a constant and, thereby, so is the
difference of the averages of $R_\mu f_2$ over $Q$ and $Q'$.

To estimate the difference of averages of $R_\mu\psi\ci{z,\delta R,\Delta R}$, note first that
$$
\|R_\mu\psi\ci{z,\delta R,\Delta R}\|\ci{L^2(\mu)}^2
\le C\|\psi\ci{z,\delta R,\Delta R}\|\ci{L^2(\mu)}^2\le C(\Delta R)^d\le C\ell(Q)^d\le C\mu(Q)\,,
$$
so the average over $Q$ is bounded by a constant. On the other hand,
$$
Q'\subset B(z',8\ell(Q'))\subset B(z',\frac\delta 4R)\subset B(z,\frac\delta 2R)\,.
$$ 
Since the distance from $B(z,\frac\delta 2R)$ to $\supp \psi\ci{z,\delta R,\Delta R}$ is at least $\frac\delta 2R$,
we have
$$
\|R(\psi\ci{z,\delta R,\Delta R}\mu)\|\ci{\Lip(B(z,\frac\delta 2R))}\le C(\delta R)^{-1}\,.
$$
Thus, all values of $R_\mu \psi\ci{z,\delta R,\Delta R}$ on $Q'\subset B(z,\frac\delta 2R)$ can differ from 
\linebreak
$
[R(\psi\ci{z,\delta R,\Delta R}\mu)](z)
$
only by a constant and the average over $Q'$ is at least $\beta-C$ in absolute value. 

Bringing all these estimates together, we conclude that the difference of averages of $R_\mu \chi\ci E$ over
$Q$ and $Q'$ is at least $\beta-C$ in absolute value for every set $E\supset B(z,2R)$ and, thereby, for 
every set $E\supset B(z\ci P,5\ell(P))$. Observe now that this conclusion can be rewritten as
$$
\mu(P)^{-\frac 12}|\langle R_\mu\chi\ci E,\psi\ci P\rangle_\mu|\ge c\rho^{\frac d2}(\beta-C)
$$
where 
$$
\psi\ci P=[\rho\ell(P)]^{\frac d2}\left(\frac{1}{\mu(Q)}\chi\ci Q-\frac{1}{\mu(Q')}\chi\ci Q'\right)
$$
and that $\psi\ci P\in\Psi_Q^\ell(N)$, where, as before, $\Psi_P^\ell(N)$ is the Haar system of depth $N$, 
with $N=\log_{16}\frac{\ell(P)}{\ell(Q')}\le \log_{16}\rho^{-1}+C$
(the normalizing factor $\rho^{\frac d2}$ in the definition of $\psi\ci P$ is just enough to make the norm
$\|\psi\ci P\|\ci{L^2(\mu)}$ bounded by a constant and all the other properties of a Haar function are obvious). 

Thus, we conclude that for such $P$, the quantity $\xi(P)$ defined by (\ref{xi}) using the Haar system
of depth $N$ and the extension factor $5$ is bounded from below by a fixed positive constant, provided
that $\beta$ has been chosen not too small. Consequently, the family $\F_1$ of such cells $P$ is Carleson.

As we have seen, for $P\notin\F_1$, we can find a geometrically $(H,5A',\alpha')$-flat cell $Q\subset P$ at most 
$\log_{16}\rho^{-1}+C$ levels down from $P$ with $H$ from some finite family $\HH$ of linear hyperplanes
(depending on the choice of $A',\alpha'$, of course). If we use the parameters $A'$ and $\alpha'$ determined 
by the Flattening
Lemma (Proposition \ref{flatteninglemma}), then the only case in which we cannot conclude that this cell is 
$(H,A,\alpha)$-flat is the case when for every set $E\supset B(z\ci Q,(A+\alpha+5A'+\alpha')\ell(Q))$ 
with $\mu(E)<+\infty$, 
we can find a mean
zero (with respect to $\mu$) Lipschitz function $g$ supported on $B(z\ci Q, 5A'\ell(Q))$ with $\|g\|\ci{\Lip}\le\ell(Q)^{-1}$
such that 
$|\langle R_\mu\chi\ci E,g \rangle_\mu|=|\langle R_{\chi\cci E\mu} 1,g \rangle_{\chi\cci E\mu}|>\alpha'\ell(Q)^d$ 
(otherwise
the Flattening Lemma is applicable to the measure $\chi\ci E\mu$ whose $(H,A,\alpha)$-flatness at $z\ci Q$ 
on the scale $\ell(Q)$ is equivalent to the $(H,A,\alpha)$-flatness of $\mu$ itself).

However the last inequality can be rewritten as 
$$
\mu(P)^{-\frac 12}|\langle R_\mu\chi\ci E,\psi\ci P \rangle_\mu|>c\rho^{d+1}\alpha'
$$
where 
$$
\psi\ci P=\rho\ell(P)^{-\frac d2} g\,.
$$
Note that $\|\psi\ci P\|\ci{\Lip}\le C\ell(P)^{-\frac d2-1}$ and 
$\supp\psi\ci P\subset B(z\ci Q,5A'\ell(Q))\subset B(z\ci Q,R)\subset B(z\ci P,5\ell(P))$,
so we see that in this case we again have $\xi(P)$ bounded from below by a fixed constant,
but now with respect to the Lipschitz wavelet system $\Psi_Q^\ell(5)$ and the extension factor 
$A+\alpha+5A'+\alpha'+5$, say.
Thus the family $\F_2$ of such exceptional cells is Carleson as well and it remains to put $\F=\F_1\cup\F_2$
to finish the proof of the main statement of this section.

\section{Alternating non-BAUP and flat layers}
\label{alternatinglayers}

Recall that our goal is to prove that the family of all non-BAUP cells $P\in\D$ is Carleson.
In view of the result of the previous section, it will suffice to show that 
we can choose $A,\alpha>0$ such that for every fixed 
linear hyperplane $H$ and for every integer $N$, 
the corresponding family $\F=\F(A,\alpha,H,N)$ 
of all non-BAUP cells $P\in D$ containing an $(H,A,\alpha)$-flat cell $Q$ at most $N$ 
levels down from $P$ is Carleson. The result of this section can be stated as follows.  
\begin{ulem}
If $\F$ is not Carleson, then for every positive integer $K$ and every $\eta>0$, there exist
a cell $P\in\F$ and $K+1$ alternating pairs of finite layers $\PP_k,\QQ_k\subset \D$ ($k=0,\dots,K$)
such that
\begin{itemize}
\item 
$\PP_0=\{P\}$.
\item
$\PP_k\subset\F\ci P$ for all $k=0,\dots,K$.
\item
All layers $\QQ_k$ consist of $(H,A,\alpha)$-flat cells only.
\item
Each individual layer (either $\PP_k$, or $\QQ_k$) consists of pairwise disjoint cells.
\item
If $Q\in\QQ_k$, then there exists $P'\in\PP_k$ such that $Q\subset P'$ ($k=0,\dots,K$).
\item
If $P'\in\PP_{k+1}$, then there exists $Q\in\QQ_k$ such that $P'\subset Q$ ($k=0,\dots,K-1$).
\item 
$\sum_{Q\in\QQ_K}\mu(Q)\ge (1-\eta)\mu(P)$.
\end{itemize}
\end{ulem}
In other words, each layer tiles $P$ up to a set of negligible measure and they have the usual Cantor type hierarchy
(due to this hierarchy, it suffices to look only at the very bottom layer to evaluate the efficiency of  the tiling for
all of them). The construction in this section is rather universal and does not depend on the meaning of the words
``non-BAUP'' in any way. All that we need to know here is that some cells are BAUP and some are not. Note that
we do not exclude here the possibility that the same cell is used in several different layers.
This will never really happen because the non-BAUPness condition is, in fact, just a particular quantitative 
negation of the flatness condition, so, when finally choosing our parameters, we will ensure  that
no non-BAUP cell can be an $(H,A,\alpha)$-flat cell as well, thus guaranteeing that we always go down when moving
from each layer to the next. Also our construction will be done in such a way that no two different $\PP$ layers
can contain the same cell. However, the disjointness of layers is not a part of the formal statement we have 
just made and the results of this and the next sections remain perfectly valid even if all layers we construct here consist of the single starting cell $P$,
which, in that case, must be simultaneously non-BAUP and $(H,A,\alpha)$-flat.   
\begin{proof}
Suppose $\F$ is not Carleson. By Lemma \ref{nonClayers}, for every $\eta'>0$ and every positive integer $M$, we can find
a cell $P\in\F$ and $M+1$ non-Carleson layers $\LL_0,\dots,\LL\ci M\subset\F\ci P$ that have the desired Cantor-type 
hierarchy and satisfy $\sum_{P'\in\LL\cci M}\mu(P')\ge (1-\eta')\mu(P)$ (see Section \ref{Cfamilies} for details).

We shall start with describing the main step of the construction, which will allow us to go from each layer $\PP_k$ 
to the next layer $\PP_{k+1}$ creating the intermediate layer $\QQ_k$ on the way. Let  $m$ be much smaller than 
$M$, so that there are as many available non-Carleson layers down from $m$ as we may possibly need. Fix a large integer $S>0$.

Let  $\LL'_m \subset \LL_m$. We shall call a cell $P''\in\LL_{m+sN}$ ($s=1,\dots,S$) 
\textit{exceptional} if it is contained in some cell $P'\in\LL'_m$ but there is no $(H,A,\alpha)$-flat cell
$Q\in\D$ such that $P''\subset Q\subset P'$. We claim that for each $s=1,\dots,S$, the sum of $\mu$-measures of 
all exceptional cells in $\LL_{m+sN}$ does not exceed $(1-c 16^{-Nd})^s\mu(P)$.

The proof goes by induction on $s$. To prove the base $s=1$, just recall that every cell $P'\in \LL'_m \subset \LL_m$ contains
some $(H,A,\alpha)$-flat cell $Q(P')\in\D$ at most $N$ levels down from $P'$. Since every cell $P''\in \LL_{m+N}$ 
that is contained in $P'\in\LL_m'$ must be at least $N$ levels down from $P'$ (the non-Carleson layers constructed in 
Section \ref{Cfamilies} cannot have repeating cells), we conclude that every cell $P''\in\LL_{m+N}$ contained in
$P'$ is either contained in $Q(P')$ or disjoint with $Q(P')$. In the first case $P''$ is, certainly, not exceptional,
so the sum of the measures of all exceptional cells in $\LL_{m+N}$ that are contained in $P'$ is at most
$\mu(P')-\mu(Q(P'))\le (1-c 16^{-Nd})\mu(P')$ whence the total sum of measures of all exceptional cells in $\LL_{m+N}$
is at most $(1-c 16^{-Nd})\sum_{P'\in\LL'_m}\mu(P')\le (1-c 16^{-Nd})\mu(P)$.

To make the induction step, assume that we already know that the claim holds for some $s$. 
Note that every exceptional cell $P''\in\LL_{m+(s+1)N}$ is contained in some cell $\wt P''\in\LL_{m+sN}$.
We claim that $\wt P''$ must be exceptional as well. Indeed, let $P'$ be the cell in $\LL_m'$ containing $P''$.
Then $\wt P''\cap P'\ne\varnothing$, which, due to the hierarchy of the non-Carleson layers, is possible only
if $\wt P''\subset P'$. If there had been any $(H,A,\alpha)$-flat cell $Q$ satisfying $\wt P''\subset Q\subset P'$,
we would also have $P''\subset Q\subset P'$, so the cell $P''$ would not be exceptional. Now it remains to note
that $P''$ must also be disjoint with $Q(\wt P'')$ and to repeat the argument above to conclude that the sum of 
measures of all exceptional cells in $\LL_{m+(s+1)N}$ is at most $(1-c 16^{-Nd})$ times 
the sum of measures of all exceptional cells in $\LL_{m+sN}$. It remains to apply the induction assumption and
to combine two factors into one.

Now let $\LL'_{m+SN}\subset \LL_{m+SN}$ be the set of all cells in $\LL_{m+SN}$
that are contained in some cell from $\LL_m'$ but are not exceptional. Then, for every cell $P''\in\LL'_{m+SN}$ and the
corresponding cell $P'\in\LL_m'$ containing $P''$, there exists an $(H,A,\alpha)$-flat cell $Q\in\D$ such that
$P''\subset Q\subset P'$. Let $\QQ$ be the set of all cells $Q$ that can arise in this way and let $\QQ^*$ be the
set of all maximal cells in $\QQ$ (i.e., cells that are not contained in any larger cell from $\QQ$). Then the 
cells $Q\in\QQ^*$ are pairwise disjoint and form an intermediate layer between $\LL'_m$ and $\LL'_{m+SN}$ in the
sense that every $Q\in\QQ^*$ is contained in some cell $P'\in\LL'_m$ and every $P''\in\LL'_{m+SN}$ is contained in some
cell $Q\in\QQ^*$.

Moreover,
\begin{multline*}
\sum_{P''\in\LL'_{m+SN}}\mu(P'')\ge \sum_{P''\in\LL_{m+SN}}-\sum_{\substack{P''\in\LL_{m+SN}: P''\not\subset P'
\\ \text{ for any }P'\in\LL'_m}}-\sum_{\substack{P''\in\LL_{m+SN}: 
\\ P''\text{ is exceptional}}}
\\
\ge(1-\eta')\mu(P)-[\mu(P)-\sum_{P'\in\LL_m'}\mu(P')]-(1-c 16^{-Nd})^S\mu(P)
\\
=\sum_{P'\in\LL_m'}\mu(P')-[\eta'+(1-c 16^{-Nd})^S]\mu(P)\,.
\end{multline*}
Now assume that $M>(K+1)SN$. Then we can start with $\LL'_0=\LL_0=\{P\}$ and apply this construction
inductively with $m=0,SN,2SN,\dots,KSN$. The resulting layers $\LL'_{kSN}$ ($k=0,\dots,K$) will satisfy all properties 
of $\PP_k$ and the intermediate layers $\QQ^*$ (one of those will arise during each step) will satisfy
all properties of $\QQ_k$ except, perhaps, the measure estimate.

However, since $\LL'_0$ covers $P$ completely and during each step the total measure loss is bounded 
by $[\eta'+(1-c 16^{-Nd})^S]\mu(P)$, we have 
$$
\sum_{Q\in\QQ_K}\mu(Q)\ge\sum_{P'\in\LL'_{(K+1)SN}}\mu(P')\ge \mu(P)-(K+1)[\eta'+(1-c 16^{-Nd})^S]\mu(P)
$$    
and it remains to note that for any fixed $K$, we can always make $(K+1)[\eta'+(1-c 16^{-Nd})^S]$ less than
$\eta$ if we choose $\eta'$ small enough and $S$ large enough.
\end{proof} 

\section{Almost orthogonality}
\label{almostorthogonality}

Fix $K$. Choose $\e>0$, $A,\alpha>0$, $\eta>0$ in this order and run the construction of the previous section.
In this section we will be primarily interested in the flat layers $\QQ_k$ ignoring the non-BAUP layers $\PP_k$ 
almost entirely.

For a cell $Q\in\D$ and $t>0$, define
$$
Q_t=\{x\in Q:\dist(x,\R^{d+1}\setminus Q)\ge t\ell(Q)\}\,.
$$
Note that $\mu(Q\setminus Q_t)\le Ct^\gamma\mu(Q)$ for some fixed $\gamma>0$ (see Section \ref{DSlattice}).
Let $\f_0$ be any $C^\infty$ function supported on $B(0,1)$ and such that $\int\f_0\,dm=1$ where $m$ is the 
Lebesgue measure in $\R^{d+1}$. Put
$$
\f\ci Q=\chi\ci{Q_{2\e}}*  \frac{1}{(\e\ell(Q))^d} \,\f_0\Bigl(\frac{\cdot}{\e\ell(Q)}\Bigr)\,.
$$
Then $\f\ci Q=1$ on $Q_{3\e}$ and $\supp\f\ci Q\subset Q_\e$. In particular, the diameter of $\supp\f\ci Q$ 
is at most $8\ell(Q)$.
In addition,
$$
\|\f\ci Q\|\ci{L^\infty}\le 1,\quad \|\nabla\f\ci Q\|\ci{L^\infty}\le \frac{C}{\e\ell(Q)},\quad
\|\nabla^2\f\ci Q\|\ci{L^\infty}\le \frac{C}{\e^2\ell(Q)^2}\,.
$$  
From now on, we will be interested only in the cells $Q$ from the flat layers $\QQ_k$. With each such cell
$Q$ we will associate the corresponding approximating plane $L(Q)$ containing $z\ci Q$ and parallel to 
$H$ and the approximating measure $\nu\ci Q=a\ci Q\f\ci Q m\ci {L(Q)}$ where $a\ci Q$ is chosen so that 
$$
\nu\ci Q(\R^{d+1})=\int \f\ci Q\,d\mu\,.
$$
Note that since $B(z\ci Q,(\frac 18-3\e)\ell(Q))\subset Q_{3\e}$ and $Q\subset B(z\ci Q,4\ell(Q))$, 
both integrals $\int\f\ci Q\,dm\ci{L(Q)}$ and $\int\f\ci Q\,d\mu$ are comparable to $\ell(Q)^d$,
provided that $\e<\frac 1{48}$, say. In particular, in this case, the normalizing factors $a\ci Q$ 
are bounded by some constant. 

Define
$$
G_k=\sum_{Q\in\QQ_k}\f\ci Q\RH[\f\ci Q\mu-\nu\ci Q]\,,\quad  k=0,\dots,K\,.
$$
We remind the reader of our convention to understand $\RH(\f\ci Q\mu)$ as $\RH_\mu\f\ci Q$
on $\supp\mu$ (see Section \ref{weaklimits}) and of Lemma \ref{smoothtolip}
that shows that $\RH\nu\ci Q$ can be viewed as a Lipschitz function in the 
entire space $\R^{d+1}$. In what follows, we will freely integrate various expressions
including both $\RH(\f\ci Q\mu)$ and $\RH(\nu\ci Q)$ with respect to $\mu$, which
makes sense in view of what we just said. However, we will be very careful with the
integration of expressions involving $\RH(\f\ci Q\mu)$ with respect to $\nu\ci Q$ and
always make sure that for each point $x$ in the integration domain, 
$x$ is not contained in the support of any function $\f\ci Q$ for
which $\RH(\f\ci Q\mu)$ in the integrand is not multiplied by some cutoff factor vanishing 
in some neighborhood of $x$. 

Now put 
$$
F_k=G_k-G_{k+1}\text{ when }k=0,\dots,K-1,\qquad F_K=G_K\,.
$$
Note that 
$$
\sum_{m=k}^K F_m=G_k\,.
$$
The objective of this section is to prove the following
\begin{prop}
\label{scalarproducts}
Assuming that $\e<\frac 1{48}$, $A>5$, and $\alpha<\e^8$, we have
$$
|\langle F_k,G_{k+1}\rangle|\le 
\sigma(\e, \alpha)\mu(P)
$$
for all $k=0,\dots,K-1$, where $\sigma(\e,\alpha)$ is some positive function such that
$$
\lim_{\e\to 0+}[\lim_{\alpha\to 0+}\sigma(\e,\alpha)]=0\,.
$$
\end{prop}

In plain English, the double limit condition  on $\sigma(\e, \alpha)$ means that we
can make $\sigma(\e,\alpha)$ as small as we want by first choosing $\e>0$ small enough and then
choosing $\alpha>0$ small enough. 
The exact formula for $\sigma(\e,\alpha)$ will be of
no importance for the rest of the argument, so we do not even mention it here despite 
it will be explicitly written in the end of the proof.

The assumptions $\e<\frac 1{48}$ and $A>5$
are there to ensure that all the results of Section \ref{flatnesscondition} can be freely applied 
with $\f\ci Q$ in the role of $\f$ and $\nu\ci Q$ in the role of $\nu$. 
The assumption $\alpha<\e^8$
is just used to absorb some expressions involving $\alpha$ and $\e$ into constants instead 
of carrying them around all the time.  

Several tricks introduced in this section will be used again and again in what 
follows so we suggest that the reader goes over all details of the proof because 
here they are presented in a relatively simple setting unobscured by any other 
technical considerations or logic twists. Also, there is a technical lemma in the body of the proof
(Lemma \ref{summandbound}) that will be used several times later despite it is not 
formally proclaimed as one of the main results of this section.

\begin{proof}
We start with showing that, under our assumptions, 
$\|G_k\|\ci{L^p(\mu)}^p\le C\mu(P)$ for $p=2,4$ and all $k=0,\dots,K$.
Since 
$$
G_k=\sum_{Q\in\QQ_k}\f\ci Q\RH[\f\ci Q\mu-\nu\ci Q]
$$
and the summands have pairwise disjoint supports, it will suffice to prove the inequality
$$
\|\f\ci Q\RH(\f\ci Q\mu-\nu\ci Q)\|\ci{L^p(\mu)}^p\le C\mu(Q)
$$
for each individual $Q\in\QQ_k$ and then observe that $\sum_{Q\in\QQ_k}\mu(Q)\le\mu(P)$.

Since we shall need pretty much the same estimate in Section \ref{approximatingmeasure},
we will state it as a separate lemma here.
\begin{lem}
\label{summandbound}
Let $p=2$ or $p=4$. For each $k=0,\dots,K$ and for each cell $Q\in\QQ_k$, we have
$$
\|\f\ci Q\RH\nu\ci Q\|\ci{L^p(\mu)}^p\le \|\chi\ci Q\RH\nu\ci Q\|\ci{L^p(\mu)}^p \le C\mu(Q)\,.
$$
As a corollary, we have
$$
\|\f\ci Q\RH(\f\ci Q\mu-\nu\ci Q)\|\ci{L^p(\mu)}^p\le \|\chi\ci Q\RH(\f\ci Q\mu-\nu\ci Q)\|\ci{L^p(\mu)}^p \le C\mu(Q)
$$
\end{lem}
\begin{proof}
As we have already mentioned in Section \ref{riesztransform}, $\RH_\mu$ is bounded in both
$L^2(\mu)$ and $L^4(\mu)$, so we even have
$$ 
\|\RH_\mu\f\ci Q\|\ci{L^p(\mu)}^p\le C\|\f\ci Q\|\ci{L^p(\mu)}^p\le C\mu(Q)
$$
for both values of $p$ we are interested in and the cutoffs $\f\ci Q$ and $\chi\ci Q$ can only diminish 
the left hand side. 
Thus, we only need to prove the first chain of inequalities in the lemma.

The left inequality is trivial because $\f\ci Q\le \chi\ci Q$ pointwise.
To prove the right inequality, fix any Lipschitz function $\wt\f_0:\R^{d+1}\to [0,1]$ such that
$\f_0=1$ on $B(0,4)$ and $\f_0=0$ outside $B(0,5)$, put
$$
\wt\f\ci Q(x)=\wt\f_0\left(\frac{x-z\ci Q}{\ell(Q)}\right)\,,
$$ 
and write 
$$
\|\chi\ci Q\RH\nu\ci{Q}\|\ci{L^p(\mu)}^p
=\int_{Q} |\RH\nu\ci{Q}|^p\,d\mu \le \int \wt\f\ci Q|\RH\nu\ci{Q}|^p\,d\mu
$$
Let 
$$
\wt a\ci Q=\left(\int\wt\f\ci Q\,d m\ci{L(Q)}\right)^{-1}\int\wt\f\ci Q\,d\mu\,.
$$
Note that both integrals in the definition of $\wt a\ci Q$ are comparable to $\ell(Q)^d$,
so $\wt a\ci Q\le C$.
Put
$$
\wt\nu\ci Q=\wt a\ci Q m\ci{L(Q)}\,.
$$
Since $\RH_{m\cci {L(Q)}}$ is bounded in $L^p(m\ci{L(Q)})$, we have
\begin{multline*}
\int |\RH\nu\ci{Q}|^p\,d\wt\nu\ci Q\le C\int |\RH\nu\ci{Q}|^p\,d m\ci{L(Q)}
\\
\le C\|\f\ci Q\|\ci{L^p(m\cci{L(Q)})}^p\le C\ell(Q)^d\le C\mu(Q)\,.
\end{multline*}
On the other hand, the $C^2$-estimates for $\f\ci Q$
in the beginning of this section combined with Lemma \ref{smoothtolip} imply that 
$$
\|\RH\nu\ci{Q}\|\ci{L^\infty}\le \frac{C}{\e^2}\text{ and }\|\RH\nu\ci{Q}\|\ci{\Lip}\le \frac{C}{\e^2\ell(Q)}\,.
$$
In addition, we clearly have $\|\wt \f\ci Q\|\ci{\Lip}\le \frac{C}{\ell(Q)}$.
Thus, when $\alpha<\e^8<1$, Lemma \ref{mutonulip} immediately yields
\begin{multline*}
\int |\RH\nu\ci{Q}|^p\,d(\wt\f\ci Q\mu-\wt\nu\ci Q)\le C\alpha\ell(Q)^{d+2}\frac 1{\e^{2(p-1)}}\frac 1{\e^2\ell(Q)}
\frac 1{\ell(Q)}
\\
=C\alpha\e^{-2p}\ell(Q)^d\le C\mu(Q)
\end{multline*}
for $p=2,4$, so
\begin{multline*}
\int\wt\f\ci Q |\RH\nu\ci{Q}|^p\,d\mu=
\int |\RH\nu\ci{Q}|^p\,d(\wt \f\ci Q\mu)
\\
=\int |\RH\nu\ci{Q}|^p\,d\wt\nu\ci Q+\int |\RH\nu\ci{Q}|^p\,d(\wt\f\ci Q\mu-\wt\nu\ci Q)\le C\mu(Q)\,,
\end{multline*}
as required.
\end{proof}

Now represent $F_k$ as
\begin{multline*}
F_k=\left(\sum_{Q\in\QQ_k}\f\ci Q\RH_\mu\f\ci Q-\sum_{Q\in\QQ_{k+1}}\f\ci Q\RH_\mu\f\ci Q\right)
\\
-\sum_{Q\in\QQ_k}\f\ci Q\RH \nu\ci Q+\sum_{Q\in\QQ_{k+1}}\f\ci Q\RH \nu\ci Q=F_k^{(1)}-F_k^{(2)}+F_k^{(3)}\,.
\end{multline*}
Note that
$$
\|\RH_\mu(\f\ci Q-\chi\ci Q)\|\ci{L^p(\mu)}^p\le C\|\f\ci Q-\chi\ci Q\|\ci{L^p(\mu)}^p
\le C\mu(Q\setminus Q_{3\e})\le C\e^{\gamma}\mu(Q)
$$
for $p=2,4$.
Also
\begin{multline*}
\|(\f\ci Q-\chi\ci Q)\RH_\mu\chi\ci Q\|\ci{L^2(\mu)}^2\le \|\f\ci Q-\chi\ci Q\|\ci{L^4(\mu)}^2
\|\RH_\mu\chi\ci Q\|\ci{L^4(\mu)}^2
\\
\le C\|\f\ci Q-\chi\ci Q\|\ci{L^4(\mu)}^2
\|\chi\ci Q\|\ci{L^4(\mu)}^2
\le C\mu(Q\setminus Q_{3\e})^{\frac 12}\mu(Q)^{\frac 12}\le C\e^{\frac\gamma 2}\mu(Q)\,.
\end{multline*}
Thus,
\begin{multline*}
\|\f\ci Q\RH_\mu\f\ci Q-\chi\ci Q\RH_\mu\chi\ci Q\|\ci{L^2(\mu)}^2
\\
\le 
2\left[
\|\f\ci Q\RH_\mu(\f\ci Q-\chi\ci Q)\|\ci{L^2(\mu)}^2
+\|(\f\ci Q-\chi\ci Q)\RH_\mu\chi\ci Q\|\ci{L^2(\mu)}^2
\right]
\\
\le C[\e^\gamma+\e^{\frac\gamma 2}]\mu(Q)\le C\e^{\frac\gamma 2}\mu(Q)\,.
\end{multline*}
If we now denote
$$
\wt F^{(1)}_k=\left(\sum_{Q\in\QQ_k}\chi\ci Q\RH_\mu\chi\ci Q-\sum_{Q\in\QQ_{k+1}}\chi\ci Q\RH_\mu\chi\ci Q\right)\,,
$$ 
we immediately see that
$$
\|\wt F^{(1)}_k-F^{(1)}_k\|\ci{L^2(\mu)}^2\le C\e^{\frac\gamma 2}\mu(P)\,.
$$
Combined with the estimate $\|G_{k+1}\|\ci{L^2(\mu)}^2\le C\mu(P)$, this yields
$$
|\langle \wt F^{(1)}_k-F^{(1)}_k,G_{k+1}\rangle_\mu|\le \|\wt F^{(1)}_k-F^{(1)}_k\|\ci{L^2(\mu)}
\|G_{k+1}\|\ci{L^2(\mu)}\le C\e^{\frac\gamma 4}\mu(P)\,.
$$

Now we can write
\begin{multline*}
\langle \wt F^{(1)}_k,G_{k+1}\rangle_\mu
\\
=
\sum_{Q\in\QQ_k,Q'\in \QQ_{k+1},Q'\subset Q}\left\langle
\chi\ci Q\RH_\mu\chi\ci Q-\chi\ci {Q'}\RH_\mu\chi\ci {Q'},\f\ci {Q'}\RH(\f\ci {Q'}\mu-\nu\ci{Q'})
\right\rangle_\mu
\end{multline*}
because all other scalar products correspond to pairs of functions with disjoint supports, and, thereby,
evaluate to $0$.

Fix $Q\in\QQ_k$. For each $Q'\in\QQ_{k+1}$ contained in $Q$, we have $\chi\ci Q=\chi\ci{Q'}=1$ on $\supp\f\ci {Q'}$,
so, when writing the scalar product as an integral, we can leave only the factor $\f\ci {Q'}$ in front of the 
product of Riesz transforms, which allows us to combine two of them into one and represent the scalar product as
$$
\langle \RH(\chi\ci{Q\setminus Q'}\mu),\f\ci{Q'}\RH(\f\ci{Q'}\mu-\nu\ci {Q'})\rangle_\mu
$$
The next estimate is worth stating as a separate lemma.
\begin{lem}
\label{collectivemutonuRH}
Suppose that $F$ is any bounded function, and $Q\in \QQ_k$. Then
\begin{multline*}
\sum_{Q'\in\QQ_{k+1},Q'\subset Q}|\langle \RH(\chi\ci{Q\setminus Q'}F\mu),\f\ci{Q'}\RH(\f\ci{Q'}\mu-\nu\ci {Q'})\rangle_\mu|
\\
\le 
C\alpha^{\frac 1{d+2}}
\e^{-3}\|F\|\ci{L^\infty(Q)}\mu(Q)\,.
\end{multline*}
\end{lem}
\begin{proof}
Let $\Psi\ci{Q'}=\RH(\chi\ci{Q\setminus Q'}F\mu)$.
By Lemma \ref{mutonuRH}, we have
\begin{multline*}
|\langle \Psi\ci{Q'}\f\ci{Q'},\RH(\f\ci{Q'}\mu-\nu\ci {Q'})\rangle_\mu|
\\
\le 
C\alpha^{\frac 1{d+2}}
\ell(Q')^{d+2}\left[\|\Psi\ci{Q'}\|\ci{L^\infty(\supp\f\cci{Q'})}
+\ell(Q')\|\Psi\ci{Q'}\|\ci{\Lip(\supp\f\cci{Q'})}\right]\|\f\ci{Q'}\|\ci{\Lip}^2\,.
\end{multline*}
Note now that by (\ref{lipriesz}), 
$$
\|\Psi\ci{Q'}\|\ci{\Lip(\supp\f\cci{Q'})}\le \frac{C\|F\|\ci{L^\infty(Q)}}
{\dist(\supp\f\ci{Q'},Q\setminus Q')}\le\frac {C\|F\|\ci{L^\infty(Q)}}{\e\ell(Q')}
$$
and 
$$
\|\f\ci{Q'}\|\ci{\Lip}\le \frac C{\e\ell(Q')}\,.
$$
Thus, in our case, the bound guaranteed by Lemma \ref{mutonuRH} does not exceed
$$
C\alpha^{\frac 1{d+2}}
\ell(Q')^{d}\e^{-2}\left[\|\Psi\ci{Q'}\|\ci{L^\infty(\supp\f\cci{Q'})}
+\e^{-1}\|F\|\ci{L^\infty(Q)}\right]\,,
$$
so, taking into account that $\ell(Q')^d\le C\mu(Q')$, we get
\begin{multline*}
\sum_{Q'\in\QQ_{k+1},Q'\subset Q}|\langle \RH(\chi\ci{Q\setminus Q'}F\mu),\f\ci{Q'}\RH(\f\ci{Q'}\mu-\nu\ci {Q'})\rangle_\mu|
\\
\le 
C\alpha^{\frac 1{d+2}}\e^{-2}\sum_{Q'\in\QQ_{k+1},Q'\subset Q}
\left[\|\Psi\ci{Q'}\|\ci{L^\infty(\supp\f\cci{Q'})}
+\e^{-1}\|F\|\ci{L^\infty(Q)}\right]\mu(Q')
\\
\le
C\alpha^{\frac 1{d+2}}\e^{-2}\left[
\e^{-1}\|F\|\ci{L^\infty(Q)}\mu(Q)+
\sum_{Q'\in\QQ_{k+1},Q'\subset Q}
\|\Psi\ci{Q'}\|\ci{L^\infty(\supp\f\cci{Q'})}\mu(Q')
\right]
\,.
\end{multline*}
Since the $L^\infty$ norm of a Lipschitz function on a set does not exceed any average 
of the absolute
value of the function over the set plus the product of the Lipschitz
norm of the function on the set and the diameter of the set, we have
\begin{multline*}
\|\Psi\ci{Q'}\|\ci{L^\infty(\supp\f\cci{Q'})}\le C\e^{-1}\|F\|\ci{L^\infty(Q)}+
\left[\left(\int\f\ci{Q'}\,d\mu\right)^{-1}\int |\Psi\ci{Q'}|^2\f\ci{Q'}\,d\mu\right]^{\frac 12}
\\
=C\e^{-1}\|F\|\ci{L^\infty(Q)}+J(Q')
\,.
\end{multline*}
However,
$$
\int\f\ci{Q'}\,d\mu\ge c\ell(Q')^d\ge c\mu(Q')
$$
and
$$
\int |\Psi\ci{Q'}|^2\f\ci{Q'}\,d\mu\le 
2\left[\int_{Q'}|\RH_\mu(F\chi\ci Q)|^2\,d\mu+\int_{Q'}|\RH_\mu(F\chi\ci{Q'})|^2\,d\mu\right]\,.
$$
Since $\RH_\mu$ is bounded in $L^2(\mu)$, we have
$$
\int_{Q'}|\RH_\mu(F\chi\ci{Q'})|^2\,d\mu\le C\|F\chi\ci{Q'}\|^2\ci{L^2(\mu)}\le C\|F\|\ci{L^\infty(Q)}^2\mu(Q')
$$
for each $Q'\subset Q$, and
\begin{multline*}
\sum_{Q'\in\QQ_{k+1},Q'\subset Q}\int_{Q'}|\RH_\mu(F\chi\ci Q)|^2\,d\mu\le
\int_{Q}|\RH_\mu(F\chi\ci Q)|^2\,d\mu
\\
\le 
C\|F\chi\ci{Q}\|^2\ci{L^2(\mu)}\le C\|F\|\ci{L^\infty(Q)}^2\mu(Q)\,.
\end{multline*}
So we get
$$
\sum_{Q'\in\QQ_{k+1},Q'\subset Q}
J(Q')^2 \mu(Q')\le C\|F\|\ci{L^\infty(Q)}^2\mu(Q)\,.
$$
Now it remains to apply Cauchy-Schwarz to conclude that
$$
\sum_{Q'\in\QQ_{k+1},Q'\subset Q}
J(Q')\mu(Q')\le C\|F\|\ci{L^\infty(Q)}\mu(Q)
$$ 
thus completing the proof of the lemma.
\end{proof}

Applying this lemma with $F=1$, we immediately get 
$$
\sum_{Q'\in\QQ_{k+1},Q'\subset Q}
|\langle \RH(\chi\ci{Q\setminus Q'}\mu),\f\ci{Q'}\RH(\f\ci{Q'}\mu-\nu\ci {Q'})\rangle_\mu|\le 
C\alpha^{\frac 1{d+2}}
\e^{-3}
\mu(Q)\,.
$$
It remains to sum these bounds over $Q\in\QQ_k$ and to combine the result with the previously obtained
estimate for $\langle \wt F^{(1)}_k-F^{(1)}_k,G_{k+1}\rangle_\mu$ to conclude that
$$
|\langle F_k^{(1)},G_{k+1}\rangle_\mu|\le C(\e^{\frac\gamma 4}+\alpha^{\frac 1{d+2}}
\e^{-3})\mu(P)\,.
$$
To estimate $\langle F_k^{(2)},G_{k+1}\rangle_\mu$, note once more that by Lemma \ref{smoothtolip}, $\RH\nu\ci Q$
is a Lipschitz function in $\R^{d+1}$ with $\|\RH\nu\ci Q\|\ci{L^\infty}\le\frac C{\e^2}$ and
$\|\RH\nu\ci Q\|\ci{\Lip}\le\frac C{\e^2\ell(Q)}$.
Since for any two Lipschitz functions $f,g$,
$$
\|fg\|\ci{L^\infty}\le\|f\|\ci{L^\infty}\|g\|\ci{L^\infty}
\text{ and }
\|fg\|\ci{\Lip}\le \|f\|\ci{\Lip}\|g\|\ci{L^\infty}+\|f\|\ci{L^\infty}\|g\|\ci{\Lip}\,,
$$
we get 
$$
\|\f\ci Q\RH\nu\ci Q\|\ci{L^\infty}\le\frac C{\e^2}\text{ and }
\|\f\ci Q\RH\nu\ci Q\|\ci{\Lip}\le\frac C{\e^3\ell(Q)}\,.
$$
Using Lemma \ref{mutonuRH} again and taking into account that $\ell(Q')\le\ell(Q)$ for $Q'\subset Q$, 
we get
\begin{multline*}
|\langle \f\ci Q\RH\nu\ci Q,\f\ci{Q'}\RH(\f\ci{Q'}\mu-\nu\ci {Q'})\rangle_\mu|
\\
\le
C\alpha^{\frac 1{d+2}}
\ell(Q')^{d+2}\left[\frac 1{\e^2}
+\ell(Q')\frac 1{\e^3\ell(Q)}\right]\left(\frac 1{\e\ell(Q')}\right)^2
\\
\le C\alpha^{\frac 1{d+2}}\e^{-5}
\ell(Q')^{d}\le C\alpha^{\frac 1{d+2}}\e^{-5}\mu(Q')\,.
\end{multline*} 
Writing $\langle F_k^{(2)},G_{k+1}\rangle_\mu$
as
$$
\sum_{Q\in\QQ_K,Q'\in\QQ_{k+1},Q'\subset Q}\langle \f\ci Q\RH\nu\ci Q,\f\ci{Q'}\RH(\f\ci{Q'}\mu-\nu\ci {Q'})\rangle_\mu
$$
(all other scalar products correspond to functions with disjoint supports)
and summing the corresponding upper bounds for the absolute values of summands,
we get 
$$
|\langle F_k^{(2)},G_{k+1}\rangle_\mu|\le C\alpha^{\frac 1{d+2}}\e^{-5}\mu(P)\,.
$$
Finally, we can
write $\langle F_k^{(3)},G_{k+1}\rangle_\mu$
as
$$
\sum_{Q'\in\QQ_{k+1}}\langle \f\ci{Q'}\RH\nu \ci{Q'},\f\ci{Q'}\RH(\f\ci{Q'}\mu-\nu\ci {Q'})\rangle_\mu\,.
$$
The argument we used to estimate $\langle F_k^{(2)},G_{k+1}\rangle_\mu$ can be applied here as well.
The only essential difference is that we will now have $\ell(Q')$ instead of $\ell(Q)$ in the denominator
of the bound for $\|\f\ci{Q'}\RH\nu \ci{Q'}\|\ci{\Lip}$, so instead of the lax cancellation 
$\frac{\ell(Q')}{\ell(Q)}\le 1$ in the main bound for individual summands, 
we will have to use the tight cancellation $\frac{\ell(Q')}{\ell(Q')}=1$. The final inequality
$$
|\langle F_k^{(3)},G_{k+1}\rangle_\mu|\le C\alpha^{\frac 1{d+2}}\e^{-5}\mu(P)
$$
has exactly the same form and it remains to bring all three inequalities together to finish
the proof of the desired almost orthogonality property with $\sigma(\e, \alpha) = C[\e^{\frac\gamma 4}+\alpha^{\frac 1{d+2}}\e^{-5}]$.
\end{proof}

\section{Reduction to the lower bound for $\|F_k\|^2\ci{L^2(\mu)}$}
\label{reduction}

At this point, we need to know that the non-BAUPness condition depends on 
a positive parameter $\delta$. We will fix that $\delta$ from now on in addition
to fixing the measure $\mu$. Note that despite the fact that we need to prove that the family
of non-BAUP cells is Carleson for every $\delta>0$, the David-Semmes uniform rectifiability
criterion does not require any particular rate of growth of the 
corresponding Carleson constant as a function of $\delta$.   

We have the identity
$$
\|G_0\|\ci{L^2(\mu)}^2=\left\|\sum_{k=0}^K F_k\right\|\ci{L^2(\mu)}^2=
\sum_{k=0}^K \|F_k\|\ci{L^2(\mu)}^2+2\sum_{k=0}^{K-1}\langle F_k,G_{k+1}\rangle_\mu\,.
$$
As we have seen, $\|G_0\|\ci{L^2(\mu)}^2\le C\mu(P)$ under the conditions of Proposition \ref{scalarproducts}
and the scalar products can be made arbitrarily small by first choosing $\e>0$ small enough and then taking 
a sufficiently small $\alpha>0$ depending on $\e$. So we will get a contradiction if we are able to
bound  $\|F_k\|\ci{L^2(\mu)}^2$ for $k=0,\dots,K-1$ from below by $\tau^2\mu(P)$, 
with some $\tau=\tau(\delta)>0$ (as usual,
the dependence on the dimension $d$ and the regularity constants of $\mu$ is suppressed) 
under the assumptions that $A>A_0(\delta),\e<\e_0(\delta), \eta<\eta_0(\e), \alpha<\alpha_0(\e,\delta)$.
We will call any set of such bounds ``{\em restrictions of admissible type}''. Note that we may impose
any finite number of such restrictions and we will still be able to choose
 some positive values of parameters to satisfy all of them.  

Assuming that we have this lower bound, we will start with choosing $K$ so that $K\tau^2$ is very large. Then we will fix 
$A>A_0(\delta)$ and choose $\e<\e_0(\delta)$ and $\alpha<\alpha_0(\e,\delta)$ in this order to make sure
that the sum of the scalar products is significantly less than $K\tau^2$, for which it would suffice to
make each individual scalar product much less than $\tau^2$. If we are allowed to choose $\e$ first
and $\alpha$ afterwards, the restrictions 
$\e<\e_0(\delta), \alpha<\alpha_0(\e,\delta)$ can never cause us any trouble. Finally, we can choose 
$\eta<\eta_0(\e)$, thus completing the formal choice of parameters.

Since the constructions of Sections \ref{abundanceofflats}--\ref{alternatinglayers} can be carried out
with any choices of $K,A,\alpha,\eta$ under the only assumption that the family of non-BAUP cells
is not Carleson, we will end up with a clear contradiction. 

The proof of the uniform lower bound for $\|F_k\|\ci{L^2(\mu)}^2$ is rather long and technical and will
be done in several steps. We shall start with an elementary reduction that will allow us to restrict
our attention to a single cell $Q\in\QQ_k$ that is tiled with its subcells $Q'\in\QQ_{k+1}$ almost completely.

\section{Densely and loosely packed cells}
\label{denselypackedcells}

Fix $k\in\{0,1,\dots,K-1$\}. We can write the function $F_k$ as
$$
F_k=\sum_{Q\in\QQ_k}F^{Q}
$$
where 
$$
F^Q=\f\ci Q\RH(\f\ci Q\mu-\nu\ci Q)-\sum_{Q'\in\QQ_{k+1},Q'\subset Q}\f\ci {Q'}\RH(\f\ci{Q'}\mu-\nu\ci{Q'})\,.
$$ 
We shall call a cell $Q\in\QQ_k$ densely packed if $\sum_{Q'\in\QQ_{k+1},Q'\subset Q}\mu(Q')\ge (1-\e)\mu(Q)$.
Otherwise we shall call the cell $Q$ loosely packed. The main claim of this section is that the loosely
packed cells constitute a tiny minority of all cells in $\QQ_k$ if $\eta\le \e^2$. Indeed, we have
\begin{multline*}
\sum_{\substack{Q\in\QQ_{k}\\
Q \text{ is packed loosely} }}\mu(Q)\le \e^{-1}\sum_{Q\in\QQ_k}\mu\left(
Q\setminus\left(\bigcup_{Q'\in\QQ_{k+1},Q'\subset Q}Q'\right)
\right)
\\
=\e^{-1}\left[\sum_{Q\in\QQ_{k}}\mu(Q)-\sum_{Q'\in\QQ_{k+1}}\mu(Q')\right]
\\
\le
\e^{-1}\left[\mu(P)-\sum_{Q'\in\QQ_{k+1}}\mu(Q')\right]\le
\frac{\eta}{\e}\mu(P)\le \e\mu(P)\,.
\end{multline*}
We can immediately conclude from here that 
\begin{multline*}
\sum_{\substack{Q\in\QQ_{k}\\
Q \text{ is densely packed } }}\mu(Q)
=\sum_{Q\in\QQ_{k}}\mu(Q)-\sum_{\substack{Q\in\QQ_{k}\\
Q \text{ is loosely packed } }}\mu(Q)
\\
\ge (1-\eta)\mu(P)-\e\mu(P)\ge (1-2\e)\mu(P)\,.
\end{multline*}
From now on, we will fix the choice $\eta=\e^2$.
We claim now that to estimate $\|F_k\|\ci{L^2(\mu)}^2$ from below by $\tau^2\mu(P)$, it suffices to show 
that for every densely packed cell $Q\in\QQ_k$, we have $\|F^Q\|\ci{L^2(\mu)}^2\ge 2\tau^2\mu(Q)$.
To see it, just write
\begin{multline*}
\|F_k\|\ci{L^2(\mu)}^2=\sum_{Q\in\QQ_k}\|F^Q\|\ci{L^2(\mu)}^2
\ge \sum_{\substack{Q\in\QQ_{k}\\
Q \text{ is  densely packed } }}\|F^Q\|\ci{L^2(\mu)}^2
\\
\ge
\sum_{\substack{Q\in\QQ_{k}\\
Q \text{ is  densely packed} }}2\tau^2\mu(Q)\ge 2(1-2\e)\tau^2\mu(P)\ge \tau^2\mu(P)\,,
\end{multline*}
provided that $\e<\frac 14$.

\section{Approximating measure}
\label{approximatingmeasure}

From now on, we will fix $k\in\{0,\dots,K-1\}$ and a densely packed cell $Q\in\QQ_k$.
We denote by $\QQ$ the set of all cells $Q'\in\QQ_{k+1}$ that are contained in the cell $Q$.
We will also always assume that the assumptions of Proposition \ref{scalarproducts} are satisfied.

\goal
The goal of this section is to show that there exists a subset $\QQ'$ of $\QQ$
such that $\sum_{Q'\in\QQ'}\mu(Q')\ge (1-C\e)\mu(Q)$ and
$$
\|F^Q\|\ci{L^2(\mu)}\ge \frac 12\|\RH(\nu-\nu\ci Q)\|\ci{L^2(\nu)} -\sigma(\e,\alpha)\sqrt{\mu(Q)}\,,
$$
where $\nu=\sum_{Q'\in\QQ'}\nu\ci{Q'}$ and $\sigma(\e,\alpha)$ is some positive function such that
$\lim_{\e\to 0+}[\lim_{\alpha\to 0+}\sigma(\e,\alpha)]=0$. 
\goalend

\begin{proof}
The proof is fairly long and technical, so we will split it into several steps.

\medskip

\leftline{\textbf{Step 1.} \textit{The choice of $\QQ'$.}}

\medskip
 
For $Q'\in\QQ$, define
$$
g(Q')=\sum_{Q''\in\QQ}\left[\frac{\ell(Q'')}{D(Q',Q'')}\right]^{d+1}
$$
where 
$$
D(Q',Q'')=\ell(Q')+\ell(Q'')+\dist(Q',Q'')
$$
is the ``long distance'' between $Q'$ and $Q''$.

We have
\begin{multline*}
\sum_{Q'\in\QQ} g(Q')\mu(Q')=\sum_{Q',Q''\in\QQ}\ell(Q'')^{d+1}\frac{\mu(Q')}{D(Q',Q'')^{d+1}}
\\
\le C\sum_{Q',Q''\in\QQ}\ell(Q'')^{d+1}\int_{Q'}\frac{d\mu(x)}{[\ell(Q'')+\dist(x,Q'')]^{d+1}}
\\
\le C\sum_{Q''\in\QQ}\ell(Q'')^{d+1}\int\frac{d\mu(x)}{[\ell(Q'')+\dist(x,Q'')]^{d+1}}
\\
\le C\sum_{Q''\in\QQ}\ell(Q'')^d\le C\sum_{Q''\in\QQ}\mu(Q'')\le C\mu(Q)\,.
\end{multline*}
Let $\QQ_*=\{Q'\in\QQ:g(Q')>\e^{-1}\}$, $\QQ'=\QQ\setminus\QQ_*$.
Then, by Chebyshev's inequality, 
$$
\sum_{Q'\in\QQ_*}\mu(Q')\le C\e\mu(Q)\,,
$$
so
$$
\sum_{Q'\in\QQ'}\mu(Q')\ge (1-C\e)\mu(Q)\,,
$$
as required.

Put 
$$
\Phi=\sum_{Q'\in\QQ'}\f\ci{Q'},\qquad \wt\Phi=\sum_{Q'\in\QQ'}\chi\ci{Q'_\e}\,.
$$

\medskip

\leftline{\textbf{Step 2.} \textit{The first modification of $F^Q$: from $\f$ to $\chi$.}}

\medskip

Our next aim will be to show that
$$
\|F^Q\|\ci{L^2(\mu)}\ge \|\wt F^Q\|\ci{L^2(\mu)}-C\e^{\frac\gamma 4}\sqrt{\mu(Q)}
$$ 
where
$$
\wt F^Q=\wt\Phi\RH_\mu\Phi-\sum_{Q'\in\QQ'}\chi\ci{Q'_\e}\RH_\mu\f\ci{Q'}
+\sum_{Q'\in\QQ'}\chi\ci{Q'_\e}\RH\nu\ci{Q'}-\chi\ci Q\RH\nu\ci Q\,.
$$

Recall first that, by Lemma \ref{summandbound}, we have
$$
\|\f\ci{Q'}\RH(\f\ci{Q'}\mu-\nu\ci{Q'})\|^2\ci{L^2(\mu)}\le 
\|\RH(\f\ci{Q'}\mu-\nu\ci{Q'})\|^2\ci{L^2(\f\cci{Q'}\mu)}\le
C\mu(Q')
$$
for all $Q'\in\QQ$. Thus
\begin{multline*}
\left\|\sum_{Q'\in\QQ_*}\f\ci{Q'}\RH(\f\ci{Q'}\mu-\nu\ci{Q'})\right\|^2\ci{L^2(\mu)}
=\sum_{Q'\in\QQ_*}\|\f\ci{Q'}\RH(\f\ci{Q'}\mu-\nu\ci{Q'})\|^2\ci{L^2(\mu)}
\\
\le C\sum_{Q'\in\QQ_*}\mu(Q')\le C\e\mu(Q)\,.
\end{multline*}
This allows us to drop the terms $\f\ci{Q'}\RH(\f\ci{Q'}\mu-\nu\ci{Q'})$ corresponding to $Q'\in\QQ_*$
in the definition of $F^Q$ at the cost of decreasing the $L^2(\mu)$ norm by at most $C\e^{\frac 12}\sqrt{\mu(Q)}$.

Next we bound the norm $\|\f\ci Q\RH_\mu\f\ci Q-\wt\Phi\RH_\mu\Phi\|\ci{L^2(\mu)}$.
First, note that for $p\ge 1$, we have 
\begin{multline}
\label{phiqphi}
\|\f\ci Q-\Phi\|\ci{L^p(\mu)}^p\\ 
\le
\mu(Q\setminus Q_{3\e})+\mu(Q\setminus(\cup_{Q'\in\QQ'}Q'))+\mu
(\cup_{Q'\in\QQ'}(Q'\setminus Q'_{3\e}))
\\
\le C\e^\gamma\mu(Q)+C\e\mu(Q)+C\e^\gamma\sum_{Q'\in\QQ'}\mu(Q')\le C\e^\gamma\mu(Q)\,,
\end{multline}
and the same estimate holds for $\|\f\ci Q-\wt\Phi\|\ci{L^p(\mu)}^p$.
Using the boundedness of $\RH_\mu$ in $L^p(\mu)$ for $p=2,4$, we get
$$
\|\f\ci Q\RH_\mu(\f\ci Q-\Phi)\|\ci{L^2(\mu)}^2\le C\|\f\ci Q-\Phi\|\ci{L^2(\mu)}^2\le C\e^\gamma\mu(Q)
$$
and
\begin{multline*}
\|(\f\ci Q-\wt\Phi)\RH_\mu\Phi\|\ci{L^2(\mu)}^2
\le \|\f\ci Q-\wt\Phi\|\ci{L^4(\mu)}^2\|\RH_\mu\Phi\|\ci{L^4(\mu)}^2
\\
\le C\|\f\ci Q-\wt\Phi\|\ci{L^4(\mu)}^2\|\Phi\|\ci{L^4(\mu)}^2
\le C\e^{\frac\gamma 2}\mu(Q)\,.
\end{multline*}
Bringing these two estimates together and using the triangle inequality, we get
$$
\|\f\ci Q\RH_\mu\f\ci Q-\wt\Phi\RH_\mu\Phi\|\ci{L^2(\mu)}\le C\e^{\frac\gamma 4}\sqrt{\mu(Q)}\,.
$$
This allows us to replace the term $\f\ci{Q}\RH(\f\ci{Q}\mu)$
in the definition of $F^Q$ by the term $\wt\Phi\RH_\mu\Phi$ appearing in the definition of $\wt F^{Q}$
at the cost of decreasing the $L^2(\mu)$ norm by at most $C\e^{\frac\gamma 4}\sqrt{\mu(Q)}$.

Next note that for every $Q'\in\QQ'$, we have
$$
\|\chi\ci{Q'}\RH(\f\ci{Q'}\mu-\nu \ci{Q'})\|\ci{L^4(\mu)}^4\le C\mu(Q')
$$
by Lemma \ref{summandbound}, so 
\begin{multline*}
\|(\f\ci{Q'}-\chi\ci{Q'_\e})\RH(\f\ci{Q'}\mu-\nu\ci{Q'})\|\ci{L^2(\mu)}^2
\\
\le \|\f\ci{Q'}-\chi\ci{Q'_\e}\|\ci{L^4(\mu)}^2\|\chi\ci{Q'}\RH(\f\ci{Q'}\mu-\nu\ci{Q'})\|\ci{L^4(\mu)}^2
\\
\le C\mu(Q'\setminus Q'_{3\e})^{\frac 12}\mu(Q')^{\frac 12}
\le C\e^{\frac\gamma 2}\mu(Q')\,.
\end{multline*}
Thus,
\begin{multline*}
\left\|\sum_{Q'\in\QQ'}(\f\ci{Q'}-\chi\ci{Q'_\e})\RH(\f\ci{Q'}\mu-\nu\ci{Q'})\right\|\ci{L^2(\mu)}^2
\\
=\sum_{Q'\in\QQ'}\|(\f\ci{Q'}-\chi\ci{Q'_\e})\RH(\f\ci{Q'}\mu-\nu\ci{Q'})\|\ci{L^2(\mu)}^2
\\
\le C\e^{\frac\gamma 2}\sum_{Q'\in\QQ'}\mu(Q')\le C\e^{\frac\gamma 2}\mu(Q)\,.
\end{multline*}
This allows us to replace all the remaining terms $\f\ci{Q'}\RH(\f\ci{Q'}\mu-\nu\ci{Q'})$ ($Q'\in\QQ'$)
in the definition of $F^Q$ by the terms $\chi\ci{Q'_\e}\RH(\f\ci{Q'}\mu-\nu\ci{Q'})$ 
appearing in the definition of $\wt F^{Q}$
at the cost of decreasing the $L^2(\mu)$ norm by at most $C\e^{\frac\gamma 4}\sqrt{\mu(Q)}$ again.

At last, using the bound $\|\chi\ci Q\RH\nu\ci Q\|\ci{L^4(\mu)}^4\le C\mu(Q)$ (the same Lemma \ref{summandbound}),
we get 
\begin{multline*}
\|(\f\ci Q-\chi\ci Q)\RH\nu\ci Q\|\ci{L^2(\mu)}^2
\le \|\f\ci{Q}-\chi\ci{Q}\|\ci{L^4(\mu)}^2\|\chi\ci{Q}\RH\nu\ci{Q})\|\ci{L^4(\mu)}^2
\\
\le C\mu(Q\setminus Q_{3\e})^{\frac 12}\mu(Q)^{\frac 12}
\le C\e^{\frac\gamma 2}\mu(Q)\,.
\end{multline*}
So, we can make the final replacement of $\f\ci Q\RH\nu\ci Q$ with $\chi\ci Q \RH\nu\ci Q$ at the cost of decreasing
the $L^2(\mu)$ norm by at most $C\e^{\frac\gamma 4}\sqrt{\mu(Q)}$.

\medskip

\leftline{\textbf{Step 3.} \textit{The second modification of $F^Q$: from $\RH(\f\mu)$ to $\RH\nu$.}}

\medskip

Recall that we finally want to switch from $\mu$ to the measure 
$$
\nu=\sum_{Q'\in\QQ'}\nu\ci{Q'}\,.
$$
Our next goal will be to show that
$$
\|\wt F^Q-(\wt\Phi\RH\nu-\chi\ci Q\RH\nu\ci Q)\|\ci{L^2(\mu)}\le C\alpha\e^{-d-3}\sqrt{\mu(Q)}
$$

Note first of all that
$$
\wt\Phi\RH_\mu\Phi-\sum_{Q'\in\QQ'}\chi\ci{Q'_\e}\RH_\mu\f\ci{Q'}=
\sum_{Q'\in\QQ'}\chi\ci{Q'_\e}\RH(\Phi\ci{Q'}\mu)
$$
where 
$$
\Phi\ci{Q'}=\sum_{Q''\in\QQ':Q''\ne Q'}\f\ci{Q''}\,.
$$
Fix some $Q'\in\QQ'$. Let $x\in Q'_\e$. Then, for every $Q''\in\QQ'\setminus\{Q'\}$, we have
$$
[\RH(\f\ci{Q''}\mu-\nu\ci{Q''})](x)=\int\Psi_x\,d(\f\ci{Q''}\mu-\nu\ci{Q''})
$$
where
$$
\Psi_x(y)=K^H(x-y)=\frac{\pi\ci H(x-y)}{|x-y|^{d+1}}\,.
$$
Since $|x-y|\ge\e D(Q',Q'')$ for all $y\in\supp\f\ci{Q''}\subset Q''_\e$, we have
$$
\|\Psi_x\|\ci{\Lip(\supp\f\cci{Q''})}\le\frac{C}{\e^{d+1}D(Q',Q'')^{d+1}}
$$ 
whence, by Lemma \ref{mutonulip}, 
\begin{multline*}
\left|\int\Psi_x\,d(\f\ci{Q''}\mu-\nu\ci{Q''})\right|\le 
C\alpha\ell(Q'')^{d+2}\|\Psi_x\|\ci{\Lip(\supp\f\cci{Q''})}\|\f\ci{Q''}\|\ci{\Lip}
\\
\le C\alpha\e^{-d-2}\left[\frac{\ell(Q'')}{D(Q',Q'')}\right]^{d+1}\,.
\end{multline*}
Therefore, for every $Q'\in\QQ'$, we have
$$
\left|\RH(\Phi\ci{Q'}\mu)-\sum_{Q''\in\QQ',Q''\ne Q'}\RH\nu\ci{Q''}\right|\le 
C\alpha\e^{-d-2}g(Q')\le C\alpha\e^{-d-3}
$$
on $Q'\ci e$.
Thus, making a uniform error of at most $C\alpha\e^{-d-3}$, we can replace 
$$
\wt\Phi\RH_\mu\Phi-\sum_{Q'\in\QQ'}\chi\ci{Q'_\e}\RH_\mu\f\ci{Q'}=\sum_{Q'\in\QQ'}\chi\ci{Q'_\e}\RH(\Phi\ci{Q'}\mu)
$$
with
$$
\sum_{Q'\in\QQ'}\chi\ci{Q'_\e}\left(\sum_{Q''\in\QQ',Q''\ne Q'}\RH\nu\ci{Q''}\right)=
\sum_{Q'\in\QQ'}\chi\ci{Q'_\e}\RH(\nu-\nu\ci{Q'})\,.
$$
Combining each term in this sum with the corresponding term $\chi\ci{Q'_\e}\RH\nu\ci{Q'}$,
we get the sum
$$
\sum_{Q'\in\QQ'}\chi\ci{Q'_\e}\RH\nu=\wt\Phi\RH\nu\,.
$$
It remains to note that the uniform bound we got is stronger than the $L^2(\mu)$ bound we need.

\medskip

\leftline{\textbf{Step 4.} \textit{The final effort: from $L^2(\mu)$ to $L^2(\nu)$.}}

\medskip

It remains to compare $\|\wt\Phi\RH\nu-\chi\ci Q\RH\nu\ci Q\|\ci{L^2(\mu)}$ with
$\|\RH(\nu-\nu\ci Q)\|\ci{L^2(\nu)}$. Since $0\le\Phi\le 1$ and both $\wt\Phi$ and $\chi\ci Q$ are identically
equal to $1$ on $\supp\Phi$, we trivially have
$$
\|\wt\Phi\RH\nu-\chi\ci Q\RH\nu\ci Q\|\ci{L^2(\mu)}\ge 
\|\wt\Phi\RH\nu-\chi\ci Q\RH\nu\ci Q\|\ci{L^2(\Phi\mu)}=
\|\RH(\nu-\nu\ci Q)\|\ci{L^2(\Phi\mu)}\,.
$$
To make the transition from $L^2(\Phi\mu)$ to $L^2(\nu)$, we will use the following comparison lemma.

\begin{lem}
\label{comparisonlemma}
Let $F$ be any Lipschitz function and let $p\ge 1$. Then
$$
\left|\int |F|^p\,d(\Phi\mu-\nu)\right|\le C(p)\alpha\e^{-1}
\left[\|F\|\ci{L^p(\Phi\mu)}^p+[\max_{Q'\in\QQ'}\ell(Q')\|F\|\ci{\Lip(\supp\f\cci{Q'})}]^p\mu(Q)\right]\,.
$$
\end{lem}

\begin{proof}
Denote $M=\max_{Q'\in\QQ'}\ell(Q')\|F\|\ci{\Lip(\supp\f\cci{Q'})}$, $S(Q')=\|F\|\ci{L^\infty(\supp\f\cci{Q'})}$.
We have
$$
\int |F|^p\,d(\Phi\mu-\nu)=\sum_{Q'\in\QQ'}\int |F|^p\,d(\f\ci{Q'}\mu-\nu\ci{Q'})\,.
$$
By Lemma \ref{mutonulip},
\begin{multline*}
\left|\int |F|^p\,d(\f\ci{Q'}\mu-\nu\ci{Q'})\right|\le 
C(p)\alpha\ell(Q')^{d+2}S(Q')^{p-1}\frac{M}{\ell(Q')}\|\f\ci{Q'}\|\ci{\Lip}
\\
\le C(p)\alpha\e^{-1}S(Q')^{p-1}M\ell(Q')^d\le C(p)\alpha\e^{-1}S(Q')^{p-1}M\mu(Q')
\\
\le
C(p)\alpha\e^{-1}[S(Q')^p+M^p]\mu(Q')
\end{multline*}
for each $Q'\in\QQ'$. Thus,
\begin{multline*}
\left|\int |F|^p\,d(\Phi\mu-\nu)\right|\le C(p)\alpha\e^{-1}\sum_{Q'\in\QQ'}[S(Q')^p+M^p]\mu(Q')
\\
\le C(p)\alpha\e^{-1}\left[M^p\mu(Q)+\sum_{Q'\in\QQ'}S(Q')^p\mu(Q')\right]\,.
\end{multline*}
It remains to note that, for each $Q'\in\QQ'$, we have $\int\f\ci{Q'}\,d\mu\ge c\ell(Q')^d\ge c\mu(Q')$ and
$$
S(Q')^p\le \left[\min_{\supp\f\cci{Q'}}|F|+8\ell(Q')\|F\|\ci{\Lip(\supp\f\cci{Q'})}\right]^p
\le C(p)\left[\left(\min_{\supp\f\cci{Q'}}|F|\right)^p+M^p\right]\,,
$$ 
so 
\begin{multline*}
\sum_{Q'\in\QQ'}S(Q')^p\mu(Q')\le C(p)\sum_{Q'\in\QQ'}\left(\min_{\supp\f\cci{Q'}}|F|\right)^p\mu(Q')+
C(p)M^p\sum_{Q'\in\QQ'}\mu(Q')
\\
\le C(p)\sum_{Q'\in\QQ'}\left(\min_{\supp\f\cci{Q'}}|F|\right)^p\int\f\ci{Q'}\,d\mu
+C(p)M^p\mu(Q)
\\
\le C(p)\int |F|^p\,d(\Phi\mu)+C(p)M^p\mu(Q)\,.
\end{multline*}
\end{proof}
Thus, we need to get a decent bound for the Lipschitz norm of $\RH(\nu-\nu\ci Q)$ on $\supp\f\ci{Q'}$. We already know (Lemma \ref{smoothtolip})
that $\|\RH\nu\ci Q\|\ci{\Lip}\le \frac C{\e^2\ell(Q)}\le \frac C{\e^2\ell(Q')}$ and 
$\|\RH\nu\ci{Q'}\|\ci{\Lip}\le \frac C{\e^2\ell(Q')}$.
Now note that
$$
\RH(\nu-\nu\ci{Q'})=\sum_{Q''\in\QQ',Q''\ne Q'}\int \Psi_y\,d\nu\ci{Q''}(y)
$$
where $\Psi_y(x)=K^H(x-y)$. Since for every $x\in\supp\f\ci{Q'}$ and every $y\in\supp\f\ci{Q''}$, we 
have $|x-y|\ge\e D(Q',Q'')$, we have 
$$
\|\Psi_y\|\ci{\Lip(\supp\f\cci{Q'})}\le \frac{C}{\e^{d+1}D(Q',Q'')^{d+1}}
$$  
for all $y\in\supp\nu\ci{Q''}$. Thus
\begin{multline*}
\|\RH(\nu-\nu\ci{Q'})\|\ci{\Lip(\supp\f\cci{Q'})}\le 
\sum_{Q''\in\QQ',Q''\ne Q'}\int \|\Psi_y\|\ci{\Lip(\supp\f\cci{Q'})}\,d\nu\ci{Q''}(y)
\\
\le C\sum_{Q''\in\QQ',Q''\ne Q'}\frac{\nu\ci{Q''}(Q'')}{\e^{d+1}D(Q',Q'')^{d+1}}
= C\sum_{Q''\in\QQ',Q''\ne Q'}\frac{\mu(Q'')}{\e^{d+1}D(Q',Q'')^{d+1}}
\\
\le C\e^{-d-1}\int\frac{d\mu(y)}{[\ell(Q')+\dist(y,Q')]^{d+1}}\le C\e^{-d-1}\ell(Q')^{-1}\,.
\end{multline*}
Bringing these three estimates together, we conclude that
$$
\ell(Q')\|\RH(\nu-\nu\ci{Q})\|\ci{\Lip(\supp\f\cci{Q'})}\le C\e^{-d-1}
$$
for all $Q'\in\QQ'$. Lemma \ref{comparisonlemma} applied with $p=2$ and $F=\RH(\nu-\nu\ci{Q})$ now yields
\begin{multline*}
\left|\int |\RH(\nu-\nu\ci{Q})|^2\,d(\Phi\mu-\nu)\right|
\\
\le C\alpha\e^{-1}
\left[\|\RH(\nu-\nu\ci{Q})\|\ci{L^2(\Phi\mu)}^2+[C\e^{-d-1}]^2\mu(Q)\right]
\\
\le C\alpha\e^{-1}
\|\RH(\nu-\nu\ci{Q})\|\ci{L^2(\Phi\mu)}^2+ C\alpha\e^{-2d-3}\mu(Q)\,,
\end{multline*}
whence
\begin{multline*}
\|\RH(\nu-\nu\ci{Q})\|\ci{L^2(\nu)}^2\le (1+C\alpha\e^{-1})\|\RH(\nu-\nu\ci{Q})\|\ci{L^2(\Phi\mu)}^2 
+C\alpha\e^{-2d-3}\mu(Q)
\\
\le (1+C\alpha\e^{-1})
\left[\|\RH(\nu-\nu\ci{Q})\|\ci{L^2(\Phi\mu)}+C\alpha^{\frac 12}\e^{-\frac{2d+3}2}\sqrt{\mu(Q)}\right]^2\,.
\end{multline*}
Assuming that $C\alpha\e^{-1}<3$, which is a restriction of the type $\alpha<\alpha_0(\e)$, and taking the 
square root, we finally get
$$
\|\RH(\nu-\nu\ci{Q})\|\ci{L^2(\Phi\mu)}\ge \frac 12\|\RH(\nu-\nu\ci{Q})\|\ci{L^2(\nu)}-
C\alpha^{\frac 12}\e^{-\frac{2d+3}2}\sqrt{\mu(Q)}\,.
$$
Combined with the bounds from Steps 2--3, this yields the statement of the lemma with
$$
\sigma(\e,\alpha)=C[\e^{\frac\gamma 4}+\alpha^{\frac 12}\e^{-\frac{2d+3}2}+\alpha\e^{-d-3}]\,.
$$
\end{proof}

\section{Reflection trick}
\label{reflectiontrick}

Fix a hyperplane $L$ parallel to $H$ at the distance $2\Delta\ell(Q)$ from $\supp\mu\cap Q$.
The reader should think of $\Delta$ as small compared to $\e$ and large compared to $\alpha$.
Let $S$ be the (closed) half-space bounded by $L$ that contains $\supp\mu\cap Q$. For $x\in S$, denote by
$x^*$ the reflection of $x$ about $L$. Define the kernels
$$
\wt K^H(x,y)=K^H(x-y)-K^H(x^*-y),\qquad x,y\in S
$$
and denote by $\wtRH$ the corresponding operator. We will assume that $\alpha<\Delta$,
so the approximating hyperplanes $L(Q')$ ($Q'\in\QQ'$) and $L(Q)$, which lie within the
distance $\alpha\ell(Q)$ from $\supp\mu\cap Q$ are contained in $S$ and lie at the distance
$\Delta\ell(Q)$ or greater from the boundary hyperplane $L$. 

\goal
The goal of this section is to show that, for some appropriately chosen $\Delta=\Delta(\alpha,\e)>0$,
and under our usual assumptions about $\e$, $A$, and $\alpha$, we have
$$
\|\RH(\nu-\nu\ci Q)\|\ci{L^2(\nu)}\ge \|\wtRH\nu\|\ci{L^2(\nu)}-\sigma(\e,\alpha)\sqrt{\mu(Q)}
$$
where, again, $\sigma(\e,\alpha)$ is some positive function such that
$$
\lim_{\e\to 0+}[\lim_{\alpha\to 0+}\sigma(\e,\alpha)]=0\,.
$$
\goalend

Thus, if $\|\RH(\nu-\nu\ci Q)\|\ci{L^2(\nu)}$ is much smaller than $\sqrt{\mu(Q)}$ 
and $\e$ and $\alpha$ are chosen so that $\sigma(\e,\alpha)$ is small, 
then $\|\wtRH\nu\|\ci{L^2(\nu)}$ must also be small. Again, the exact 
formula for $\sigma(\e,\alpha)$ is not important for the rest of the
argument.
 
Note that the correction kernel
$K^H(x^*-y)$ is uniformly bounded as long as $x$ or $y$ stay in $S$ away from the 
boundary hyperplane $L$, so it defines a nice bounded operator 
in $L^2(\mu\ci Q)$, where $\mu\ci Q=\chi\ci Q\mu$, and we can define the operator
$\wtRH_{\mu\cci Q}$ with the kernel $\wt K^H(x,y)$ as the difference of the operator
$\RH_{\mu\cci{Q}}$ and the integral operator $T$ with the kernel $K^H(x^*-y)$. 

Our first observation is that the norm of the operator $\wtRH_{\mu\cci Q}$ in $L^2(\mu\ci Q)$
is bounded by some constant depending only on the dimension and the goodness 
parameters of $\mu$. Indeed, all we need is to bound the norm of the integral
operator $T$. Note however that 
$$
K^H(x^*-y)=K_{\Delta\ell(Q)}^H(x-y)+[K^H(x^*-y)-K_{\Delta\ell(Q)}^H(x-y)]\,.
$$ 
The first term on the right corresponds to the operator $\RH_{\mu\cci Q,\Delta\ell(Q)}$, whose norm
is bounded by some constant independent of $\Delta$ according to our definition of a good
measure. On the other hand, we have
$$
|K^H(x^*-y)-K_{\Delta\ell(Q)}^H(x-y)|\le \frac{C\Delta\ell(Q)}{[\Delta\ell(Q)+|x-y|]^{d+1}}
$$
for all $x,y\in S$ with $\dist(x,L),\dist(y,L)\in(\Delta\ell(Q),4\Delta\ell(Q))$, and all points
$x,y\in\supp\mu\ci Q$ satisfy this restriction, provided that $\alpha<\Delta$. Since this bound 
is symmetric in $x,y$ and since 
$$
\int\frac{\Delta\ell(Q)}{[\Delta\ell(Q)+|x-y|]^{d+1}}\,d\mu(y)\le C
$$
independently of the choice of $\Delta$, we conclude that the norm of the operator
corresponding to the second term in the decomposition of $K^H(x^*-y)$ in $L^2(\mu\ci Q)$
is bounded by some fixed constant as well.

Note now that $\wt K^H(x,y)=0$ whenever $x\in L$ or $y\in L$. We also have the antisymmetry
property: 
$$
\wt K^H(y,x)=-\wt K^H(x,y)\,.
$$
At last $\wt K^H(x,y)$ is harmonic in each variable as long as $x,y\in S$, $x\ne y$.

The next important thing to note is that the correction term \linebreak $K^H(x^*-y)$ is uniformly bounded and Lipschitz 
in $x\in S$ as long as $y\in S,\dist(y,L)\ge\Delta\ell(Q)$. More precisely, for all such $y$,
$$
\|K^H(\cdot^*-y)\|\ci{L^\infty(S)}\le\frac {1}{\Delta^d\ell(Q)^d}\text{ and }
\|K^H(\cdot^*-y)\|\ci{\Lip(S)}\le\frac {C}{\Delta^{d+1}\ell(Q)^{d+1}}\,.
$$
To pass from the smallness of $\|\RH(\nu-\nu\ci Q)\|\ci{L^2(\nu)}$ to that of
$\|\wtRH\nu\|\ci{L^2(\nu)}$, it suffices to estimate the norm
$
\|\RH\nu\ci Q-T\nu\|\ci{L^2(\nu)}
$.

We start with showing that $\RH\nu\ci Q-T\nu\ci Q$ is uniformly bounded by $C\Delta\e^{-2}$ 
on $S$. Indeed, using the identities $K^H(x^*-y)=K^H(x-y^*)$ ($x,y\in S$) and $y^*=y-z$ ($y\in L(Q)$) 
where $z$ is the inner normal vector to the boundary of $S$ of length $2\dist(L(Q),L)\le 6\Delta\ell(Q)$.
Thus, 
\begin{multline*}
[T\nu\ci Q](x)=\int K^H(x-y^*)\,d\nu\ci Q(y)
\\
=\int K^H(x+z-y)\,d\nu\ci Q(y)=[\RH\nu\ci Q](x+z)\,,
\end{multline*}
whence, by Lemma \ref{smoothtolip},
$$
|\RH\nu\ci Q(x)-T\nu\ci Q(x)|=|\RH\nu\ci Q(x)-\RH\nu\ci Q(x+z)|\le
\|\RH\nu\ci Q\|\ci{\Lip}|z|\le \frac{C\Delta}{\e^2}\,.
$$
Now we will estimate $\|T\nu\ci Q-T\nu\|\ci{L^2(\nu)}$. 
Note that
\begin{multline*}
\|T(\nu\ci Q-\nu)\|\ci{\Lip(S)}
\\
\le 
\sup_{y\in(\supp\nu\,\cup\,\supp\nu\cci Q)}\|K^H(\cdot^*-y)\|\ci{\Lip(S)}(\nu(\R^{d+1})+\nu\ci Q(\R^{d+1}))
\\
\le \frac{C}{\Delta^{d+1}\ell(Q)^{d+1}}\mu(Q)\le \frac C{\Delta^{d+1}\ell(Q)}\,.
\end{multline*}
Similarly,
\begin{multline*}
\|T(\nu\ci Q-\nu)\|\ci{L^\infty(S)}
\\
\le 
\sup_{y\in(\supp\nu\,\cup\,\supp\nu\cci Q)}\|K^H(\cdot^*-y)\|\ci{L^\infty(S)}(\nu(\R^{d+1})+\nu\ci Q(\R^{d+1}))
\\
\le \frac{C}{\Delta^{d}\ell(Q)^{d}}\mu(Q)\le \frac C{\Delta^{d}}\,.
\end{multline*}
Thus, by Lemma \ref{mutonulip},
\begin{multline*}
\left|\int |T(\nu\ci Q-\nu)|^2\,d(\f\ci{Q'}\mu-\nu\ci{Q'})\right|\le 
C\alpha\ell(Q')^{d+2}\frac{1}{\Delta^d}\frac{1}{\Delta^{d+1}\ell(Q)}\frac{1}{\e\ell(Q')}
\\
\le C\alpha\Delta^{-2d-1}\e^{-1}\ell(Q')^d\le C\alpha\Delta^{-2d-1}\e^{-1}\mu(Q')\,.
\end{multline*}
Summing over $Q'\in\QQ'$, we get
$$
\int |T(\nu\ci Q-\nu)|^2\,d\nu \le
\int |T(\nu\ci Q-\nu)|^2\,d(\Phi\mu)+ C\alpha\Delta^{-2d-1}\e^{-1}\mu(Q)\,,
$$
so
$$
\|T(\nu\ci Q-\nu)\|\ci{L^2(\nu)}\le \|T(\nu\ci Q-\nu)\|\ci{L^2(\Phi\mu)}+
C\alpha^{\frac 12}\Delta^{-\frac{2d+1}2}\e^{-\frac 12}\sqrt{\mu(Q)}\,.
$$

On the other hand, applying Lemma \ref{mutonulip} again, we see that for every $x\in\supp\mu\ci Q$,
\begin{multline*}
\left|[T(\f\ci Q\mu-\nu\ci Q)](x)\right|=
\left|\int K^H(x^*-\cdot)d(\f\ci Q\mu-\nu\ci Q)\right|
\\
\le
C\alpha\ell(Q)^{d+2}\|K^H(x^*-\cdot)\|\ci{\Lip(S)}\|\f\ci Q\|\ci{\Lip}
\\
\le
C\alpha\ell(Q)^{d+2}\frac{1}{\Delta^{d+1}\ell(Q)^{d+1}}\frac{1}{\e\ell(Q)}
\le C\alpha\Delta^{-d-1}\e^{-1}
\end{multline*}
because 
$$
\|K^H(x^*-\cdot)\|\ci{\Lip(S)}\le \frac{C}{\Delta^{d+1}\ell(Q)^{d+1}}
$$
as long as $x\in S$, $\dist(x,L)\ge\Delta\ell(Q)$ (this is the same inequality as we used before
only with the roles of $x$ and $y$ exchanged).

Similarly, for every $Q'\in\QQ'$, we have 
\begin{multline*}
\left|[T(\f\ci{Q'}\mu-\nu\ci{Q'})](x)\right|=
\left|\int K^H(x^*-\cdot)d(\f\ci{Q'}\mu-\nu\ci{Q'})\right|
\\
\le
C\alpha\ell(Q')^{d+2}\|K^H(x^*-\cdot)\|\ci{\Lip(S)}\|\f\ci {Q'}\|\ci{\Lip}
\\
\le
C\alpha\ell(Q')^{d+2}\frac{1}{\Delta^{d+1}\ell(Q)^{d+1}}\frac{1}{\e\ell(Q')}
\\
\le C\alpha\Delta^{-d-1}\e^{-1}\frac{\ell(Q')^d}{\ell(Q)^d}
\le  C\alpha\Delta^{-d-1}\e^{-1}\frac{\mu(Q')}{\mu(Q)}\,.
\end{multline*}
Summing these inequalities over $Q'\in\QQ'$, we get 
$$
\left|[T(\Phi\mu-\nu)](x)\right|\le C\alpha\Delta^{-d-1}\e^{-1}
$$
for all $x\in\supp\mu\ci Q$\,.

Relaxing the $L^\infty$ bounds to the $L^2$ ones, we conclude that
$$
\|T(\nu\ci Q-\nu)\|\ci{L^2(\Phi\mu)}\le
\|T((\f\ci Q- \Phi)\mu)\|\ci{L^2(\mu\cci Q)}+C\alpha\Delta^{-d-1}\e^{-1}\sqrt{\mu(Q)}\,.
$$
However, since the operator norm of $T$ in $L^2(\mu\ci Q)$ is bounded by a constant, we have
$$
\|T((\f\ci Q-\Phi)\mu)\|\ci{L^2(\mu\cci Q)}\le \|\f\ci Q-\Phi\|\ci{L^2(\mu)}
\le C\e^{\frac\gamma 2}\sqrt{\mu(Q)}
$$
by \eqref{phiqphi}.
Thus, we finally get
\begin{multline*}
\|\RH(\nu-\nu\ci Q)\|\ci{L^2(\nu)}
\\
\ge \|\wtRH\nu\|\ci{L^2(\nu)}-
C\left[\e^{\frac\gamma 2}
+\Delta\e^{-2}
+ \alpha^{\frac 12}\Delta^{-\frac{2d+1}2}\e^{-\frac 12}
+\alpha\Delta^{-d-1}\e^{-1}\right]\sqrt{\mu(Q)}\,.
\end{multline*}
Putting $\Delta=\e^3$, say, we obtain the desired bound with
$$
\sigma(\e,\alpha)=C\left[\e^{\frac\gamma 2}
+\e
+ \alpha^{\frac 12}\e^{-3d-2}
+\alpha\e^{-3d-4}\right]\,.
$$

\section{The intermediate non-BAUP layer}
\label{nonbauplayer}

Until now, we worked only with a flat cell $Q\in\QQ_k$ and the family $\QQ'$
of flat cells $Q'\in\QQ_{k+1}$ contained in $Q$, completely ignoring the non-BAUP
layer $\PP_{k+1}$. At this point, we finally bring it into the play. We will start
with the definition of a $\delta$-non-BAUP cell.

\begin{udef}
Let $\delta>0$. We say that a cell $P\in\D$ is $\delta$-non-BAUP if there exists
a point $x\in P\cap\supp\mu$ such that for every hyperplane $L$ passing through $x$,
there exists a point $y\in B(x,\ell(P))\cap L$ for which $B(y,\delta\ell(P))\cap\supp\mu=\varnothing$. 
\end{udef}

Note that in this definition the plane $L$ can go in any direction. In what follows, we will
need only planes parallel to $H$ but, since $H$ is determined by the flatness direction of
some unknown subcube of $P$, we cannot fix the direction of the plane $L$ in the definition of
non-BAUPness from the very beginning. For every non-BAUP cell $P'\in\PP_{k+1}$, we will denote
by $x\ci {P'}$ the point $x$ from the definition of the non-BAUPness for $P'$ and by $y\ci{P'}$
the point $y$ corresponding to $x=x\ci{P'}$ and $L$ parallel to $H$.

\goal
The goal of this section is to show that under our usual assumptions ($\e$ is sufficiently small in terms of $\delta$,
$A$ is sufficiently large in terms of $\delta$, 
$\alpha$ is sufficiently small in terms of $\e$ and $\delta$), there exists a family $\PP'\subset\PP_{k+1}$ such that
\begin{itemize}
\item
Every cell $P'\subset\PP'$ is contained in $Q_\e$ and satisfies $\ell(P')\le 2\alpha\delta^{-1}\ell(Q)$.
\item
$\sum_{P'\in\PP'}\mu(P')\ge c\mu(Q)$.
\item
The balls $B(z\ci{P'},10\ell(P'))$, $P'\in\PP'$ are pairwise disjoint.
\item
The function 
$$
h(x)=\sum_{P'\in\PP'}\left[\frac{\ell(P')}{\ell(P')+\dist(x,P')}\right]^{d+1}
$$
satisfies $\|h\|\ci{L^\infty}\le C$.
\end{itemize}
\goalend

\begin{proof}
We start with showing that every $\delta$-non-BAUP cell $P'$ contained in $Q$ has much smaller size than $Q$.
Indeed, we know that $\supp\mu\cap B(z\ci Q,A\ell(Q))$ is contained in the $\alpha\ell(Q)$-neighborhood of
$L(Q)$ and that $B(y,\alpha\ell(Q))\cap\supp\mu\ne\varnothing$ for every $y\in B(z\ci Q,A\ell(Q))\cap L(Q)$.
Suppose that $P'\subset Q$ is $\delta$-non-BAUP. If $A>5$, then  
$$
B(x\ci{P'},\ell(P'))\subset B(z\ci Q,5\ell(Q))\subset B(z\ci Q,A\ell(Q))\,.
$$ 
Moreover, since $y\ci{P'}-x\ci{P'}\in H$, we have
$$
\dist(y\ci{P'},L(Q))=\dist(x\ci{P'},L(Q))\le\alpha\ell(Q)\,.
$$
Let $y^*\ci{P'}$ be the projection of $y\ci{P'}$ to $L(Q)$. Then
$|y^*\ci{P'}-y\ci{P'}|\le\alpha\ell(Q)$ and $|y^*\ci{P'}-z\ci{Q}|\le |y\ci{P'}-z\ci{Q}|< A\ell(Q)$.
Thus, the ball $B(y\ci{P'},2\alpha\ell(Q))\supset B(y^*\ci{P'},\alpha\ell(Q))$ intersects $\supp\mu$,
so $\delta\ell(P')<2\alpha\ell(Q)$, i.e., $\ell(P')\le 2\alpha\delta^{-1}\ell(Q)$.

Let now $\PP=\{P'\in\PP_{k+1}:P'\subset Q\}$. Consider the function 
$$
g(P')=\sum_{P''\in\PP}\left[\frac{\ell(P'')}{D(P',P'')}\right]^{d+1}
$$
(the same function as the one we used in Section \ref{approximatingmeasure}, only corresponding
to the family $\PP$ instead of $\QQ$). The same argument as in Section \ref{approximatingmeasure}
shows that 
$$
\sum_{P'\in\PP} g(P')\mu(P')\le C_1\mu(Q)
$$
for some $C_1>0$ depending on the dimension $d$ and the goodness parameters of $\mu$ only.
Define 
$$
\PP^*=\{P'\in\PP:P'\subset Q_\e, g(P')\le 3C_1\}\,.
$$
Note that
$$
\sum_{P'\in\PP^*}\mu(P')\ge \sum_{P'\in\PP}\mu(P')-\sum_{P'\in\PP:P'\not\subset Q_\e}\mu(P')-
\sum_{P'\in\PP:g(P')>3C_1}\mu(P')\,. 
$$
However,
$$
\sum_{P'\in\PP}\mu(P')\ge\sum_{Q'\in\QQ}\mu(Q')\ge (1-\e)\mu(Q)\,.
$$
Further, since the diameter of each $P'\in\PP$ is at most $8\ell(P')\le 8\alpha\delta^{-1}\ell(Q)$,
every cell $P'\in\PP$ that is not contained in $Q_\e$ is contained in $Q\setminus Q_{2\e}$, provided
that $\alpha<\frac{1}{8}\e\delta$. Thus, under this restriction,
$$
\sum_{P'\in\PP:P'\not\subset Q_\e}\mu(P')\le\mu(Q\setminus Q_{2\e})\le C\e^{\gamma}\mu(Q)\,.
$$
Finally, by Chebyshev's inequality,
$$
\sum_{P'\in\PP:g(P')>3C_1}\mu(P')\le \frac{\mu(Q)}{3}\,.
$$
Bringing these three estimates together, we get the inequality 
$\sum_{P'\in\PP^*}\mu(P')\ge \frac 12\mu(Q)$, provided that $A,\e,\alpha$ satisfy some restrictions of the 
admissible type.

Now we will rarefy the family $\PP^*$ a little bit more. Consider the balls
$B(z\ci{P'},10\ell(P'))$, $P'\in\PP^*$. By the classical Vitali covering theorem, we can
choose some subfamily $\PP'\subset\PP^*$ such that the balls $B(z\ci{P'},10\ell(P'))$, $P'\in\PP'$
are pairwise disjoint but 
$$
\bigcup_{P'\in\PP'} B(z\ci{P'},30\ell(P'))\supset \bigcup_{P'\in\PP^*} B(z\ci{P'},10\ell(P'))\supset
\bigcup_{P'\in\PP^*}P'\,.
$$
Then we will still have 
\begin{multline*}
\sum_{P'\in\PP'}\mu(P')
\ge 
c\sum_{P'\in\PP'}\ell(P')^d
\\
\ge
c\sum_{P'\in\PP'}\mu(B(z\ci{P'},30\ell(P')))\ge c\sum_{P'\in\PP^*}\mu(P')\ge c\mu(Q)\,.
\end{multline*}
It remains only to prove the bound for the function $h$. Take any $x\in\R^{d+1}$.
Let $P'$ be a nearest to $x$ cell in $\PP'$. We claim that for every cell $P''\in\PP'$,
we have
$$
\dist(x,P'')+\ell(P'')\ge \frac 14 D(P',P'')\,.
$$  
Indeed, if $P'=P''$, the inequality trivially holds even with $\frac 12$ in place of $\frac 14$. Otherwise,
the disjointness of the balls $B(z\ci{P'},10\ell(P'))$ and $B(z\ci{P''},10\ell(P''))$ implies
that
\begin{multline*}
\dist(P',P'')\ge |z\ci{P'}-z\ci{P''}|-4(\ell(P')+\ell(P''))
\\
\ge 
10(\ell(P')+\ell(P''))-4(\ell(P')+\ell(P''))=6(\ell(P')+\ell(P''))\,,
\end{multline*}
so
$$
D(P',P'')=\dist(P',P'')+\ell(P')+\ell(P'')\le 2\dist(P',P'')\,.
$$
On the other hand,
$$
\dist(P',P'')\le \dist(x,P')+\dist(x,P'')\le 2\dist(x,P'')\,.
$$
Thus
$$
\dist(x,P'')+\ell(P'')\ge \dist(x,P'')\ge \frac 14D(P',P'')\,.
$$
Now it remains to note that
\begin{multline*}
h(x)=\sum_{P''\in\PP'}\left[\frac{\ell(P'')}{\ell(P'')+\dist(x,P'')}\right]^{d+1}
\\
\le \sum_{P''\in\PP'}\left[\frac{4\ell(P'')}{D(P',P'')}\right]^{d+1}
\le Cg(P')\le C\,.
\end{multline*}
\end{proof}

\section{The function $\eta$}
\label{functioneta}

Fix the non-BAUPness parameter $\delta\in(0,1)$. Fix any $C^{\infty}$ radial function $\eta_0$ supported in $B(0,1)$ such that $0\le\eta_0\le 1$
and $\eta_0=1$ on $B(0,\frac 12)$. For every $P'\in\PP'$, define
$$
\eta\ci {P'}(x)=\eta_0\left(\frac{1}{\delta\ell(P')}(x-x\ci{P'})\right)-\eta_0\left(\frac{1}{\delta\ell(P')}(x-y\ci{P'})\right)\,.
$$
Note that $\eta\ci{P'}$ is supported on the ball $B(z\ci{P'},6\ell(P'))$. This ball is contained in $Q$, provided that 
$12\alpha\delta^{-1}<\e$ (recall that $\ell(P')\le 2\alpha\delta^{-1}\ell(Q)$ and $P'\subset Q_\e$). Also
$\eta\ci{P'}\ge 1$ on $B(x\ci{P'},\frac\delta 2\ell(P'))$ and the support of the negative part of $\eta\ci{P'}$ is disjoint
with $\supp\mu$. Put 
$$
\eta=\sum_{P'\in\PP'}\eta\ci{P'}\,.
$$
Since even the balls $B(z\ci{P'},10\ell(P'))$ corresponding to different $P'\in\PP'$ are disjoint, we have $-1\le\eta\le 1$.

\goal
The goal of this section is to show that, under our usual assumptions, we have
$\supp\eta\subset S$, $\dist(\supp\eta,L)\ge\Delta\ell(Q)=\e^3\ell(Q)$,
and 
$$
\int\eta\,d\nu\ge c(\delta)\mu(Q)
$$
with some $c(\delta)>0$ (we remind the reader that we suppress the dependence of constants 
on the dimension $d$ and the goodness parameters of the measure $\mu$
in our notation).
\goalend

\begin{proof}

The first part of our claim is easy because for every $P'\in\PP'$, we have
$\supp\eta\ci{P'}\subset B(z\ci{P'},6\ell(P'))$ and 
$$
\dist(z\ci{P'},L)-6\ell(P')\ge 2\Delta\ell(Q)-12\alpha\delta^{-1}\ell(Q)\ge\Delta\ell(Q)
$$
as long as $12\alpha<\delta\Delta=\delta\e^3$.

To get the second part, recall that, by Lemma \ref{mutonulip}, for every $Q'\in\QQ'$, we have
\begin{multline*}
\left|\int \eta\,d(\f\ci{Q'}\mu-\nu\ci{Q'})\right|\le
C\alpha\ell(Q')^{d+2}\|\eta\|\ci{\Lip(\supp\f\cci{Q'})}\|\f\ci {Q'}\|\ci{\Lip}
\\
\le C\alpha\e^{-1}\mu(Q')\ell(Q')\|\eta\|\ci{\Lip(\supp\f\cci{Q'})}
\end{multline*}
So our first step will be to show that
for every $Q'\in\QQ'$, we have
$$
\|\eta\|\ci{\Lip(\supp\f\cci{Q'})}\le \frac{C}{\delta\e\ell(Q')}\,.
$$
Since the building blocks $\eta\ci{P'}$ ($P'\in\PP'$) of the function $\eta$ have disjoint supports,
it suffices to check this inequality for each $\eta\ci{P'}$ separately. 

Since $\|\eta\ci{P'}\|\ci{\Lip}\le\frac{C}{\delta\ell(P')}$, the inequality is trivial if $2\ell(P')\ge \e\ell(Q')$.
Otherwise, we cannot have $Q'\subset P'$, so we must have $Q'\cap P'=\varnothing$. However, $\supp\eta\ci{P'}$
is contained in the $2\ell(P')$-neighborhood of $P'$, so it cannot reach $\supp\f\ci{Q'}\subset Q'_\e$ and, thereby,
$\eta\ci{P'}=0$ on $\supp\f\ci{Q'}$ in this case.

Now, we get
\begin{multline*}
\int \eta\,d\nu=\sum_{Q'\in\QQ'}\int \eta\,d\nu\ci{Q'}
\\
\ge
\sum_{Q'\in\QQ'}\left[\int \eta\,d(\f\ci{Q'}\mu)-C\alpha\delta^{-1}\e^{-2}\mu(Q')\right]
\ge \int\eta\,d(\Phi\mu)-C\alpha\delta^{-1}\e^{-2}\mu(Q)\,.
\end{multline*}
On the other hand, since $\supp\eta\subset Q$ and $\supp\eta_-\cap\supp\mu=\varnothing$, we have
$$
\int\eta\,d(\Phi\mu)=\int\eta_+\,d(\Phi\mu)\ge 
\int\eta_+\,d\mu-\int (\chi\ci Q-\Phi)\,d\mu\,.
$$
However, 
$$
\int\eta_+\,d\mu\ge c\sum_{P'\in\PP'}(\delta\ell(P'))^d\ge c\delta^{d}\sum_{P'\in\PP'}\mu(P')
\ge c\delta^d\mu(Q)\,,
$$
while, as we have seen in the beginning of Step 2 in Section \ref{approximatingmeasure},
$$
\int (\chi\ci Q-\Phi)\,d\mu=\|\chi\ci Q-\Phi\|\ci{L^1(\mu)}\le C\e^{\gamma}\mu(Q)\,.
$$ 
So, we end up with
$$
\int\eta\,d\nu\ge [c\delta^d-C(\e^{\gamma}+\alpha\delta^{-1}\e^{-2})]\mu(Q)\ge c\delta^d\mu(Q)\,,
$$
provided that we demand that $\e>0$ is small in terms of $\delta$ and $\alpha>0$ is small in terms
of $\delta$ and $\e$, as usual.
\end{proof}

\section{The vector field $\psi$}
\label{fieldpsi}

Let $m$ denote the Lebesgue measure in $\R^{d+1}$. 

\goal
The goal of this section is to construct a Lipschitz compactly supported vector field $\psi$ such that
\begin{itemize}
\item
$\psi=\sum_{P'\in\PP'}\psi\ci{P'}$, $\supp\psi\subset S$, $\dist(\supp\psi,L)\ge\Delta\ell(Q)=\e^3\ell(Q)$.
\item
$\psi\ci{P'}$ is supported in the $2\ell(P')$-neighborhood of $P'$ and satisfies
$$
\int\psi\ci{P'}=0,\quad \|\psi\ci{P'}\|\ci{L^\infty}\le \frac{C}{\delta\ell(P')},\quad
\|\psi\ci{P'}\|\ci{\Lip}\le \frac{C}{\delta^2\ell(P')^2}\,.
$$
\item
$\int |\psi|\,dm\le C\delta^{-1}\mu(Q)$.
\item
$(\RH)^*(\psi m)=\eta$.
\item
$\|T^*(\psi m)\|\ci{L^\infty(\supp\nu)}\le C\alpha\delta^{-2}\e^{-3d-3}$.
\item
$\|\wtRH(|\psi|m)\|\ci{L^2(\nu)}\le C\delta^{-1}\sqrt{\mu(Q)}$\,.
\end{itemize} 
\goalend

\begin{proof}
Fix $P'\in\PP'$. Let $e\ci{P'}$ be the unit vector in the direction $y\ci{P'}-x\ci{P'}$.
Note that $K^H=-c_d\nabla\ci H U$ where $U$ is the fundamental solution of the Laplace operator
in $\R^{d+1}$, so for every $C^\infty_0$-function $u$ in $\R^{d+1}$, we have
$$
K^H*(\Delta u)=-c_d\nabla\ci H[U*(\Delta u)]=-c_d\nabla\ci H u\,.
$$
In particular,
$$
\langle\RH[(\Delta u)m],e\ci P\rangle=-c_d\nabla\ci{e\cci P} u\,.
$$
Note that for every reasonable finite vector-valued measure $\sigma$, we have
$$
(\RH)^*\sigma=-\sum_j\left\langle\RH\langle\sigma,e_j\rangle,e_j\right\rangle
$$
where $e_1,\dots,e_d$ is any orthonormal basis in $H$. If we apply this identity
to $\sigma=-c_d^{-1}(\Delta u)e\ci{P'}m$ and choose the basis $e_1,\dots,e_d$ so that $e_1=e\ci{P'}$,
we will get
$$
(\RH)^*[-c_d^{-1}(\Delta u)e\ci{P'}m]=-c_d^{-1}\langle\RH[(\Delta u)m],e\ci {P'}\rangle=\nabla\ci{e\cci {P'}} u\,.
$$
We will now define a function $u\ci{P'}\in C_0^\infty$ for which $\nabla\ci{e\cci {P'}} u=\eta\ci{P'}$.
To this end, we just put
$$
u\ci{P'}(x)=\int_{-\infty}^0 \eta\ci{P'}(x+te\ci{P'})\,dt\,.
$$
Since the restriction of $\eta\ci{P'}$ to any line parallel to $e\ci{P'}$ consists of two 
opposite bumps, the support of $u\ci{P'}$ is contained in the convex hull of $B(x\ci{P'},\delta\ell(P'))$
and $B(y\ci{P'},\delta\ell(P'))$. Also, since $\|\nabla^j\eta\ci{P'}\|\ci{L^\infty}\le C(j)[\delta\ell(P')]^{-j}$
and since $\supp\eta\ci{P'}$ intersects any line parallel to $e\ci{P'}$ over two intervals of total
length $4\delta\ell(P')$ or less, we have 
$$
|\nabla^j u\ci{P'}(x)|\le \int_{-\infty}^0 |(\nabla^j\eta\ci{P'})(x+te\ci{P'})|\,dt\le 
\frac{C(j)}{[\delta\ell(P')]^{j-1}}
$$ 
for all $j\ge 0$.
Define the vector fields 
$$
\psi\ci{P'}=-c_d^{-1}(\Delta u\ci{P'})e\ci{P'},\qquad \psi=\sum_{P'\in\PP'}\psi\ci{P'}\,.
$$
Then, clearly, $(\RH)^*(\psi m)=\eta$ and we have all other properties of the individual
vector fields $\psi\ci{P'}$ we need (the mean zero property holds because the integral
of any Laplacian of a compactly supported $C^\infty$ function over the entire space
is $0$ and the support property holds because even the balls $B(z\ci{P'},6\ell(P'))$ 
lie deep inside $S$). We also have
\begin{multline*}
\int|\psi|\,dm=\sum_{P'\in\PP'}\int|\psi\ci{P'}|\,dm\le
C\sum_{P'\in\PP'}[\delta\ell(P')]^{-1}m(B(z\ci{P'},6\ell(P')))
\\
\le
C\delta^{-1}\sum_{P'\in\PP'}\ell(P')^d
\le C\delta^{-1}\sum_{P'\in\PP'}\mu(P')\le C\delta^{-1}\mu(Q)\,.
\end{multline*}
To get the uniform estimate for $T^*(\psi m)$, note that for every vector-valued Lipschitz 
function $F$ in $S$ and every $P'\in\PP'$, we have 
\begin{multline*}
\left|\int \langle F,\psi\ci{P'}\rangle\,dm\right|
=\left|\int\langle F-F(z\ci{P'}),\psi\ci{P'}\rangle\,dm\right|
\\
\le
6\|F\|\ci{\Lip(S)}\ell(P')\int|\psi\ci{P'}|\,dm
\le C\delta^{-1}\|F\|\ci{\Lip(S)}\ell(P')^{d+1}\,.
\end{multline*}
Since the kernel of $T$ is still antisymmetric, we have
\begin{multline*}
|[T^*(\psi\ci{P'}m)](x)|=\left|\int\langle K^H(x^*-\cdot),\psi\ci{P'}\rangle\,dm\right|
\le C\delta^{-1}\|K^H(x^*-\cdot)\|\ci{\Lip(S)}\ell(P')^{d+1}
\\
\le C\delta^{-1}\Delta^{-d-1}\frac{\ell(P')^{d+1}}{\ell(Q)^{d+1}}
\\
\le C\alpha\delta^{-2}\Delta^{-d-1}\frac{\mu(P')}{\mu(Q)}
\end{multline*}
for every $x\in\supp\nu$ (we remind the reader that $\ell(P')\le 2\alpha\delta^{-1}\ell(Q)$).
Adding these estimates up and recalling our choice $\Delta=\e^3$, we get
$$
\|T^*\psi\|\ci{L^\infty(\supp\nu)}\le
C\alpha\delta^{-2}\e^{-3d-3}\sum_{P'\in\PP'}\frac{\mu(P')}{\mu(Q)}
\le C\alpha\delta^{-2}\e^{-3d-3}\,.
$$
It remains to bound $\wtRH(|\psi|m)$ in $L^2(\nu)$. As usual, we will prove
the $L^2(\mu)$ bound first and then use the appropriate Lipschitz properties to switch
to the $L^2(\nu)$ bound.

Recall that for every $P'\in\PP'$, we have $\int |\psi\ci{P'}|\,dm\le C\delta^{-1}\ell(P')^d$.
Hence, we can choose constants $b\ci{P'}\in(0,C\delta^{-1})$ so that $|\psi\ci{P'}|m-b\ci{P'}\chi\ci{P'}\mu$
is a balanced signed measure, i.e., 
$$
\int |\psi\ci{P'}|\,dm=b\ci{P'}\int \chi\ci{P'}\,d\mu\,.
$$
Let 
$$
f=\sum_{P'\in\PP'}b\ci{P'}\chi\ci{P'}\,.
$$
Note that $\|f\|\ci{L^2(\mu)}^2\le C\delta^{-2}\mu(Q)$. For each $P'\in\PP'$, denote by $V(P')$ the set of all points 
$x\in\R^{d+1}$ such that $\dist(x,P')\le\dist(x,P'')$ for all $P''\in\PP'$. Note that the sets $V(P')$ are closed and
cover the entire space $\R^{d+1}$, possibly, with some overlaps. Introduce some linear order $\prec$ on the
finite set $\PP'$ and put
$$
V'(P')=V(P')\setminus\left(\bigcup_{P''\in\PP',P''\prec P'}V(P'')\right)\,.
$$ 
Then the Borel sets $V'(P')\subset V(P')$ form a tiling of $\R^{d+1}$.

Let $x\in V'(P')$. We have
\begin{multline*}
[\wtRH(|\psi|m-f\mu)](x)
\\
=[\wtRH(|\psi\ci{P'}|m)](x)-[\wtRH(b\ci{P'}\chi\ci{P'}\mu)](x)
+\sum_{P''\in\PP',P''\ne P'}[\wtRH(|\psi\ci{P''}|m-b\ci{P''}\chi\ci{P''}\mu)](x)\,.
\end{multline*}
We have seen in Section \ref{nonbauplayer} that for every $P''\in\PP'\setminus\{P'\}$, we have
$$
\dist(x,P'')\ge\frac 14D(P',P'')\ge\frac 14\ell(P'')\,.
$$
Thus, 
\begin{multline*}
\left|\RH(|\psi\ci{P''}|m-b\ci{P''}\chi\ci{P''}\mu)](x)\right|=
\left|\int K^H(x-\cdot)\,d(|\psi\ci{P''}|m-b\ci{P''}\chi\ci{P''}\mu)\right|
\\
=\left|\int [K^H(x-\cdot)-K^H(x-z\ci{P''})]\,d(|\psi\ci{P''}|m-b\ci{P''}\chi\ci{P''}\mu)\right|
\\
\le 2\|K^H(x-\cdot)-K^H(x-z\ci{P''})\|\ci {L^\infty(P'')}\int |\psi\ci{P''}|\,dm
\\
\le \frac{C\ell(P'')}{\dist(x,P'')^{d+1}}\delta^{-1}\ell(P'')^d
\le C\delta^{-1}\left[\frac{\ell(P'')}{\ell(P'')+\dist(x,P'')}\right]^{d+1}\,,
\end{multline*}
and the same estimate (with the same proof) holds for 
$T(|\psi\ci{P''}|m-b\ci{P''}\chi\ci{P''}\mu)](x)$.

Hence, 
$$
\sum_{P''\in\PP',P''\ne P'}|[\wtRH(|\psi\ci{P''}|m-b\ci{P''}\chi\ci{P''}\mu)](x)|\le
C\delta^{-1}h(x)\le C\delta^{-1}
$$
for all $x\in V'(P')$ (here $h$ is the function introduced in Section \ref{nonbauplayer}).

Note also that 
$$
\|\wtRH(|\psi\ci{P'}|m)\|\ci{L^\infty}\le C\delta^{-1}
$$
(this is just the trivial bound $C\ell(P')$ for the integral of the absolute value of the kernel 
over a set of diameter $12\ell(P')$ multiplied by the bound 
$\frac{C}{\delta\ell(P')}$ for the maximum of $|\psi\ci{P'}|$).

Thus, we have the pointwise (or, more precisely, $\mu$-almost everywhere) estimate
$$
|\wtRH(|\psi|m)|\le C\delta^{-1}+|\wtRH(f\mu)|+\sum_{P'\in\PP'}\chi\ci{V'(P')}|\wtRH(b\ci{P'}\chi\ci{P'}\mu)|\,,
$$
which converts into
\begin{multline*}
\|\wtRH(|\psi| m)\|\ci{L^2(\mu)}^2
\\
\le
C\left[\delta^{-2}\mu(Q)+\|f\|\ci{L^2(\mu)}^2+\sum_{P'\in\PP'}\|b\ci{P'}\chi\ci{P'}\|\ci{L^2(\mu)}^2\right]
\le C\delta^{-2}\mu(Q)\,.
\end{multline*}
Due to Lemma \ref{comparisonlemma}, it only remains to bound the quantities \linebreak
$\ell(Q')\|\wtRH(|\psi|m)\|\ci{\Lip(\supp\f\cci{Q'})}$, $Q'\in\QQ'$, by some expression depending on $\delta$ and $\e$ only
(plus, of course, the dimension and the goodness constants of $\mu$, which go without mentioning).

Note first of all that for every $P'\in\PP'$, we have 
$$
\|\wtRH(|\psi\ci{P'}|m)\|\ci{\Lip}\le C\delta^{-2} \ell(P')^{-1}
$$
because $|\nabla |\psi\ci{P'}||\le |\nabla\psi\ci{P'}|\le C \delta^{-2} \ell(P')^{-2}$
and $\supp\psi\ci{P'}\subset B(z\ci{P'},6\ell(P'))$. 
We also have another estimate
$$
\|\wt{R}^H (|\psi\ci{P'}| m)\|\ci{\Lip(Q'_{\e})}\le \frac{C\delta^{-1} \ell(P')^d}{\dist(Q'_{\e}, \supp\psi\ci{P'})^{d+1}}\,,
$$
because $\int |\psi\ci{P'}|\,dm \le C\delta^{-1} \ell(P')^d$. 

To estimate $\|\wtRH (|\psi| m)\|\ci{\Lip(Q'_{\e})}$, we fix $Q'\in\QQ'$ and split 
\begin{multline*}
\wtRH (|\psi| m)
\\
 = \sum_{P': Q'_\e\cap B(z\cci{P'}, 8\ell(P'))\neq\varnothing} \wtRH (|\psi\ci{P'}| m) + \sum_{P': Q'_\e\cap B(z\cci{P'}, 8\ell(P'))=\varnothing} \wtRH (|\psi\ci{P'}| m) \,.
\end{multline*}
Notice that each $P'$ in the first sum satisfies $\ell(P') \ge \frac \e 8 \ell(Q')$. 
Indeed, if $\ell(P')<\ell(Q')$, then we must have $P'\cap Q'=\varnothing$ and $z\ci{P'}\notin Q'$ whence 
$8\ell(P')>\dist(z\ci{P'},Q'_\e)\ge\e\ell(Q')$. 
On the other hand, if the cell $P'$ in the first sum satisfies $\ell(P') \ge 2\ell(Q')$ then $z\ci{Q'}\in B(z\ci {P'}, 10\ell(P') )$. 
However, the balls $B(z\ci{P'}, 10\ell(P') )$ are pairwise disjoint, so there may be only one cell $P'$ in the first family 
with this property. 
Thus, the total number of cells $P'$ in the first sum is bounded by $C\e^{-d}$. 
Since each corresponding function $\wtRH (|\psi\ci{P'}| m)$ has Lipschitz norm at most $C\delta^{-2} \ell(P')^{-1}\le C\delta^{-2} \e^{-1}\ell(Q')^{-1}$, we conclude that  the Lipschitz constant of the first sum on $Q_\e'$ is bounded by $C\e^{-d-1} \delta^{-2}\ell(Q')^{-1}$.

For each $P'$ in the second sum, we have 
$$
\|\wtRH (|\psi\ci{P'}| m)\|\ci {\Lip(Q'_{\e})}
\le \frac{C\delta^{-1} \mu(P')}{\dist(Q'_{\e}, \supp\psi\ci{P'})^{d+1}}\le 
\frac{C\delta^{-1}\mu(P')}{[\e D(Q',P')]^{d+1}}\,.
$$
Thus, the Lipschitz constant of the first sum on $Q_\e'$ is bounded by 
$$
C \delta^{-1} \e^{-(d+1)} \int\frac{d\mu(x)}{[\ell(Q') +\text{dist}(x, Q')]^{d+1}} \le C\delta^{-1}\e^{-(d+1)} \ell(Q')^{-1}\,.
$$

\section{Smearing of the measure $\nu$}
\label{smearing}

\goal
The goal of this section is to replace the measure $\nu$ by a compactly supported measure $\wt\nu$ that has
a bounded density with respect to the $(d+1)$-dimensional Lebesgue measure $m$ in $\R^{d+1}$. More
precisely, for every $\varkappa>0$, we will construct a measure $\wt\nu$ with the following properties:
\begin{itemize}
\item
$\wt\nu$ is absolutely continuous and has bounded density with respect to $m$.
\item
$\supp\wt\nu\subset S$ and $\dist(\supp\wt\nu,L)\ge\Delta\ell(Q)$.
\item
$\wt\nu(S)=\nu(S)\le\mu(Q)$.
\item
$\int\eta\,d\wt\nu\ge\int\eta\,d\nu-\varkappa$.
\item
$\int|\wtRH(|\psi|m)|^2\,d\wt\nu\le\int|\wtRH(|\psi|m)|^2\,d\nu+\varkappa$.
\item
$\int|\wtRH\wt\nu|^2\,d\wt\nu\le\int|\wtRH\nu|^2\,d\nu+\varkappa$.
\end{itemize}
\goalend

It is important to note that this step is purely qualitative. The boundedness of the density 
$\frac{d\wt\nu}{dm}$ will
be used to show the existence of a minimizer in a certain extremal problem and the continuity
of the corresponding Riesz potential but the bound itself will not enter any final estimates.

Fix some radial non-negative $C^\infty$-function $\f_1$ with $\supp\f_1\subset B(0,1)$ and $\int\f_1\,dm=1$.
For $0<s\le1$, define
$$
\f_s(x)=s^{-d-1}\f_1(s^{-1}x)
$$
and 
$$
\nu_s=\nu*\f_s\,.
$$
Clearly, all the supports of the measures $\nu_s$ are contained in some compact set and $\nu_s$ 
converge to $\nu$ weakly as $s\to 0+$. If $s$ is much less than $\Delta\ell(Q)$, 
we have $\supp\nu_s\subset S$ and $\dist(\supp\nu_s,L)>\Delta\ell(Q)$. Also, the total
mass of $\nu_s$ is the same as the total mass of $\nu$ for all $s$.

Note that both $\eta$ and $|\wtRH(|\psi|m)|^2$ are continuous functions in $S$,
so the weak convergence is enough to establish the convergence of the corresponding 
integrals.
What is less obvious is that the integrals $\int|\wtRH\nu_s|^2\,d\nu_s$ also converge
to the integral $\int|\wtRH\nu|^2\,d\nu$ because formally it is a trilinear
form in the measure argument with a singular kernel.

Note, however, that for every finite measure $\sigma$, we have 
$\wtRH\sigma=\RH(\sigma-\sigma^*)$ where $\sigma^*$ is the reflection
of the measure $\sigma$ about the boundary hyperplane $L$ of $S$, i.e.,
$\sigma^*(E)=\sigma(E^*)$ where $E^*=\{x^*:x\in E\}$. Moreover, $\RH$
commutes with shifts and, since $\f_s$ is radial (all we really need
is the symmetry about $H$), we have $(\nu*\f_s)^*=\nu^**\f_s$.

Hence, 
$$
\wtRH\nu_s=\RH[\nu*\f_s-\nu^**\f_s]=\RH[(\nu-\nu^*)*\f_s]=[\RH(\nu-\nu^*)]*\f_s\,.
$$
However, $\RH(\nu-\nu^*)$ is a bounded Lipschitz function, so the convergence
$[\RH(\nu-\nu^*)]*\f_s\to \RH(\nu-\nu^*)$ as $s\to 0+$ is uniform on compact sets
and so is the convergence $|[\RH(\nu-\nu^*)]*\f_s|^2\to |\RH(\nu-\nu^*)|^2$. Thus,
despite all the singularities in the kernel, $|\wtRH\nu_s|^2$ converges to
$|\wtRH\nu|^2$ uniformly, which is enough to ensure that 
$$
\int|\wtRH\nu_s|^2\,d\nu_s\to \int|\wtRH\nu|^2\,d\nu
$$
as $s\to 0+$.
So, we can take $\wt\nu=\nu_s$ with sufficiently small $s>0$.

\section{Extremal problem}
\label{extremalproblem}

Fix $\lambda=\lambda(\delta)\in (0,1)$ to be chosen later (as usual, the dependence on the dimension
and the goodness parameters of $\mu$ is suppressed) and assume that
$$
\int|\wtRH\nu|^2\,d\nu<\lambda\mu(Q)\,.
$$  
Then, choosing sufficiently small $\varkappa>0$, we can ensure that the measure $\wt\nu$
constructed in the previous section, satisfies
$$
\int|\wtRH\wt\nu|^2\,d\wt\nu<\lambda\mu(Q)\,,\quad
\int\eta\,d\wt\nu\ge\theta\mu(Q)\,,\quad
\int|\wtRH(|\psi|m)|^2\,d\wt\nu\le\Theta\mu(Q)
$$
where $\theta,\Theta>0$ are two quantities depending only on $\delta$ (plus, of course, the dimension $d$ and the goodness and AD-regularity constants of $\mu$).

Our aim is to show that if $\lambda=\lambda(\delta)>0$ is chosen small enough, then these three conditions
are incompatible. Then, since the last two inequalities hold, the first one should fail, that is,
we must have
$$
\int|\wtRH\nu|^2\,d\nu\ge \lambda\, \mu(Q)\,.
$$
We can next deduce from the estimates in Section \ref{reflectiontrick} that
$$
\|\RH(\nu-\nu\ci Q)\|\ci{L^2(\nu)}^2\ge [\lambda-\sigma(\e,\alpha)]\mu(Q)\,.
$$
Combining this inequality with the results from Section \ref{approximatingmeasure}, we obtain
the estimate
$$
\|F^Q\|\ci{L^2(\mu)}^2\ge \left[\frac{\lambda}2-\sigma(\e,\alpha)\right]\mu(Q) =
2\tau^2\mu(Q)
$$
for every densely packed cell $Q\in\QQ_k$,
where the last identity is the definition of the constant $\tau$. 
As explained in Sections \ref{reduction} and \ref{denselypackedcells}, 
this finishes the proof of our theorem. So, the rest of the paper will be devoted just to the 
proof of the incompatibility in question.

For non-negative $a\in L^{\infty}(m)$, define $\wt\nu_a=a\wt\nu$ and consider
the extremal problem 
$$
\Xi(a) = \lambda \mu(Q) \|a\|\ci{L^\infty(m)} + \int|\wtRH\wt\nu_a|^2 d\wt\nu_a \to \min
$$
under the restriction $\int \eta\,d\wt\nu_a\ge\theta\mu(Q)$. Note that since $\wt\nu$ is absolutely
continuous and has bounded density with respect to $m$, the measure $\wt\nu_a$ is well defined
and has the same properties.

\goal
The goal of this section is to show that the minimum is attained and for every minimizer $a$,
we have $\|a\|\ci{L^\infty(m)}\le 2$ and
$$
|\wtRH\wt\nu_a|^2+2(\wtRH)^*[(\wtRH\wt\nu_a)\wt\nu_a]\le 6\lambda\theta^{-1}
$$
everywhere in $S$.
\goalend

Take any minimizing sequence $a_k\in L^\infty(m)$. Note that  we can assume without loss 
of generality that $\|a_k\|\ci{L^\infty(m)} \le 2$ because otherwise 
$\Xi(a_k)>2\lambda \mu(Q)>\Xi(1)$. Passing to a subsequence, if necessary, we can also assume that 
$a_k\rightarrow a$ weakly in $L^{\infty}(m)$ (considered as $[L^1(m)]^*$).

Then $ \wtRH \wt\nu_{a_k} \rightarrow \wtRH\wt\nu_a$ uniformly on $\supp\wt\nu$, 
because the set of functions $\wt{K}^H(x-\cdot)\frac{d\wt\nu}{dm}$ ($x\in\supp\wt\nu$)
is compact in $L^1(m)$ as it is the image of the compact set $\supp\wt\nu$ under the continuous map 
$S\ni x\mapsto \wt{K}^H(x-\cdot)\frac{d\wt\nu}{dm}\in L^1(m)$. 

Thus 
$$
\int |\wtRH\wt\nu_{a_k}|^2\,d\wt\nu_{a_k} \to \int |\wtRH\wt\nu_a|^2\, d\wt\nu_a\,.
$$
Also $a\ge 0$, $\|a\|\ci{L^\infty(m)} \le \liminf_{k\to \infty} \|a_k\|\ci{L^\infty(m)}$, and
$\int\eta\,d\wt\nu_{a_k}\to \int\eta\,d\wt\nu_{a}$.

Combining these observations, we see that $a$ satisfies all restrictions of the extremal problem and
$$
\Xi(a)\le \liminf_{k\to\infty} \Xi(a_k)\,.
$$
Since $a_k$ was a minimizing sequence, we conclude that
$a$ is a minimizer of the functional $\Xi$.

Note that for every $a$ in the domain of minimization (admissible $a$), the function $\wtRH\wt\nu_a$ is continuous 
in $S$. Moreover, its maximum and modulus of continuity are controlled by $\|a\|\ci{L^\infty(m)}$ 
(although the exact constant in this control can be very large).

Let $U\subset\R^{d+1}$  be any Borel set with $\wt\nu_a(U) >0$. For $t\in (0,1)$, 
consider the function $a_t = (1-t\chi\ci U)a$. In general, it is {\em not} admissible, 
but it is still non-negative and satisfies $\|a_t\|\ci{L^\infty(m)} \le \|a\|\ci{L^\infty(m)}$.

Note that
\begin{multline*}
\int |\wtRH\wt\nu_{a_t}|^2 d\wt\nu_{a_t}
\\
= \int |\wtRH\wt\nu_{a}|^2 d\wt\nu_{a}  
- t 
\left[ 
\int_U |\wtRH\wt\nu_{a}|^2 d\wt\nu_{a}  +
 2\int\left\langle \wtRH\wt\nu_a, \wtRH(\chi\ci U\wt\nu_a)\right\rangle\, d\wt\nu_a
 \right] 
 +O(t^2) 
 \\
 = \int |\wtRH\wt\nu_{a}|^2 d\wt\nu_{a} -
 t \int_U  \left[|\wtRH\wt\nu_{a}|^2  +2 (\wtRH)^*[(\wtRH\wt\nu_a)\wt\nu_a]\,\right]\,d\wt\nu_a +O(t^2)
\end{multline*}
as $t\to 0+$.
For small $t>0$, consider $\wt a_t= \left(1-t\frac{\wt\nu_a(U)}{\theta\mu(Q)}\right)^{-1} a_t$. 
Since $a$ is admissible and $\eta\le 1$, we have
\begin{multline*}
\int \eta\, \wt a_t \,d\wt\nu= \frac{\theta\mu(Q)}{\theta\mu(Q)-t \wt\nu_a(U)} 
\left(\int \eta\, d\wt\nu_a - t \int_U\eta\, d \wt\nu_a\right)
\\
\ge  
\frac{\theta\mu(Q)}{\theta\mu(Q)-t \wt\nu_a(U)} 
\left[\theta\mu(Q) - t  \wt\nu_a(U)\right]= \theta\mu(Q)\,.
\end{multline*}
Hence, $\wt a_t$ is admissible. On the other hand, 
$$
\|\wt a_t\|\ci{L^\infty(m)} \le \left(1-t\frac{\wt\nu_a(U)}{\theta\mu(Q)}\right)^{-1} \|a\|\ci{L^\infty(m)}
$$
and
$$
\int |\wtRH\wt\nu\ci{\wt a_t}|^2\,d \wt\nu\ci{\wt a_t} = \left(1-t\frac{\wt\nu_a(U)}{\theta\mu(Q)}\right)^{-3}
\int |\wtRH\wt\nu_{a_t}|^2\,d \wt\nu_{a_t}\,.
$$
Thus,
\begin{multline*}
\Xi(\wt a_t) \le \left[ 1-t\frac{\wt\nu_a(U)}{\theta\mu(Q)}\right]^{-3}\Xi(a_t) 
\\
\le  \Xi(a)+ 
t \left[3\Xi(a) \frac{\wt\nu_a(U)}{\theta\mu(Q)} -
\int_U \left[|\wtRH\wt\nu_{a}|^2 +2 (\wtRH)^*[(\wtRH\wt\nu_a)\wt\nu_a]\right]\,d\wt\nu_a\right] + O(t^2) 
\end{multline*}
as $t\to 0+$.

Since $a$ is a minimizer, the coefficient at $t$ must be non-negative:
\begin{multline*}
\int_U \left[|\wtRH\wt\nu_{a}|^2 +2 (\wtRH)^*[(\wt{R}^H\wt\nu_a)\wt\nu_a]\right]\,d\wt\nu_a 
\\
\le 
\frac{3\Xi(a)}{\theta\mu(Q)} \wt\nu_a(U)
\le \frac{6\lambda\mu(Q)}{\theta\mu(Q)} \wt\nu_a(U) \le 6\lambda\theta^{-1} \wt\nu_a(U)\,.
\end{multline*}
Since this inequality holds for every set $U$ of positive $\wt\nu_a$ measure, we conclude that 
$$
|\wtRH\wt\nu_{a}|^2 +2 (\wtRH)^*[(\wt{R}^H\wt\nu_a)\wt\nu_a] \le 6\lambda\theta^{-1} 
$$
almost everywhere with respect to the measure $\wt\nu_a$. However, the left hand side is a 
continuous function (another use of the fact that the density of $\wt\nu$ with respect to $m$ is bounded), 
and, thereby, the last estimate extends to $\supp\wt\nu_a$ by continuity. 
Since the left hand side is subharmonic in $S\setminus \supp\wt\nu_a$, vanishes on the hyperplane $L$, 
and tends to zero at infinity, the classical maximum principle for subharmonic functions allows us 
to conclude that the last inequality holds everywhere in the half-space 
$S$.

\section{Contradiction}
\label{contradiction}

Integrate the last inequality against $|\psi|\, dm$, where  $\psi$ is the vector field 
constructed in Section \ref{fieldpsi}. 
We get
\begin{multline*}
\int |\wtRH\wt\nu_{a}|^2\cdot |\psi|\, dm + 
2\int \left[(\wtRH)^*[(\wtRH\wt\nu_a)\wt\nu_a]\right]\cdot|\psi|\,dm 
\\
\le 6\lambda\theta^{-1}\int|\psi| dm \le  C\lambda\theta^{-1}\delta^{-1} \mu(Q)\,.
\end{multline*}
Rewrite the second integral on the left as 
$$
\int \left\langle \wtRH\wt\nu_a,\wtRH(|\psi| m)\right\rangle\, d\wt\nu_a\,.
$$ 
Then, by the Cauchy inequality,
\begin{multline*}
\int  \left[(\wtRH)^*[(\wtRH\wt\nu_a)\wt\nu_a]\right]\cdot|\psi|\,dm 
\\
\le \left[\int |\wtRH\wt\nu_{a}|^2\,d\wt\nu_a\right]^{\frac 12}
\left[\int |\wtRH(|\psi| m) |^2\,d\wt\nu_a\right]^{\frac 12}
\\
\le
\Xi(a)^{\frac 12}\left[\int |\wtRH(|\psi| m) |^2\,d\wt\nu_a\right]^{\frac 12}\,.
\end{multline*}
Recall that $\|a\|\ci{L^\infty(m)}\le 2$, so we can replace $\wt\nu_a$ by $\wt\nu$ in the last integral
losing at most a factor of $2$. Taking into account that 
$$
\int |\wtRH(|\psi| m)|^2 \,d\wt\nu \le \Theta\mu(Q)\,,
$$ 
we get
$$
\left|\int  \left[(\wtRH)^*[(\wtRH\wt\nu_a)\wt\nu_a]\right]\cdot|\psi|\,dm\right|
\le C\left[\lambda\Theta\right]^{\frac 12} \mu(Q)\,.
$$
Thus, 
$$
\int |\wtRH\wt\nu_{a}|^2\cdot |\psi|\,dm \le C(\delta)\lambda^{1/2} \mu(Q)\,.
$$
Using the Cauchy inequality again, we obtain
$$
\int\langle \wtRH\wt\nu_a, \psi\rangle\,dm \le \left[
\int |\wt{R}^H\wt\nu_{a}|^2\cdot |\psi|\,dm\right]^{\frac 12} 
\left[
\int |\psi|\,dm\right]^{\frac 12}\le C(\delta)\lambda^{\frac 14}\mu(Q)\,.
$$
However, the integral on the left equals
\begin{multline*}
\int [(\wtRH)^*(\psi m)]\, d\wt\nu_a 
\\
=\int [(\RH)^*(\psi m)]\,d\wt\nu_a -
\int [T^*(\psi m)]\,d\wt\nu_a \ge 
\int\eta\, d\wt\nu_a - \sigma(\e,\alpha) \wt\nu_a(S)
\end{multline*}
(see Section \ref{fieldpsi}).
This yields
$$
\int [(\wtRH)^*(\psi m)]\,d\wt\nu_a \ge \theta\mu(Q) - \sigma(\e,\alpha)\wt\nu_a(S) 
\ge
[\theta- 2\sigma(\e,\alpha)]\mu(Q)\ge \frac{\theta}2 \mu(Q)\,,
$$
if $\e$ and $\alpha$ are chosen small enough (in this order). 
Thus, if $\lambda$ has been chosen smaller than a certain constant  depending on $\delta$ only, we get a contradiction.
\end{proof}

There, still, may be some other results one can obtain
using these and some additional (yet unknown) ideas,
more wonderful than any you can find in this paper; 
but now, when we try to get a clear view of those, 
they are gone before we can catch hold
of  them.  Despite we part with even the most
patient and the most faithful readers at this point, it 
isn't really Good-bye,
because, as it was once said at the end of another much better known tale,
{\em the Forest will always be there\dots\ and anybody  who
is Friendly with Bears can find it}.

\end{document}